\documentclass[11pt,english,oneside]{amsart}
\usepackage[T1]{fontenc}
\usepackage[latin9]{inputenc}
\usepackage{geometry}
\usepackage{amssymb}
\usepackage{color}
\usepackage{graphicx,color}
\usepackage{amsmath, amssymb, graphics}
\usepackage{graphicx}
\usepackage{placeins}

\makeatletter
\numberwithin{equation}{section} %% Comment out for sequentially-numbered
\numberwithin{figure}{section} %% Comment out for sequentially-numbered
  \@ifundefined{theoremstyle}{\usepackage{amsthm}}{}
  \theoremstyle{plain}
  \newtheorem{thm}{Theorem}[section]
  \theoremstyle{plain}
  \newtheorem{cor}[thm]{Corollary}
  \theoremstyle{plain}
  \newtheorem{prop}[thm]{Proposition}
  \theoremstyle{Remark}
  \newtheorem{rem}[thm]{Remark}
  \theoremstyle{remark}
  
  \theoremstyle{plain}
  \newtheorem{lem}[thm]{Lemma}
  \newtheorem{mydef}{Definition}

\usepackage{geometry}

\def\bfR#1{{\bf R}^#1}

\def\com#1{ \hbox{#1}}

\smallskip
\def\<{{\langle }}
\def\>{{\rangle }}

\makeatother

\usepackage{babel}

\def\bfR#1{{\bf R}^#1}

\def\com#1{ \quad\hbox{#1}\quad}

\smallskip
\def\<{{\langle }}
\def\>{{\rangle }}

\makeatother

\begin{document}

\title[cylindrical rotating  drops]{Equilibrium Shapes of Cylindrical Rotating Liquid Drops}

\author{Bennett Palmer and  Oscar M. Perdomo}

\date{\today}

\curraddr{O. Perdomo\\
Department of Mathematics\\
Central Connecticut State University\\
New Britain, CT 06050 USA\\
e-mail: perdomoosm@ccsu.edu
}
\curraddr{B. Palmer\\
Department of Mathematics\\
Idaho  State University\\
Pocatello, ID, 83209 USA\\
e-mail: palmbenn@isu.edu
}

%\email{ perdomoosm@ccsu.edu}

%\subjclass[2000]{58E12, 58E20, 53C42, 53C43}

\maketitle

\begin{abstract}

We will consider surfaces whose mean curvature at a point is a linear function of the square of the distance from that point to the vertical axis. We restrict ourselves here to surfaces which are
cylinders over a curve in a horizontal plane. We describe the moduli space of the set of solutions which includes numerous properly immersed and embedded examples. We also analyze the 
stability of these cylindrical surfaces.

\end{abstract}
\begin{figure}[ht]
\centerline{\includegraphics[width=8cm,height=4cm]{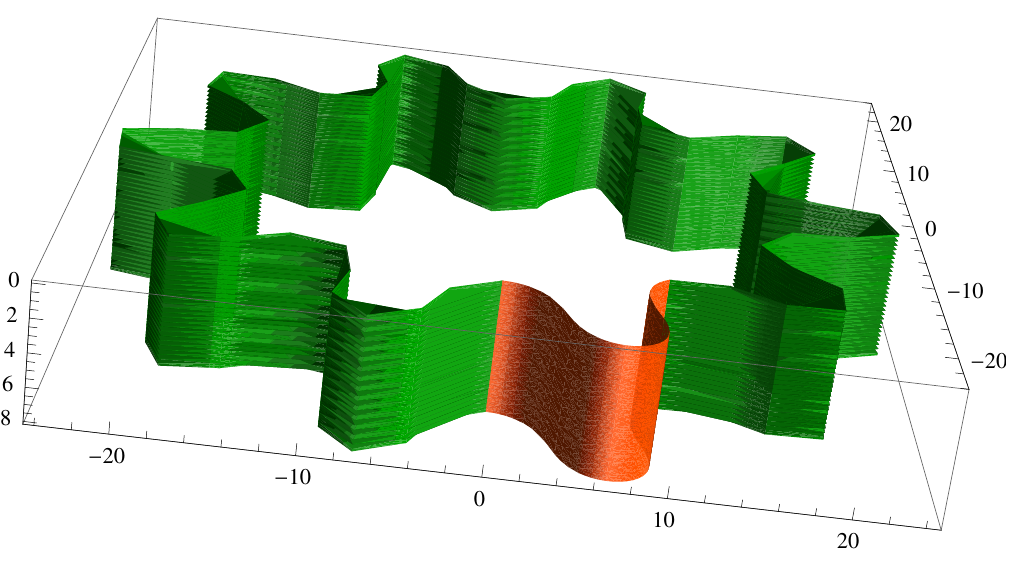}}
\end{figure}
\section{Inroduction}

In this paper we will study the equilibrium shapes of rotating liquid drops assuming which can be represented as cylinders $\alpha \times {\bf R}$ over a curve $\alpha$ in a plane orthogonal to the rotation axis. The current investigation is part off a larger project, to be continued elsewhere \cite{PP}, which will study the equilibrium drops which are invariant under a helicoidal motion. The special class of cylindrical solutions turns out to be surprisingly rich.

The equilibrium shapes of rotating drops have been widely studied , \cite{BS},\cite{Ch}, \cite{S}, \cite{MH}, and many others. These studies have primarily concentrated on the case where the surface of the drop is closed.
Here we regard the drop as the solution of either a fixed or free boundary value problem. In either case, the Euler Lagrange equation tells us that in the interior, there holds
\begin{equation}
\label{EL}2H=-\frac{a}{2}R^2+\Lambda_0\:,\end{equation}
where $a$ and $\Lambda_0$ are constants. The approach we use here and in the work to follow is based on a particular representation of the generating curve $\alpha$, called the
{\it TreadmillSled representation} which was developed by the second author \cite{P1} to study the helicoidal surfaces with constant mean curvature. This is used to obtain a complete description of the moduli space of cylindrical solutions of (\ref{EL}). Among these solutions, we find many embedded examples.

In the second part of the paper, we study the stability of the equilibria for both the free and fixed boundary problems. For the free boundary problem with the supporting surface(s) consisting of two horizontal planes, we show that, among cylindrical surfaces,  the only embedded stable critical points are particular examples of round cylinders.

\section{Preliminaries}

We will consider the equilibrium shape of a liquid drop rotating with a constant angular velocity $\Omega$ about a vertical axis. The surface of the drop, which we denote by $\Sigma$, is represented as a smooth surface. 
The bulk of the drop is assumed to be occupied by an imcompressible liquid of a constant mass density $\rho_1$ while the drop is surrounded by a fluid of constant mass density $\rho_2$.
Since the drop is liquid, its free surface energy is proportional to its surface area ${\mathcal A}$ and we take the constant of proportionality to be one. 
 The rotation contributes a second  energy term of the form $-\Omega^2 \Delta {\mathcal I}$, where $\Delta {\mathcal I}$ is  difference of moments of inertia about the vertical axis,
 $$\Delta {\mathcal I}:=(\rho_1-\rho_2) \int_UR^2\:dv\:. $$
 
  This term represents
 twice the rotational kinetic energy. 
 
 The total energy is thus of the form
 \begin{equation}
 \label{E}
 {\mathcal E}:={\mathcal A}-\frac{\Omega^2}{2}{\Delta \mathcal I}+\Lambda_0{\mathcal V}\:,
 \end{equation}
 where ${\mathcal V}$ denotes the volume of the drop and $\Lambda_0$ is a Lagrange multiplier. Let $\Delta \rho:=\rho_1-\rho_2$, then by introducing a constant  $a :=(\Delta \rho) \Omega^2$, we can write the functional in the form
  \begin{equation}
   \label{E1}
{\mathcal E}_{a, \Lambda_0}= {\mathcal A}-\frac{a}{2} \int_U R^2\:dV+\Lambda_0 {\mathcal V} \:,%\nonumber
 \end{equation}
where $U$ is the three dimensional region occupied by the bulk of the drop and $R:=\sqrt{x^2+y^2}$. When the drop is not embedded, we regard ${\mathcal V}$ as the signed algebraic volume.
The first variation formula  yields the Euler-Lagrange equation  (\ref{EL}); see, for example \cite{RL}. We will call surfaces which satisfy (\ref{EL}), {\it equilibrium surfaces}. 

In this paper, we restrict out attention solutions of (\ref{EL}) which are cylinders over a planar curve $\alpha$ which we assume to be parameterized by arc length.
We therefore write the cylinder $\Sigma$  as the image of a parameterization 
$$\phi(s,t)=(x(s),y(s),t)\:.$$
We will refer to the curve $\alpha$ as the {\it profile curve} of the surface.  We can define an angle $\theta$ by the equations
$$x^\prime(s)=\cos(\theta(s))\com{and} y^\prime(s)=\sin(\theta(s))\:.$$
If $\kappa(s)$ denotes the curve $\alpha$, then  $x''(s)=\kappa(s) y'(s)$ and $y''(s)=-\kappa(s)x'(s)$, the angle $\theta$ satisfies
$$\theta'(s)=-\kappa(s)\:.$$
Since $\Sigma$ is a cylinder, we have that the mean curvature $H$ satisfies
$$H(s,t) = \frac{\kappa(s)}{2}=-\frac{ \theta^\prime(s)}{2 } $$

We define $\xi_1(s)$ and $\xi_2(s)$ as

\begin{eqnarray}\label{TS coordinates}
 \xi_1(s)=x(s)\cos(\theta(s))+y(s)\sin(\theta(s)) \com{and}  \xi_2(s)=x(s)\sin(\theta(s))-y(s)\cos(\theta(s))\:.
\end{eqnarray}

The surface will satisfy the equation

%%Compare with equation (1.1). They are not the same.

$$2 H=\Lambda_0-a \frac{R^2}{2}$$

if and only if

\begin{eqnarray}\label{dtheta}
\theta^\prime(s)= -  \Lambda_0 + \frac{1}{2} a (\xi_1^2+\xi_2^2)\:,
\end{eqnarray}
holds,
since $\xi_1^2+\xi_2^2=x^2+y^2=R^2$.

We next introduce the function 
$$G(\xi_1, \xi_2):=2\xi_2+\Lambda_0(\xi_1^2+\xi_2^2)-\frac{a}{4}(\xi_1^2+\xi_2^2)^2\:.$$

From the definition of $\xi_1$,  $\xi_2$ and $\theta$, we  obtain 
\begin{equation}
\label{dz}
 \xi_1^\prime=1+\kappa  \xi_2\:,\quad  \xi_2^\prime=-\kappa \xi_1\:.
 \end{equation}
  Using Equation (\ref{dtheta}) we conclude that $\xi_1$ and $\xi_2$ must satisfy

\begin{eqnarray} \label{the ode}
\xi_1^\prime &=&  1+ \xi_2\, \Lambda_0 - \frac{1}{2} \, \xi_2 \,  a\,  (\xi_1^2+\xi_2^2)=\frac{1}{2}\frac{\partial G}{\partial \xi_2} \:,   \\
\xi_2^\prime &=&  -\xi_1\,  \Lambda_0  +  \frac{1}{2}  \, a\, \xi_1\,  (\xi_1^2+\xi_2^2))=-\frac{1}{2}\frac{\partial G}{\partial \xi_1} \nonumber\:.
\end{eqnarray}
Note that (\ref{the ode}) are just Hamilton's equation for the Hamiltonian $G/2$.
We have shown:
\begin{prop} \label{G must be const} The curve $\alpha$ is the generating curve of a cylindrical surface satisfying $2H=\Lambda_0-aR^2/2$ if and only if the curve $\beta(s):=(\xi_1(s), \xi_2(s))$
satisfies $G(\xi_1, \xi_2)\equiv$ constant.
\end{prop}

The image of the generating curve $\alpha$ under the map $(\xi_1,\xi_2):\alpha \rightarrow \{G=C\}$ will be called the {\it TreadmillSled} of $\alpha$.
In most cases, e.g. when $\alpha$ is a closed curve, the map $(\xi_1,\xi_2)$ will be surjective but we will give examples where this is not the case.
 The name arises from the fact that
if for any $s$ fixed, if we consider the rigid motion $T_s$ from $\bfR{2}$ to $\bfR{2}$ that sends the point $\alpha(s)$ to $(0,0)$ and the ray $\{\alpha(s)+t\alpha^\prime(s): t>0\}$ to the ray $\{(t,0):t>0\}$, then this transformation $T_s$ will send the origin to $(-\xi_1(s),\xi_2(s)))$. Using this fact we can view the curve $\beta$ as generated by the trace of the origin of the curve $\alpha$ when $\alpha$ moves on a treadmill placed at the origin and oriented along the $x$-axis. For more details look at the papers \cite{P3} and \cite{P1}. A picture showing this relation between the two curves is shown in Figure \ref{TSofcyl}

% Therefore $-\beta(s)$ can be viewed as the trace of the origin when we make every point  $\alpha(s)$  pass by the origin in such a way %that the velocity vector of $\alpha(s)$ points in the positive direction of the $x-axis$ when $\alpha(s)$  is passing through the origin.

% the generating curve $\alpha$ generates the level set $G(\xi_1, \xi_2)= c$ as the trace of a point rigidly attached to the $\alpha$ curve as %it rolls along a moving line.

\begin{rem}\label{xi1positive} The level sets of $G$ are symmetric with the $\xi_2$-axis, therefore in order to understand the level set of $G$, it is enough to understand those points in the level set with $\xi_1\ge 0$\:.

\end{rem}
\section{Explicit construction of cylindrical drops.}
In order to study the level sets of the function $G$ we will replace the variables  $\xi_1$ and $\xi_2$ by the  variables $r$ and $\xi_2$ where

$$ r=\xi_1^2+\xi_2^2 \:.$$

Under this change, we get that the equation $G=C$ reduces to

$$ 2 \xi_2  + \Lambda_0 r -\frac{a}{4} \, r^2= C \:.$$

Therefore,

$$ \xi_2=\frac{1}{8} (4 C + r (-4 \Lambda_0 + a r))\:. $$

 By remark (\ref{xi1positive}), it is enough to consider those points with $\xi_1\ge0$. Since $\xi_1=\sqrt{r-\xi_2^2}$, we obtain 

 $$\xi_1=\sqrt{r-\frac{1}{64} (4 c + r (-4 b + a r))^2  }=:\frac{1}{8}\sqrt{q(r,a,c, \Lambda_0)}\:,$$
where

\begin{eqnarray}
q=q(r,a,C, \Lambda_0)=-16 C^2 + 64 r + 32 C r - 16 \Lambda_0^2  r^2 - 8 a C r^2 + 8 a \Lambda_0 r^3 - a^2 r^4\:.
\end{eqnarray}

\begin{rem}
Since $q$ is a polynomial in $r$ of degree 4  with negative leading coefficient when $a\ne0$ and $q$ is a polynomial of degree two when $a=0$, we get that the values of $r$ for which $q$ is positive are bounded. Since $r=\xi_1^2+\xi_2^2=x^2+y^2$,  we conclude that all the profile curves of the cylindrical solutions to the rotating drop equation are bounded.
\end{rem}

\begin{mydef}\label{def of rho} Let $r_1$ and $r_2$ be two values that satisfies $q(r_1)=q(r_2)=0$ and $q(r)>0$ for all $r\in [r_1,r_2]$. We define  $\rho:[r_1,r_2]\longrightarrow \bfR{2}$ by

$$ \rho(r)= \big(\, \frac{1}{8}\sqrt{q(r,a,C, \Lambda_0)} \, ,\, \frac{1}{8} (4 C + r (-4 \Lambda_0 + a r))  \, \big) \:.$$

\end{mydef}

\begin{rem} \label{prop of rho} As pointed out previously, all the level set of the function $G$ are bounded. Moreover, they are closed curves. We have that  the map $\rho$ parametrizes half of a connected component of the level set $G=C$.
\end{rem}

We will prove that, except for a few exceptions,  the connected components of the level set $G(\xi_1,\xi_2)=C$ are regular closed curves. When this happens, we define a {\it fundamental piece} of the profile curve of the cylindrical rotating drop as a connected part of the profile curve such that the parametrized curve $(\xi_1(s),\xi_2(s))$ given by equations (\ref{TS coordinates}) parametrizes exactly one of these closed regular curves that form the  level set of $G=C$.

\begin{rem}
From the definition of TreadmillSled given in \cite{P1}, we obtain that the profile curves of  the  equilibrium cylindrical rotational drops are characterized by the property that their TreadmillSleds are contained in the level sets of $G$. In other words, using the notation of \cite{P1}, we have that $TS(\alpha)=\beta$ where $\beta$ is a parametrization of a connected component of the level set of $G=C$ and $\alpha$ is the profile curve of the cylindrical rotating drop. % In \cite{PP}, this result is obtained using Noether's Theorem.  
We will see that for a few exceptional examples, the profile curve is a bounded complete curve that winds infinity many times around the origin and whose  limit  cycle is a circle. For the non-exceptional examples, we can define a fundamental piece of the profile curve, and we can define an initial point $p_1$ and a final point $p_2$ of this fundamental piece. Furthermore, we have that the whole profile curve is the union of rotations of fundamental pieces and we also have that if  $R_1=\hbox{min}\{ |m| : m\in TS(\alpha)\}$ and $R_2=\hbox{max}\{ |m| : m\in TS(\alpha)\}$ and if $\Delta \tilde{\theta}$ is the variation of the angle between $\overrightarrow{0p_1}$ and $\overrightarrow{0p_2}$ then, the profile curve is properly immersed if $\Delta\tilde{\theta}/\pi$ is a rational number, otherwise the profile curve is dense in the set
 $\{(x,y)\in \bfR{2}: R_1\le |(x,y)|\le R_2\}$

\end{rem}

\begin{figure}[ht]
\centerline{\includegraphics[width=5cm,height=5cm]{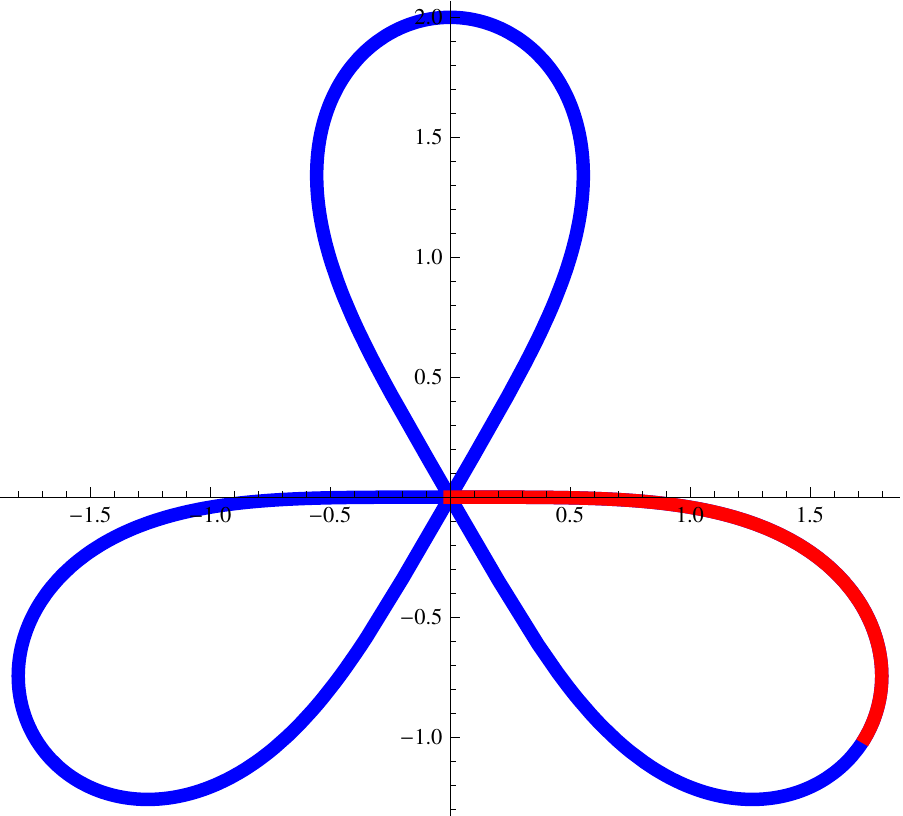}\hskip.3cm \includegraphics[width=5cm,height=5cm]{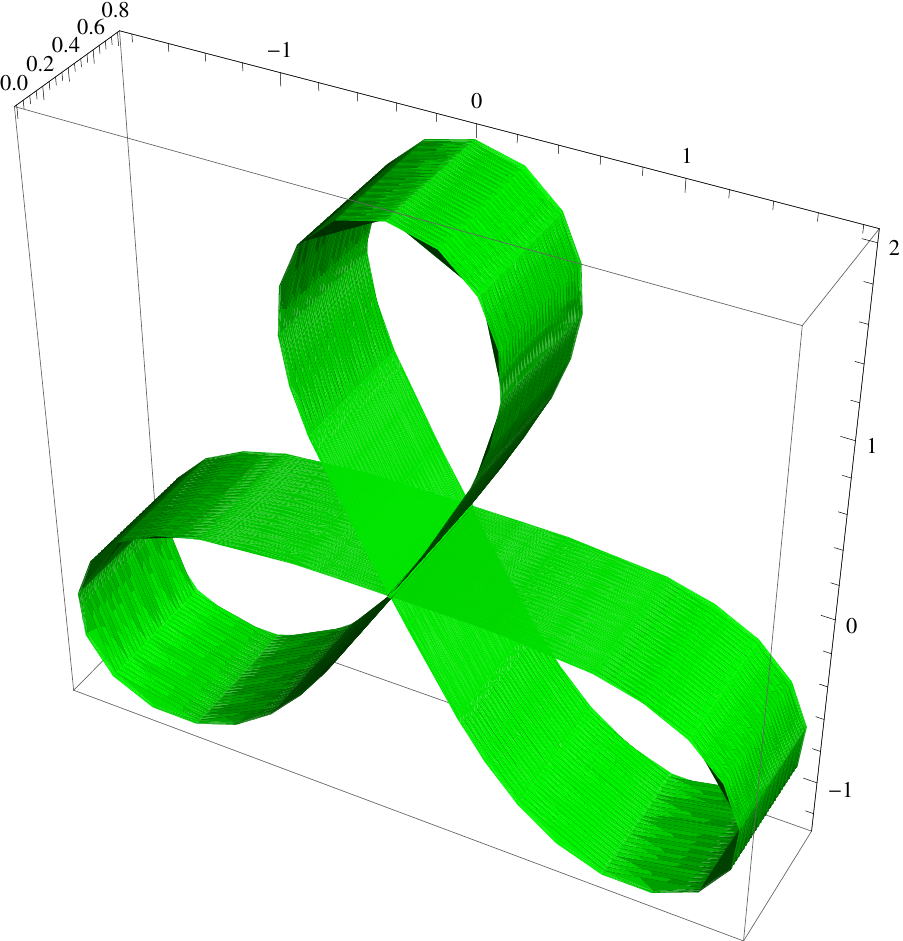}}
\label{surface c e 0}
\caption{Profile curve of  a cylindrical drop with $\Lambda_0=0$, $a=1$, $c=0$. The first picture shows the whole profile curve and the second picture shows part of the cylinder.}
\end{figure}

We now compute the variation $\Delta \tilde{\theta}$ in terms of the parameter $r$. We assume that $\alpha(s)$ is the profile curve of a cylindrical rotational drop. Recall that we are assuming that $s$ is the arc-length parameter for the curve $\alpha$. If $\beta(s)=(\xi_1(s),\xi_2(s))$, then we have that for some function $s=\sigma(r)$, $\beta(\sigma(r))=\rho(r)$. By the chain rule we have that

\begin{eqnarray}\label{ds}
|\frac{ds}{dr}|=|\frac{d\sigma}{dr}|=\frac{|\rho^\prime(r)|}{|\beta^\prime(s)|}=\sqrt{\frac{4|\rho^\prime(r)|^2}{(G_{\xi_1}(\beta(s))^2+(G_{\xi_1}(\beta(s))^2}}
=\frac{4}{\sqrt{q(r,a,C, \Lambda_0)}}
\end{eqnarray}
and we also have that $\tilde{\theta}$ denotes the angle of the profile curve when it is written in polar coordinates, then $\tilde{\theta}^\prime(s)= \xi_2(s)/r$ and

\begin{eqnarray}\label{dttheta}
\frac{d\tilde{\theta}}{dr} = \frac{d\tilde{\theta}}{ds}\, \frac{ds}{dr}=  \frac{4C+a r^2-4 \Lambda_0 r}{2 r\sqrt{ q(r,a,C, \Lambda_0)}}\:.
\end{eqnarray}

Since  the map $\rho:[r_1,r_2]\longrightarrow \bfR{{2}}$ parametrizes half of the TreadmillSled of the the fundamental piece of the profile curve, we obtain that the the following expression for $\Delta {\tilde \theta}$,

\begin{eqnarray}\label{change}
\Delta {\tilde \theta} = \Delta{\tilde \theta}(C,a,r_1,r_2)  =\, \int_{r_1}^{r_2} \frac{(4C+ a r^2-4 \Lambda_0 r) }{r \sqrt{q(r,a,C, \Lambda_0)}}\, dr\:.
\end{eqnarray}

Figure \ref{Deltatildath} illustrates the meaning of the function $\Delta {\tilde \theta}$.

\begin{figure}[ht]
\centerline{\includegraphics[width=6cm,height=6cm]{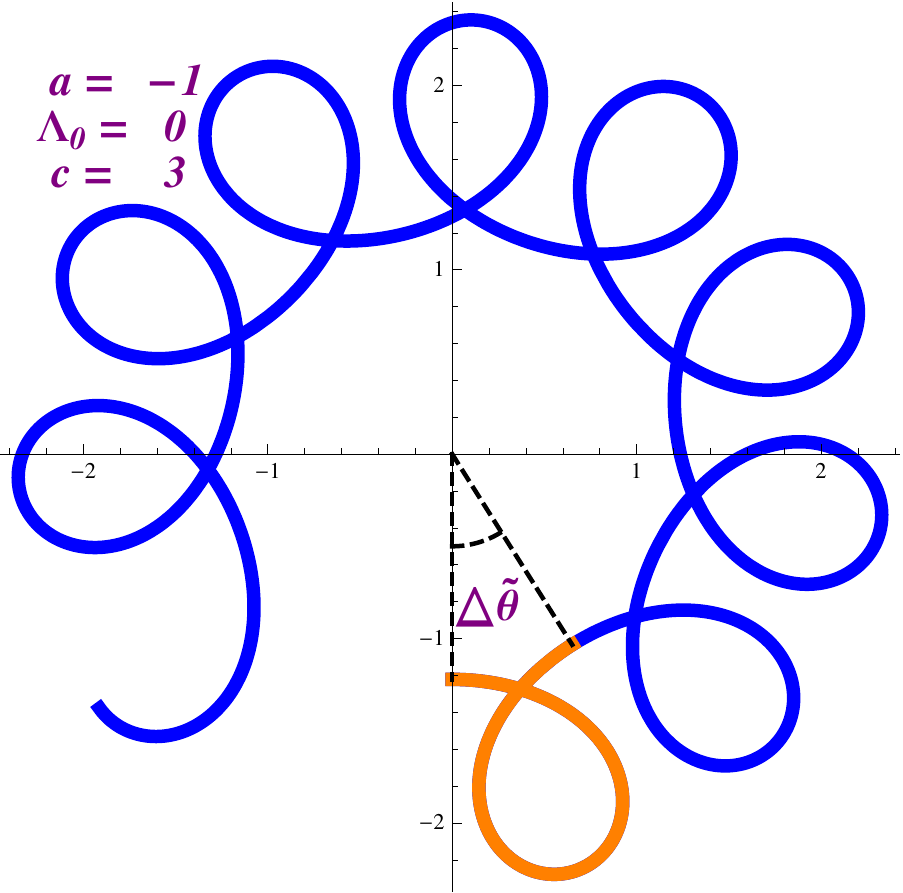} \includegraphics[width=6cm,height=6cm]{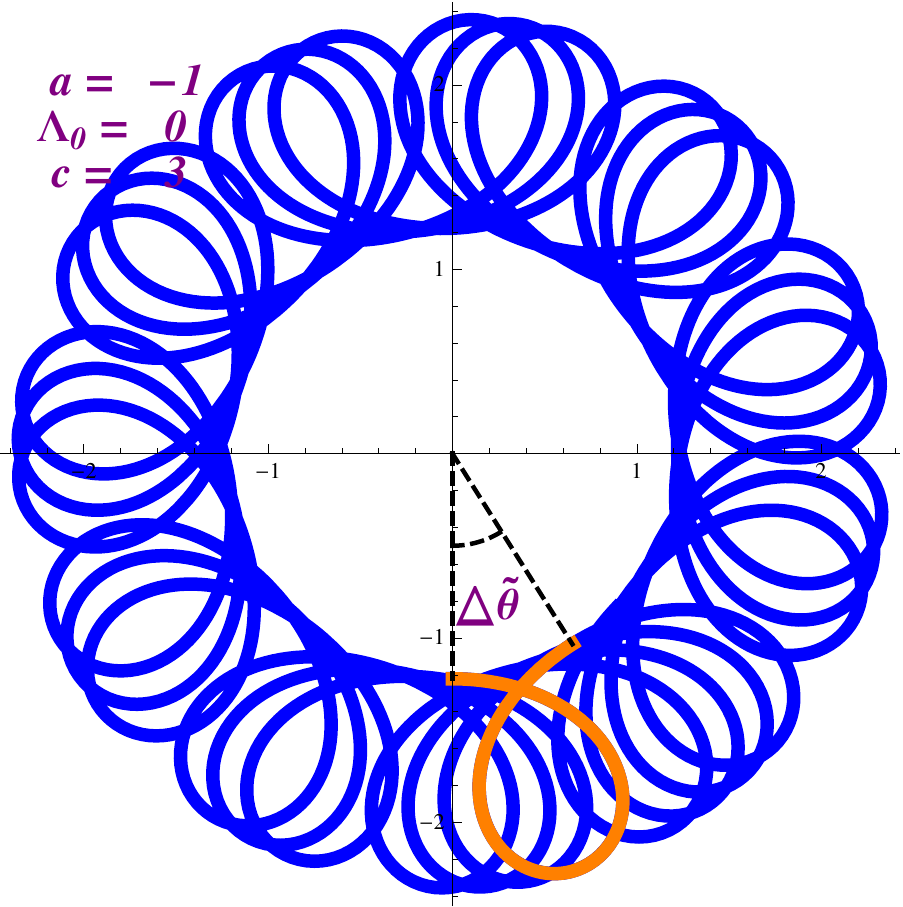}}
\caption{The profile curve is the union of rotations of one fundamental piece. The function $\Delta\tilde{\theta}$ measures the angle between the starting and final point of the fundamental piece. The graph on the right shows the part of the profile curve consisting of 70 fundamental pieces. Most likely this fundamental piece will never be a closed curve and therefore this profile curve will be dense in the region bounded by two circles.}
\label{Deltatildath}
\end{figure}

\begin{rem}
Under rescaling of an equilibrium drop $\Sigma \mapsto \lambda \Sigma$, $\lambda \in {\bf R}^*$, we have $H\mapsto \lambda^{-1}H$ and so $\Lambda_0\mapsto \lambda^{-1}\Lambda_0$ and $a\mapsto \lambda^{-4}a$. In particular, the set of equilibria is invariant under rescaling.
Changes in orientation induced by, for example, changes in the orientation  of the profile curve, will reverse the signs of  $\Lambda_0$, $a$ and $H$.  With these two observations in mind, we have that in order to consider all the cylindrical rotational drops, up to reparametrizations, rigid motions and dilations, it is enough to consider two cases: Case I, $\Lambda_0=0$ and $a=-1$ and Case II, $\Lambda_0=1$ and $a$ any real number.

\end{rem}

\subsection{Case I: $\Lambda_0=0$ and $a=-1$}

In this case, we let $q(r, C):=q(r,-1,C,1)$, i.e.

$$q=q(r,C)=-16 C^2 + 64 r + 8 C r^2 - r^4 $$

Recall that we are interested in finding two positive consecutive roots of the polynomial $q$. Notice that when $C$ is a  large negative number, the polynomial $q$ has no real roots and if $C$ is a  large positive number then the polynomial $q$ has more than one real root. In every case $q(0,C)=-16 C^2\le0$ and the limit when $r\to \inf$  of $q(r)$ is $-\infty$. The following values provide the possible values of $C$ for which the polynomial $q(r ,C)$ has two positive roots.

\begin{lem}\label{lem sect 3.1}
For any $C>C_0=-\frac{3}{2^{\frac{2}{3}}}$, the polynomial $q(r,C)$ has exactly two positive real roots. 
When $C=C_0$, $\sqrt[3]{4}$ is the only real root of $q(r,C)$  and when $C<C_0$, $q(r,C)$ has not real roots.  
\end{lem}

\begin{proof} 

We have that $q^\prime(r,C)=64+16 Cr-4r^3$. A direct computation shows that the only real roots of the system

$$q(r,C)=0\com{and} q^\prime(r,C)=0 $$

are  $C=-\frac{3}{2^{\frac{2}{3}}}$ and $r=2^\frac{2}{3}$. This  also follows from the fact that a Groebner basis of the polynomials $\{q,q^\prime\}$ is the set $\{27+4 C^3,-4 C^2+9 r\}$.  By continuity,  we conclude that the lemma holds. Notice that if for some value of $C$ the polynomial $q(r,C)$ has more than 2 positive roots, then there should exist another solution of the equations $\{q(r,C)=0,\, q^\prime(r,C)=0\} $.
\end{proof}

Now we will compute the limit of $\Delta {\tilde \theta}$ when $C$ goes to $C_0$.
We will use the following lemma from \cite{P2}

\begin{lem}\label{lemma 1} Let $f(c,r)$ and $g(c,r)$ be smooth functions such that $g(c_0,r_0)= \frac{\partial g}{\partial r} (c_0,r_0)=0 $ and
$\frac{\partial^2 g}{\partial^2 r} (c_0,r_0)= -2 A$ where $A>0$.  If $\{c_n\}$, $\{u_n\}$ and $\{v_n\}$ are sequence such that   $c_n$ converges to $c_0$, $u_n$ and $v_n$ converges to $r_0$ with $u_n<r_0<v_n$, and $g(u_n) =g(v_n)=0$  and  $g(r)>0$ for all $r\in(u_n,v_n)$, then

$$\int_{u_n}^{v_n}\frac{f(c_n,r)\, dr}{\sqrt{g(c_n,r)}} \longrightarrow  \, f(c_0,r_0)\, \frac{\pi}{\sqrt{A}} \com{as}  n\longrightarrow \infty\:. $$

\end{lem}

Notice that cylindrical rotating drops are defined when $C$ takes value from  $C_0=-\frac{3}{\sqrt[3]{4}}$ to $\infty$. When $C=C_0$ the only root of the polynomial $q$ is $r_0=\sqrt[3]{4}$. If we apply Lemma \ref{lemma 1} with $f(r,c)=-\frac{(4C-r^2)}{r }$ and $g(r,c)=q(r,C)$, we obtained that

\begin{eqnarray}\label{lowerbound}
\lim_{C\to C_0^+} \Delta {\tilde \theta}=- \frac{2 \pi}{\sqrt{3}}
\end{eqnarray}

Figure \ref{functiondth} shows the graph of the function $\Delta {\tilde \theta}$.  For every value $C$ between $-\frac{3}{\sqrt[3]{4}}$ and  $8$, the function gives the variation the angle  between $\overrightarrow{0p_1}$ and $\overrightarrow{0p_2}$ where $p_1$ and $p_2$ are the initial and final points of a fundamental piece

\begin{figure}[ht]
\centerline{\includegraphics[width=12cm,height=6cm]{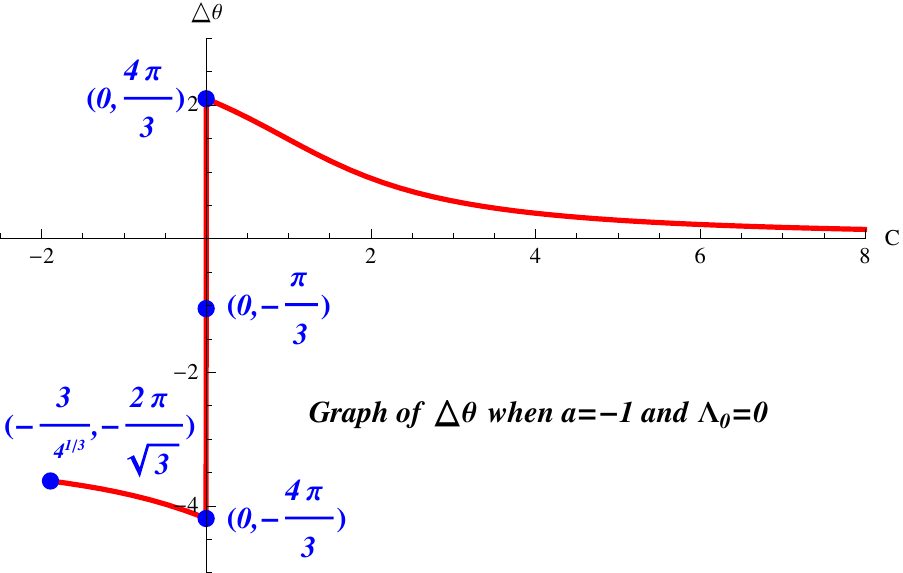}}
\caption{For $a=-1$ and $\Lambda_0=0$ the domain of the function $\Delta {\tilde \theta}$ is $(-\frac{3}{\sqrt[3]{4}},\infty)$, here we show the graph for values of $C$ between $-\frac{3}{\sqrt[3]{4}}$ and $8$.}
\label{functiondth}
\end{figure}

\begin{rem} \label{embedded}

Recall that whenever  $\Delta  {\tilde \theta}=n \pi/m$ holds for some pair of  integers $m$ and $n$,  the entire fundamental piece is properly immerse and it is invariant under the group $Z_m$. Figure \ref{four surfaces} shows some special values of $C$, that correspond to points in the function $\Delta  {\tilde \theta}$ associated with properly immersed cylindrical rotating drops.

\end{rem}

\begin{figure}[ht]
\centerline{\includegraphics[width=12cm,height=6cm]{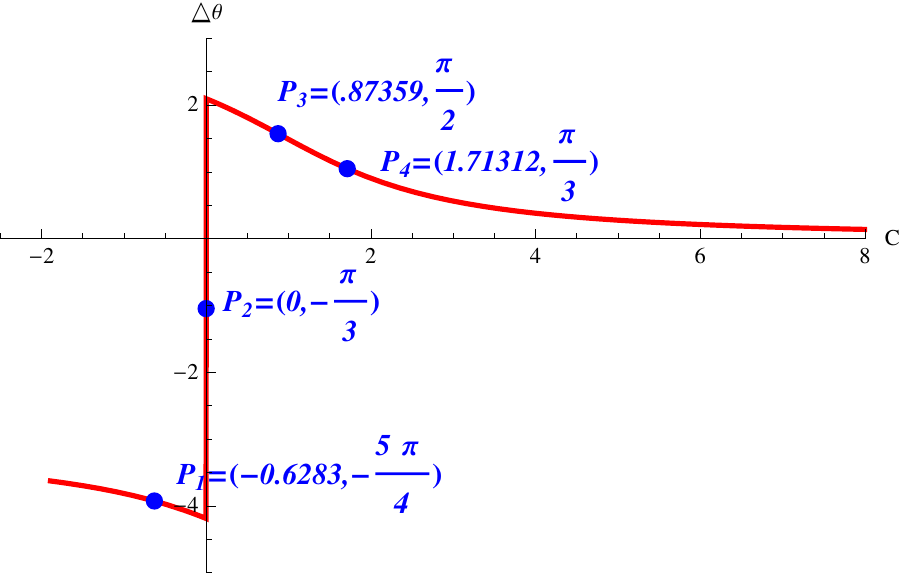}}
\centerline{\includegraphics[width=4cm,height=4cm]{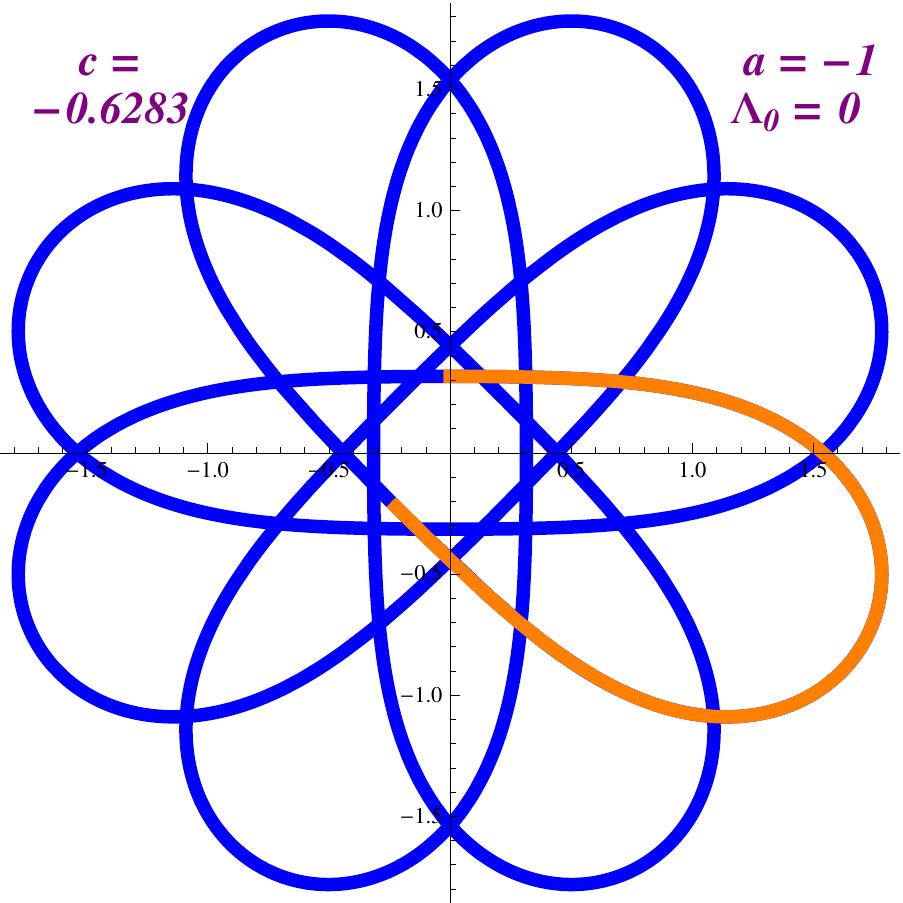}\includegraphics[width=4cm,height=4cm]{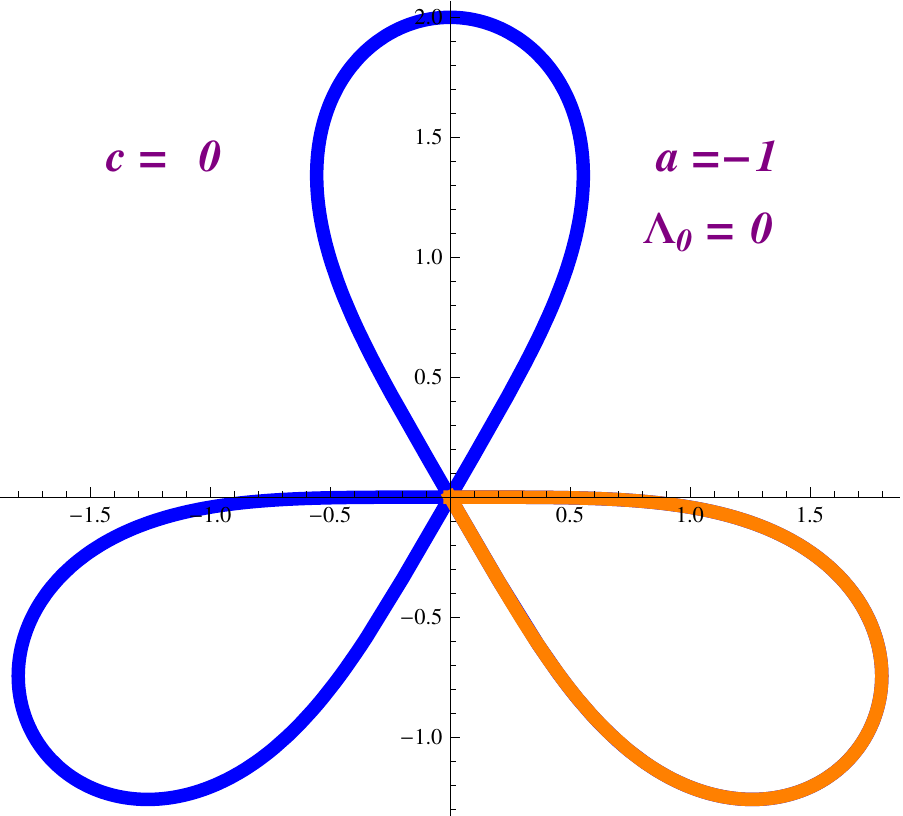}\includegraphics[width=4cm,height=4cm]{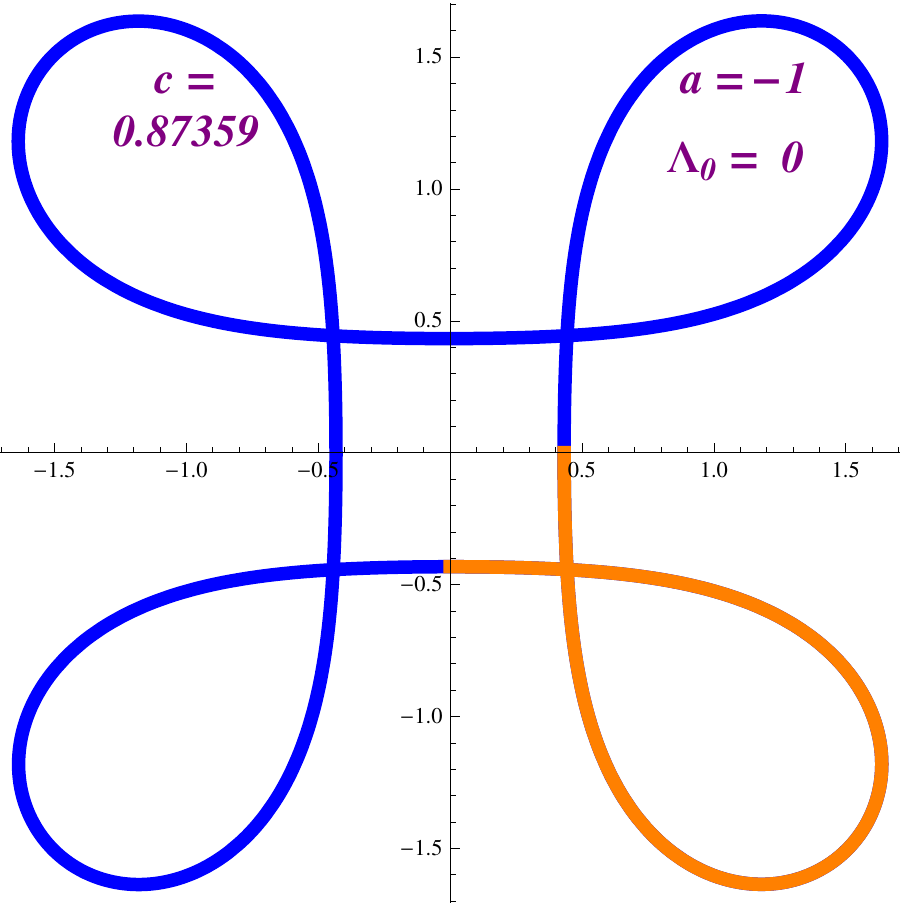}\includegraphics[width=4cm,height=4cm]{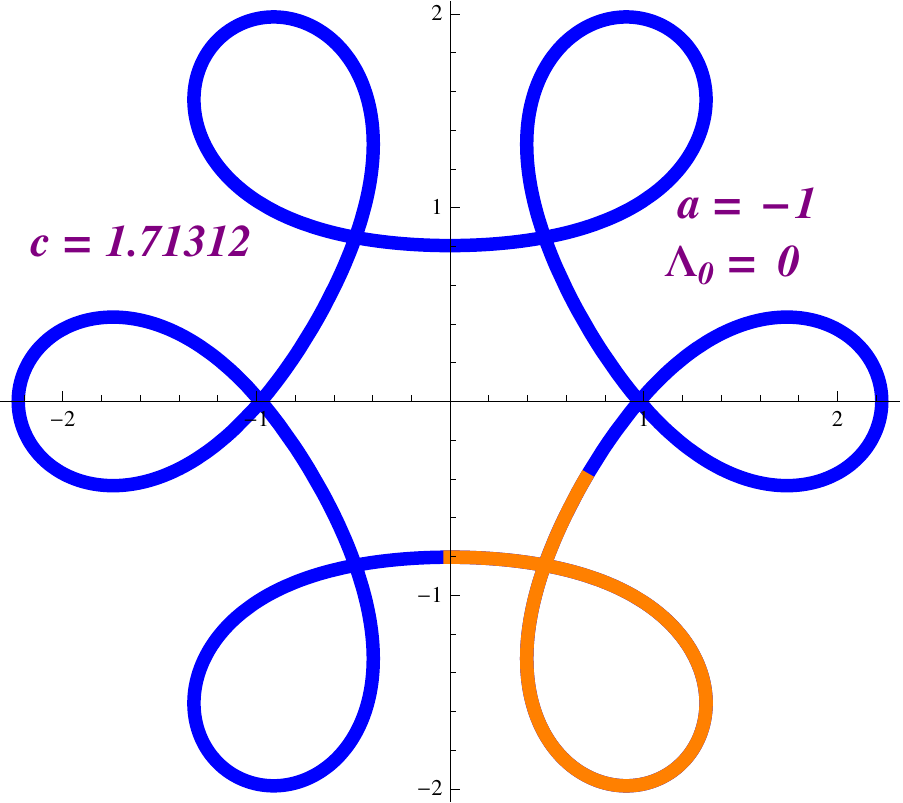}}
\caption{Points on the graph of $\Delta {\tilde \theta}$ such that the second entry is a rational multiple of $\pi$ determine properly immersed rotational cylindrical drops. Here we show the profile curve of four of these surfaces.}
\label{four surfaces}
\end{figure}

\begin{thm}
Up to dilations and rigid motions, the moduli space for all cylindrical rotating drops with $\Lambda_0=0$ is the semiline

$$\{C: C\ge C_0=-\frac{3}{\sqrt[3]{4}} \}$$

Moreover, we have: {\bf (i)}  the surface associated with $C=C_0$ is a round cylinder of radius $\sqrt[3]{2}$, and {\bf (ii)} Only  a dense, countable set of values of $C$ represent properly immersed surfaces. All other values of $C$ represent surfaces that are dense in the region bounded by two cylinders.

\end{thm}

\begin{proof}
By Proposition \ref{G must be const}, any cylindrical rotating drop corresponds to a level set $G=C$. Considering the parametrization of the level set $G=C$ given by the map $\rho$ in Definition \ref{def of rho}, we have, using Remark \ref{prop of rho}, that a cylindrical drop exists when the polynomial $q$ is nonnegative for some positive values of $r$.  Lemma \ref{lem sect 3.1} guarantees  the existence of cylindrical rotating drops when $C>C_0$. When $C=C_0$, the polynomial $q$ has $r_0=\sqrt[3]{4}$ as the only root. A direct computation shows that in this case the level set reduces to a point in the $\xi_2$-axis and that the round cylinder with radius $\sqrt[3]{2}$ is a cylindrical rotating drop. Moreover, since there are not roots of the polynomial $q$ when $C<C_0$, then $G=C$ is the empty set and therefore there are not cylindrical rotations drops when $C<C_0$. The properly immersed property follows from the fact that the profile curve is the union of rotation of fundamental pieces as Figure  \ref{Deltatildath} illustrates. A detailed argument for this part of the theorems can be found in (\cite{P1} page 475 (The proof of Theorem 7.2)) where a similar proof is made for helicoidal constant mean curvatures surfaces in $\bfR{3}$.

\end{proof}

\subsection{Case II: $\Lambda_0=1$} First we notice that the case $a=0$ correspond to helicoidal surface with constant mean curvature. These surfaces were studied using similar techniques in \cite{P1} and for this reason we will assume here that $a\ne0$.
Before analyzing the polynomial $q$, we study those solution whose profile curve are circles.

\begin{lem}
For any $a\ne0$, the polynomial $Q= a R^3-2 R+2$ has exactly one positive root when $a<0$, it has three roots (two positive and one negative) when $a$ is between $0$ and $8/27$ and it has exactly one negative root when $a>8/27$. When $a=8/27$, the only roots are $R=1.5$ with multiplicity $2$ and $R=-3$ with multiplicity $1$. 

\end{lem}

\begin{proof}
We want to analyze the set $a R^3-2 R+2=0$. In order to do this, we solve for $a$ in the previous equation and define the function 
$h(R)=\frac{2(R-1)}{R^3}$, a direct computation shows that the derivative of $h(a)$ is given by $h^\prime(a) = \frac{6-4R}{R^4}$. A very standard verification tell us the the function $h$ has a vertical asymptote at $x=0$, a vertical asymptote at $y=0$ it is increasing on $(-\infty,0)\cup(0,3/2)$ and decreasing on $(3/2,\infty)$. Figure  \ref{rootsQ} shows the graph of $h(R)$. Since the roots of the polynomial $Q$ are the can be define as the local inverse of the function $h$, then the result follows.

\end{proof} 

\begin{figure}[ht]
\centerline{\includegraphics[width=8cm,height=5cm]{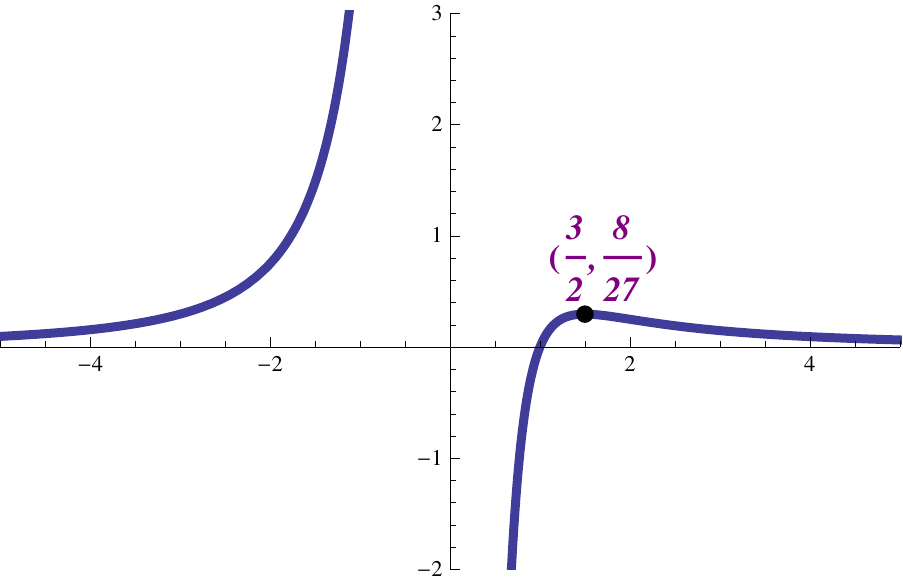}}
\caption{The graph of the function $h(R)=\frac{2(R-1)}{R^3}$ help us to understand the roots of the polynomial $Q(R)=a R^3-2 R+2$\:.}
\label{rootsQ}
\end{figure}

\begin{prop}\label{round} 
We have that a circle of radius $R>0$ and center at the origin is the profile curve of a cylindrical rotating drop with $\Lambda_0=1$  and $a\ne0$ if and only if $R$ is the absolute value of a root of the polynomial $Q=a R^3-2 R+2$

\end{prop}
\begin{proof}

Depending of the orientation of the parametrization that we use, the mean curvature of the round cylinder with radius $R$ is either   $-\frac{1}{2R}$ or $\frac{1}{2R}$. Therefore the equation $2 H=1-\frac{a}{2}\, R^2$,

$$ \frac{1}{R}=1-\frac{a}{2}\, R^2 \com{or} -\frac{1}{R}=1-\frac{a}{2}\, R^2\:.$$

That is, $R$ must be a root of the polynomial $Q=a R^3-2 R+2$ or $\tilde{Q}=a R^3-2 R-2$.
A direct computation shows that $x$ is root of $Q$ if and only if $-x$ is root of $\tilde{Q}$. Therefore every positive root of $\tilde{Q}$ is the absolute value of the root of $Q$. This completes the proof.

\end{proof}

\begin{mydef}\label{def of ri}
Let $h(R)=\frac{2(R-1)}{R^3}$, see Figure \ref{rootsQ}, and let us define $R_1:(-\infty,0)\cup (0,\infty)\to {\bf R}$, $R_2:(0,8/27)\to {\bf R}$ and $R_3:(0,8/27)\to {\bf R}$

$$\begin{array}{cccc}
R_1(a)=R  \com{such that}&  h(R)=a & \com{with} &R<1\:,\\
R_2(a)=R \com{such that}& h(R)=a  &\com{with}  &1<R<3/2\:,\\
R_3(a)=R \com{such that} & h(R)=a  &\com{with} & 3/2<R<\infty \:.
\end{array}$$

We will also define the functions $r_1:(-\infty,0)\cup (0,\infty)\to {\bf R}$, $r_2:(0,8/27)\to {\bf R}$ and $r_3:(0,8/27)\to {\bf R}$ by

$$ r_1(a)=R_1^2(a),\quad  r_2(a)=R_2^2(a),\quad  r_3(a)=R_3^2(a)\:. $$

Figure \ref{graphs of r} shows the graphs of the functions $r_1$, $r_2$ and $r_3$.

\end{mydef}

\begin{figure}[ht]
\centerline{\includegraphics[width=4cm,height=3cm]{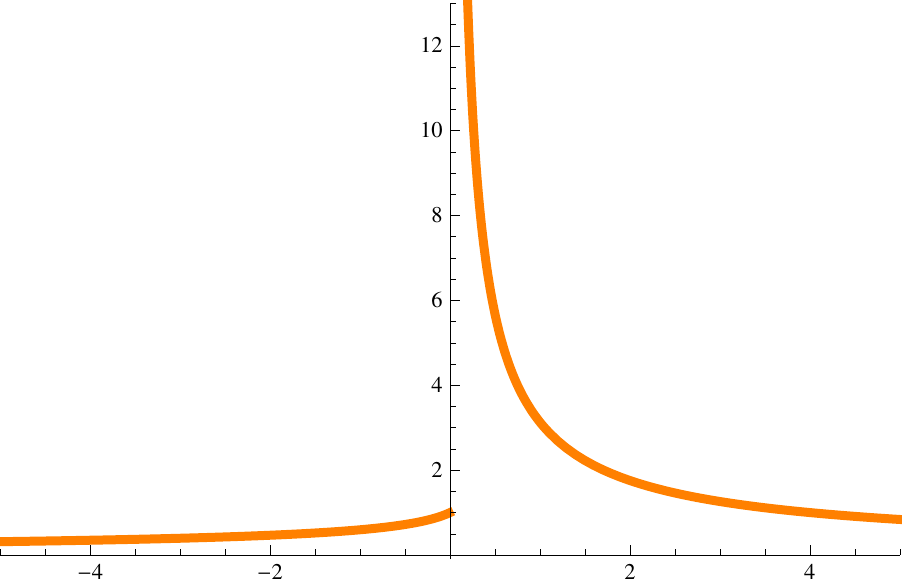}\hskip.2cm \includegraphics[width=4cm,height=4cm]{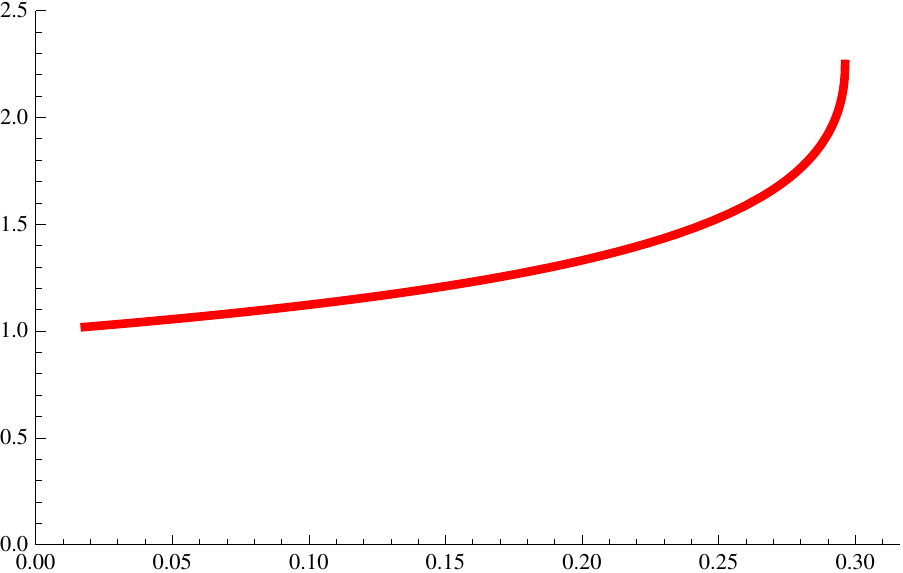}\hskip.2cm\includegraphics[width=4cm,height=3cm]{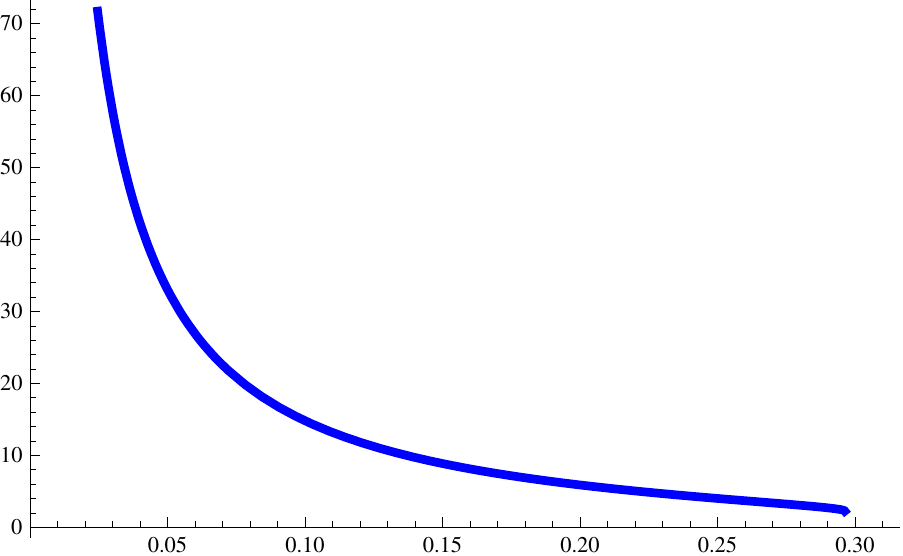}\hskip.2cm\includegraphics[width=4cm,height=3cm]{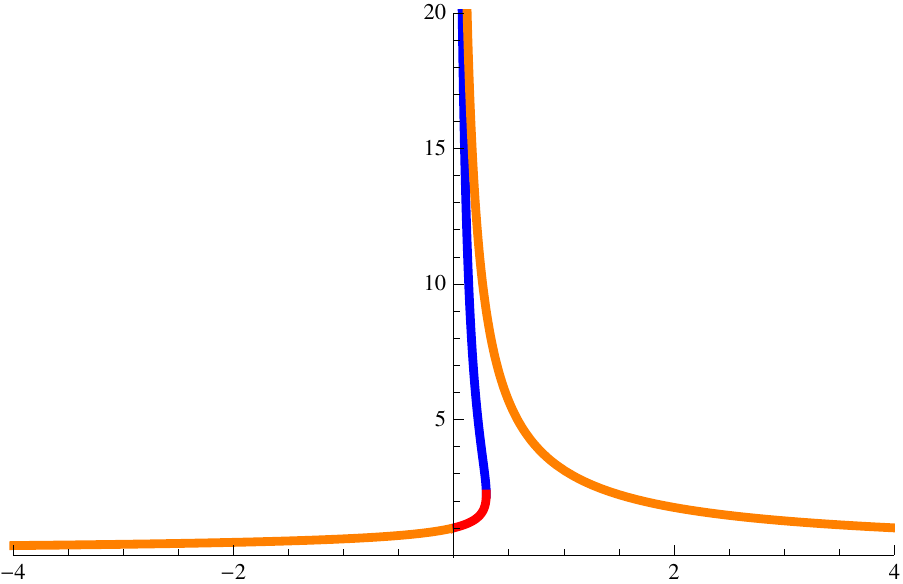}}
\caption{Graph of the function $r_1(a)$, $r_2(a)$ and $r_3(a)$. The last graph show all three function in the same cartesian plane\:.}
\label{graphs of r}
\end{figure}

When $\Lambda_0=1$, the polynomial $q(r,a,C)$ reduces to

$$q=q(r,a,C)=-16 C^2 + 64 r+32 C\,  r-16 \, r^2 - 8 a C\,  r^2+8 a \, r^3 -a^2\,  r^4\:. $$

Recall that, given two particular values $a$ and $C$, the level set $G=C$ is not empty  if and only if there exist two positive consecutive roots of the polynomial $q(r,aC)$. In every case $q(0)=-16 C^2\le0$ and the limit when $r$ goes to infinity of $q(r)$ is negative infinity. The following lemma provides the number of roots of $q(r,C)$ depending on the values $a$ and $C$.

\begin{lem} \label{roots}
Let $r_1$, $r_2$ and $r_3$ be as in Definition \ref{def of ri} and for $i=1, 2, 3$  define
$$ C_i=\frac{16-8r_i+6a{r_i}^2-a^2{r_i}^3}{4(-2+ar_i)} $$

and

$$q=q(r, a,C)=-16 C^2 + 64 r+32 C\,  r-16 \, r^2 - 8 a C\,  r^2+8 a \, r^3 -a^2\,  r^4 \:.$$

Recall that the domain of $C_i$ is the same domain of $r_i$. That is, the domain of $C_1(a)$ is   $\{a\ne0\}$ and the domain of $C_2(a)$ and $C_3(a)$ is the interval  $(0,\frac{8}{27}]$.   For any $a\ne0$ and any $C$, the polynomial $q$ has non negative real roots whose multiplicities are given in the following table.

\begin{tabular}{| l | l | c |}
\hline
{\rm range of }$a$&{\rm range of} $C$&{\rm number of distinct real roots of }$q$\\ \hline

$a<0\Rightarrow C_1(a)<0$&$C<C_1(a)$&$0$\\
&$C=C_1(a)$&$1$\\
&$C>C_1(a)$&2\\
\hline \hline
$a\in (0,\frac{8}{27}) \Rightarrow C_2(a)<0$,  $C_1(a)>0$,\\[2mm] $C_2(a)<C_3(a)<C_1(a)$&$C>C_1(a)$&$0$\\
&$C=C_1(a)$&$1$\\
&$C_3(a)<C<C_1(a)$&$2$\\
&$C=C_3(a)$&$3$\\
&$C_2(a)<C<C_3(a)$&$4$\\
&$C=C_2(a)$&$3$\\
&$C<C_2(a)$&$3$\\
\hline \hline
$a=\frac{8}{27}\Rightarrow C_2(a)=C_3(a)=-9/8$, \\[2mm] $C_1(\frac{8}{27})=9$&$C<-9/8$&$2$\\
&$C=-9/8$&$2$(*)\\
&$-9/8<C<9$&$2$\\
&$C=9$&$1$(**)\\
&$9>C$&$0$\\
\hline \hline
$a> 8/27 \Rightarrow C_1(a)>0$& $C>C_1(a)$&$0$\\
&$C=C_1(a)$&$1$\\
&$C<C_1(a)$&$2$\\
\hline 
\end{tabular}

(*) In this case the roots are $9/4$ with multiplicity $3$ and $81/4$ with multiplicity one.\\
(**) In this case the only real root is $9$ with multiplicity $2$. 
\end{lem}

\begin{proof} We first show that the polynomial $q$ has no negative roots. Since  $ \lim_{r \to -\infty} q= -\infty$
and $q(0)\le 0$, we have that if $q$ has a non negative root, then $q$ must have a local non negative maximum at some negative $r_0$. This is impossible because solving the equation

$$q^\prime(r,C,a)=-4 (-16 - 8 C + 8 r + 4 a C r - 6 a r^2 + a^2 r^3)  = 0 $$

we obtain that

$$C_0=\frac{16-8 r_0+6 a{r_0}^2 - a^2 {r_0}^3 }{4(-2+ar_0)} \:.$$

Replacing this value for $C$ in the expression for $q$ we obtain

$$q(r_0,C_0)=64\left(  r_0 - \frac{4}{(-2+a r_0)^2}\right)\:.$$

Notice that it is impossible to have $q(r_0,C_0)$ equal to zero with $r_0$ negative. This finishes the proof of the fact that $q$ has no negative roots. Using the interpretation of the polynomial $q$ with the cylindrical rotating drops, we have that a circular cylinder with radius $R$ is a rotating drop if and only if $r=R^2$ and some $C$ solve the system of equations

$$q(r,C)=0\com{and} q^\prime(r,C)=0\:. $$

Using Lemma \ref{round} and the definition of $r_i$,  we obtain that the only solution of the system of equations $q(r,C)=0\com{and} q^\prime(r,C)=0 $ are $(r_1,C_1)$, $(r_2,C_2)$ and $(r_3,C_3)$. Now the proof follows by a continuity argument.
For example to show the case $a<0$ we notice that the curve
$C=C_1(a)$ divides the region $\{(a,C):a<0\}$ into two pieces. A direct computation shows that the limit when $a$ goes to $-\infty$ of $C_1(a)$ is $0$ and the limit when $a$ goes to $0^-$ of $C_1(a)$ is $-1$. We can directly check that if $a=-2$ then when $c=0$ the polynomial has two positive roots, when $a=-2$ and $C=C_1(-2)=-0.790706...$ then $q$ has exactly one root, $r_1$, and finally if $a=-2$ and $C=-1$ then the polynomial $q$ has not roots. Figure 1.3 shows the graphs of these polynomials.

\begin{figure}[ht]
\centerline{\includegraphics[width=4cm,height=3cm]{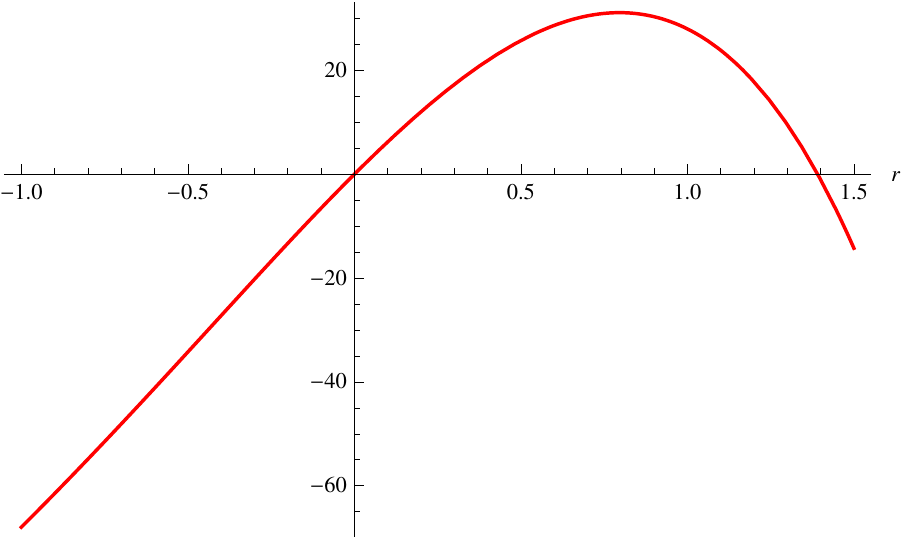}\hskip.2cm \includegraphics[width=4cm,height=4cm]{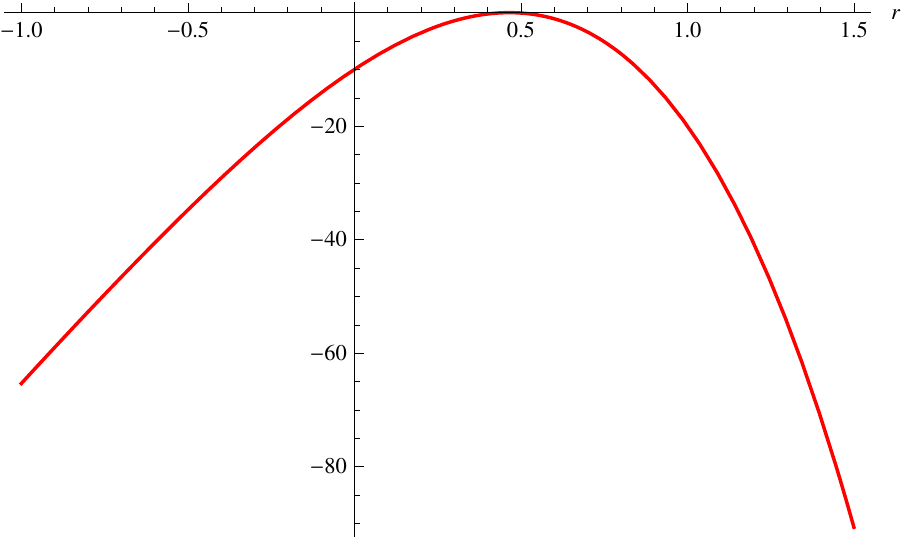}\hskip.2cm\includegraphics[width=4cm,height=3cm]{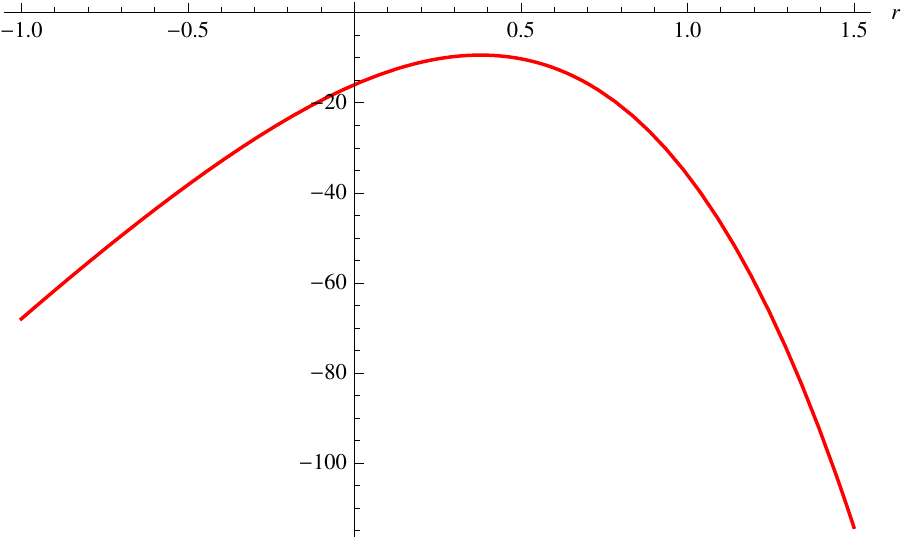}}
\label{c graphs}
\caption{  Graph of the polynomial $q$ when $a=-2$ and $C=0$, $C=C_1(-2)=-0.79...$ and $C=-1$  respectively. }
\end{figure}

Now we can argue by contradiction that the polynomial $q$ does not have two roots when  $a<0$ and $C<C_1(a)$. Let us assume that for some $(\tilde{a},\tilde{C})$, $q(r, \tilde{a}, \tilde{C})$ has one or more roots and let us consider a curve $\beta$  that connects the point $(\tilde{a},\tilde{C})$ with the point $(-2,-1)$  such that $\beta$ is contained in the portion of the $a$-$C$ plane given by $\{ (a,C):a<0, \quad C<C_1(a) \}$. By the continuity of the function $q(r,C,a)$,  we have that for some value $(a,C)$ on the curve $\beta$, the polynomial $q$ must have a root with multiplicity greater than $1$. This is impossible because the only solution of the system  $q(r,C)=0\com{and} q^\prime(r,C)=0 $ are $(r_1,C_1)$, $(r_2,C_2)$ and $(r_3,C_1)$. All the other cases are similar. Figure \ref{graph of q} shows the graphs of the polynomial $q$ for some other values of $a$ and $C$.

\end{proof}

\begin{figure}[ht]
\centerline{\includegraphics[width=3cm,height=2.5cm]{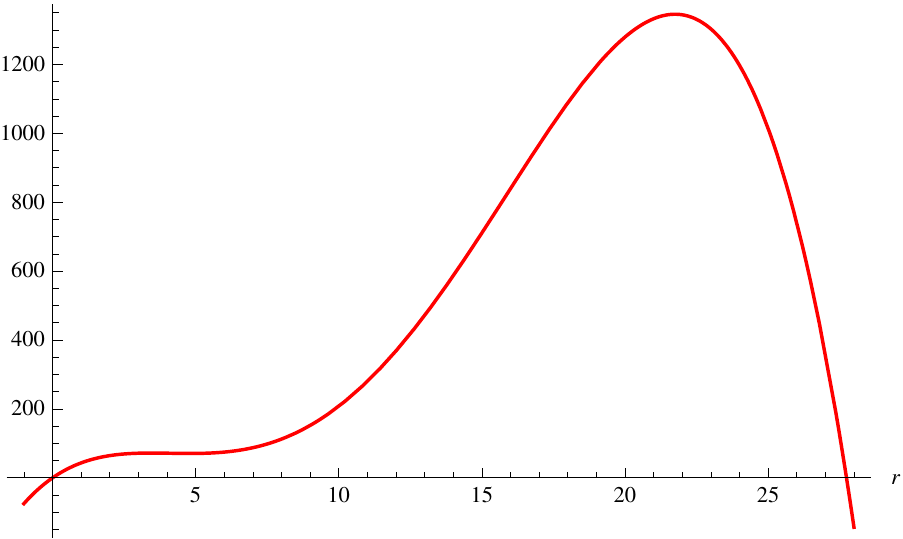}\hskip.2cm \includegraphics[width=3cm,height=2.5cm]{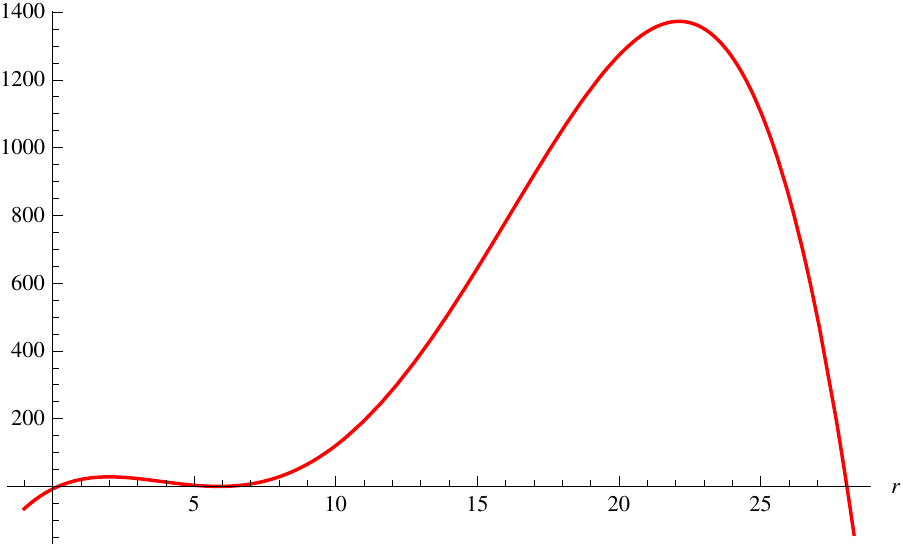}\hskip2cm\includegraphics[width=3cm,height=2.5cm]{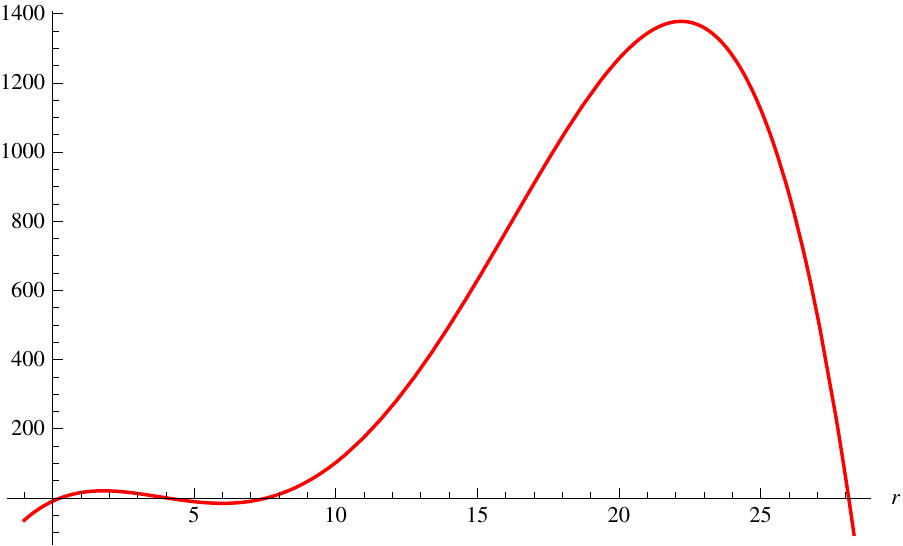}\hskip.2cm\includegraphics[width=3cm,height=2.5cm]{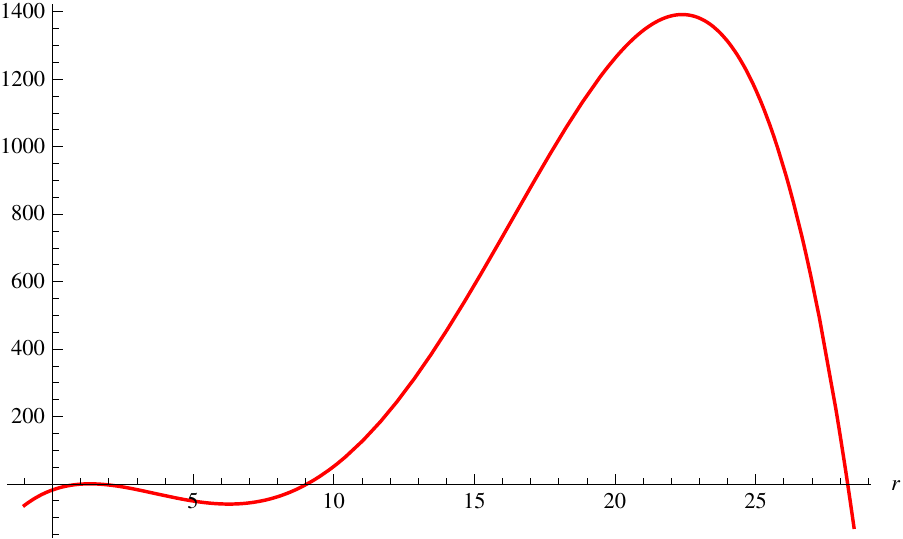}\hskip.2cm \includegraphics[width=3cm,height=2.5cm]{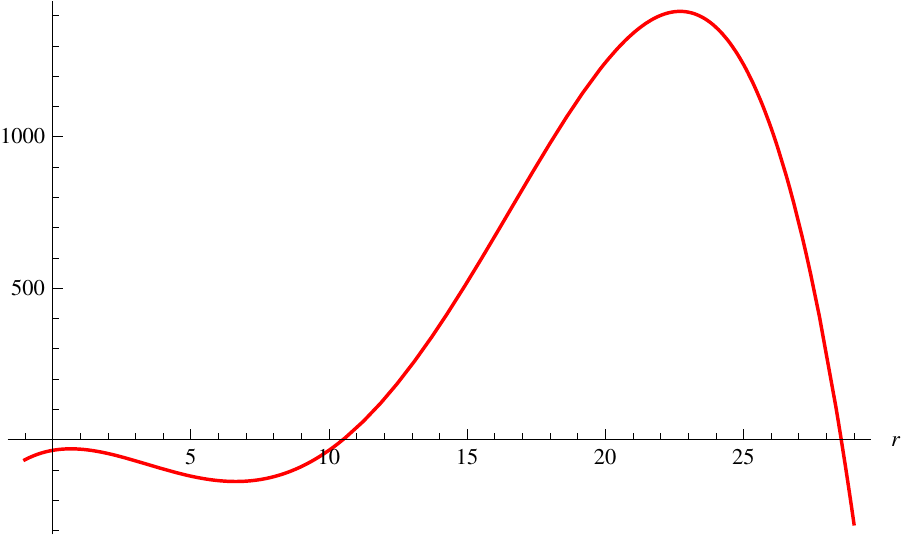}}
\caption{ Graph of the polynomial $q$ when $a=0.2$ and $C=-0.2$, $C=C_3(0.2)=-0.69...$, $C=-0.8$, $C=C_2(0.2)=-1.06...$ and $C=-1.5$ respectively.}
\label{graph of q}
\end{figure}

\begin{figure}[ht]
\centerline{\includegraphics[width=4cm,height=3cm]{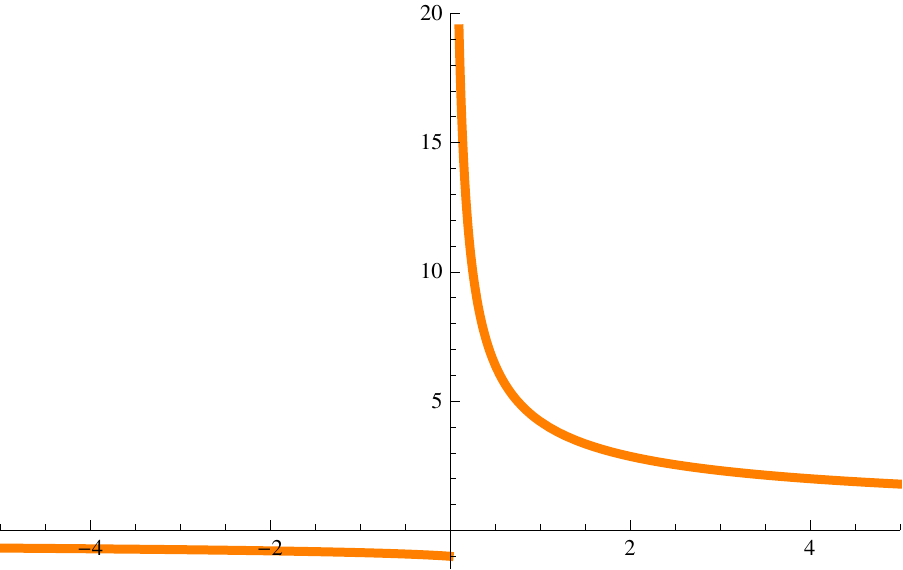}\hskip.2cm \includegraphics[width=4cm,height=4cm]{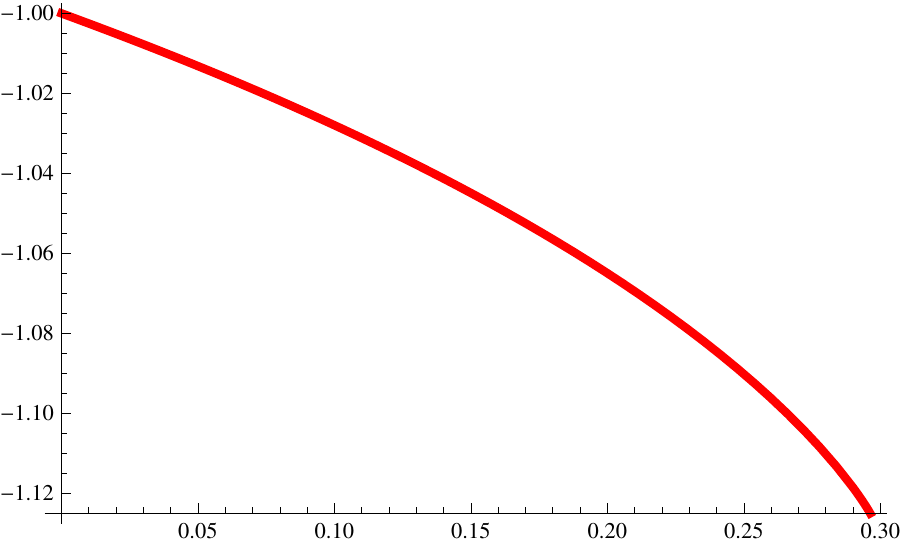}\hskip.2cm\includegraphics[width=4cm,height=3cm]{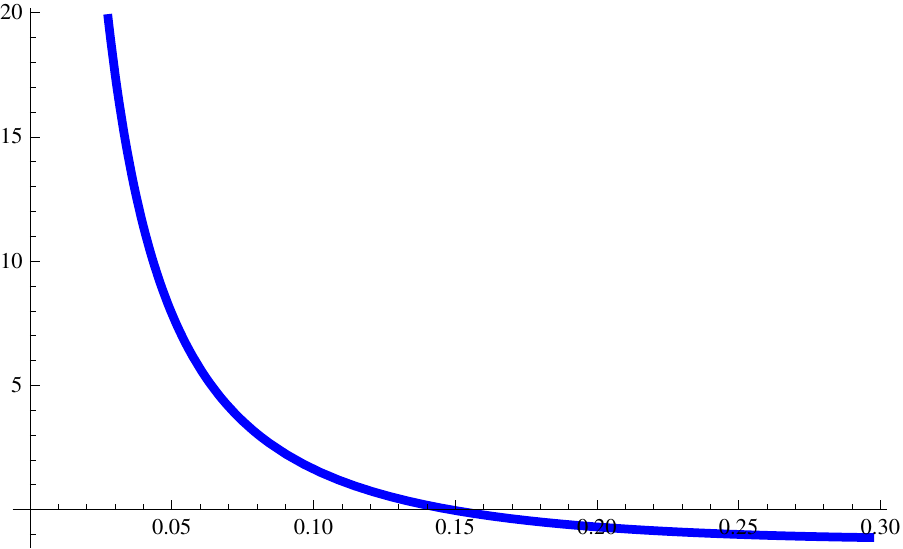}\hskip.2cm\includegraphics[width=4cm,height=3cm]{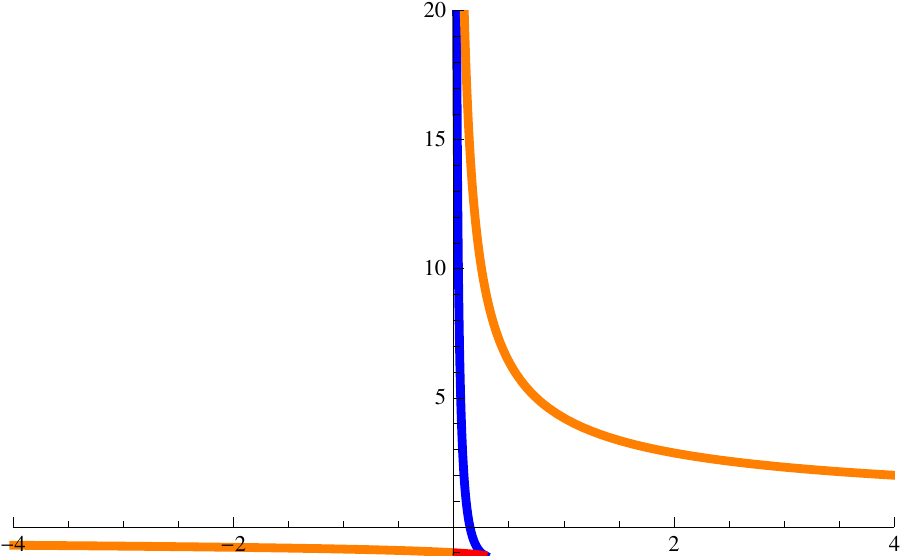}}
\label{c graphs 2}
\caption{ Graph of the functions $C_1(a)$, $C_2(a)$ and $C_3(a)$. The last graph shows all three functions in the same Cartesian plane.}
\end{figure}

\begin{thm}\label{conf space b=1} Let $\Lambda_0=1$ and let $\Delta{\tilde \theta}$ be the function defined in equation (\ref{change}).  Define

$$\Omega_1=\{(a,C): C> C_1(a),\, a < 0\}\quad \Omega_2=\{(a,C): C_2(a) < C < C_3(a),\, 0<a< \frac{8}{27}\}\,\:,$$

$$\Omega_3=\{(a,C): C< C_1(a),\, a > 0\}\quad \Omega =\Omega_1\cup\Omega_3\setminus \Omega_2\:,$$

$$\beta_1=\{ (a,C)\, :\,  C=C_1(a),\,  a\ne 0  \}\:,$$

$$\beta_2= \{ (a,C)\, :\,  C=C_2(a),\, 0< a<\frac{8}{27}  \} \:,$$

$$\beta_3=\{ (a,C)\, :\,  C=C_3(a),\, 0< a\le \frac{8}{27}  \}\:. $$

Under the convention that a point $(a,C)$ represents a cylindrical rotating drop if the TreadmillSled of its profile curve is contained in the level set $G=C$, we have:

{\bf i.)}  Every point $(a,C)$ in the interior of  $\Omega$ represents a cylindrical rotating drop such that the length of the fundamental piece is finite. The TreadmillSled of the profile curve of this surface is parametrized by $\rho$ defined for values of $r$ between the only two roots of the polynomial $q(r,a,C)$.

{\bf ii.)}  Every point in  $\Omega_2 $ represents two cylindrical rotating drops, both with   fundamental pieces of bounded length. The TreadmillSled of the profile curves of these surfaces are parametrized by $\rho$ which is defined for those values of $r$  which move between the first and second root of the polynomial $q(r,a,C)$ and the third and fourth root of the polynomial $q(r,a,C)$ respectively.

{\bf iii.) } Every point on the curve $\beta_1$ represents a circular cylindrical rotating drop.

{\bf iv.)} Every point on the curve  $\beta_2$ represents two cylindrical rotating drops:  a circular cylinder and a non circular cylinder with bounded length of its fundamental piece.

{\bf v.)} Every point in the curve $\beta_3$ represent three cylindrical rotating drops. One is a circular rotating drop. The second one has TreadmillSled parametrized by $\rho$ defined for those values of $r$  that move between the first and second root of the polynomial $q(r,a,C)$, recall the the second root has multiplicity $2$. The third surface   has TreadmillSled parametrized by $\rho$ defined for those values of $r$  that moves between the second and third root of the polynomial $q(r,a,C)$. The second and third surface are not properly immersed and their profile curve converge to a round circle going around infinitely many times. Solutions like the second or third surface in this case will be call \textbf{ drops of exceptional type}. See Figure \ref{pc5}.

{\bf vi.)} The point $(a,C)=(\frac{8}{27},-\frac{9}{8} )$ represents two cylindrical rotating drops:  a circular cylinder and one drop of exceptional type.

{\bf vii.)} Up to a rigid motion, every cylindrical drop falls into one of the cases above.

{\bf viii.)} Every drop that is not exceptional is either properly immersed (when $d\Delta{\tilde \theta}(a,C)/\pi$ is a rational number) or it is dense in the region bound by two round cylinders (when $d\Delta{\tilde \theta}(a,C)/\pi$  is an irrational number).

\end{thm}

\begin{proof}

We already know that the TreadmillSled of the profile curve of any rotating drop satisfies the equation

 $$G(\xi_1,\xi_2)=  2 \xi_2 +  (\xi_1^2+\xi_2^2) -\frac{a}{4} (\xi_1^2+\xi_2^2)^2  =C\:.$$

 Figure \ref{TSofcyl}  How the TreadmillSled of a cylindrical drop generates the level set of $G$.

\begin{figure}[ht]
\centerline{\includegraphics[width=4cm,height=3cm]{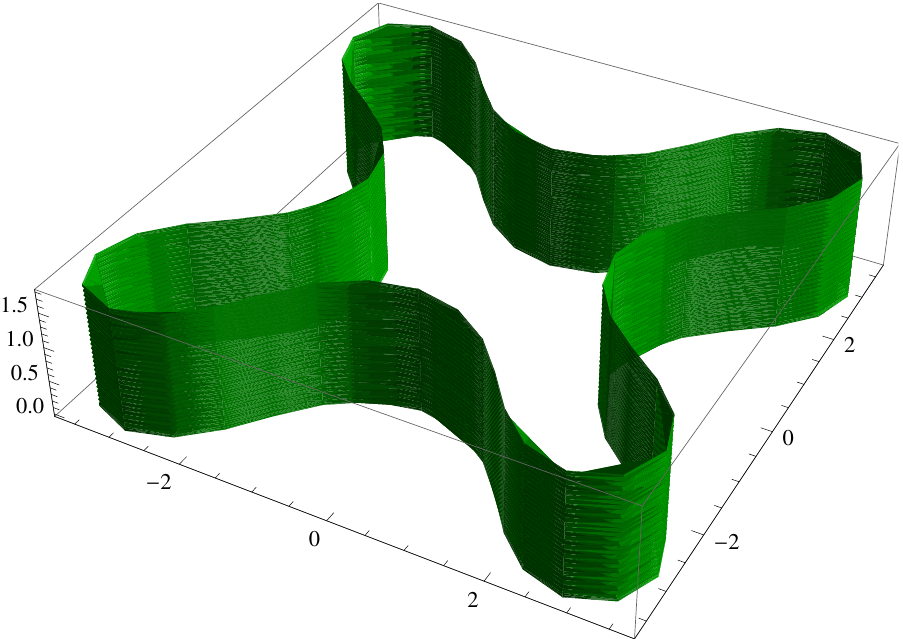}\hskip.2cm \includegraphics[width=4cm,height=4cm]{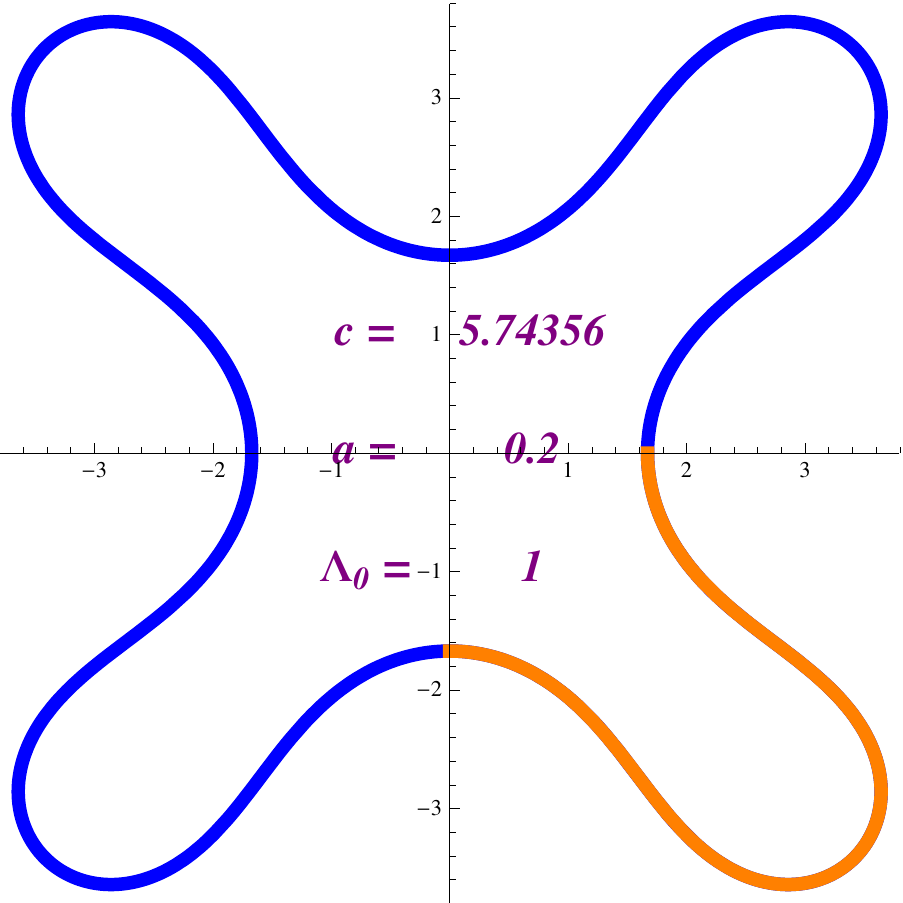}\hskip.2cm\includegraphics[width=3.5cm,height=3.5cm]{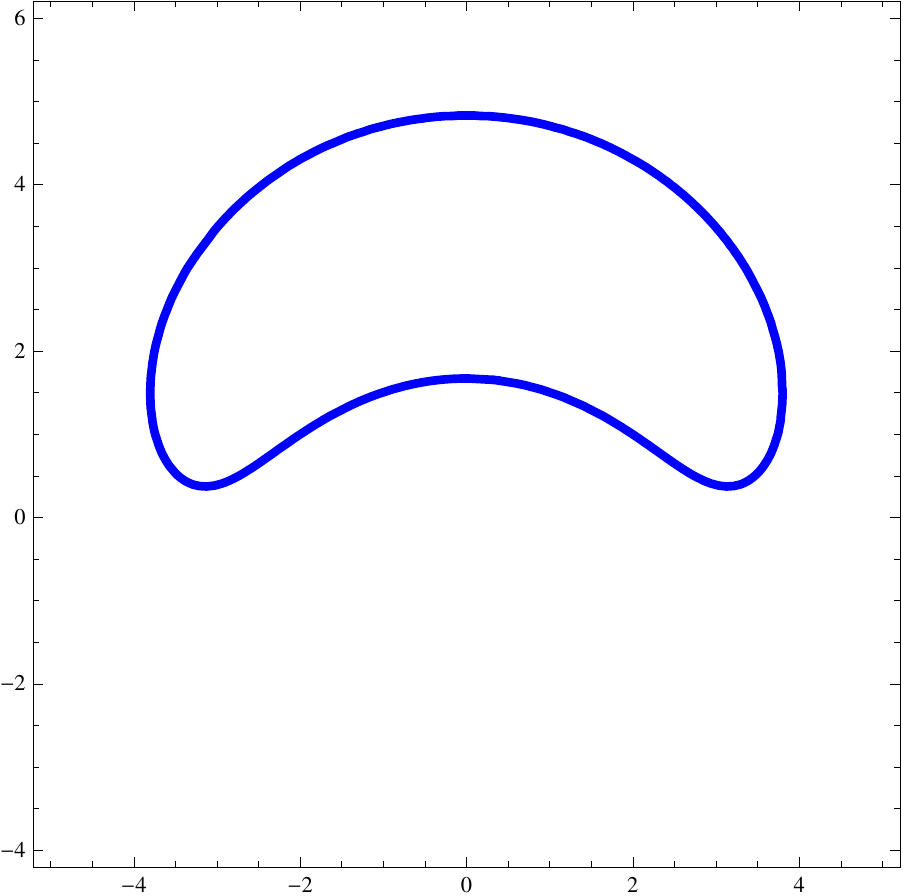}\hskip.2cm\includegraphics[width=3.5cm,height=3.5cm]{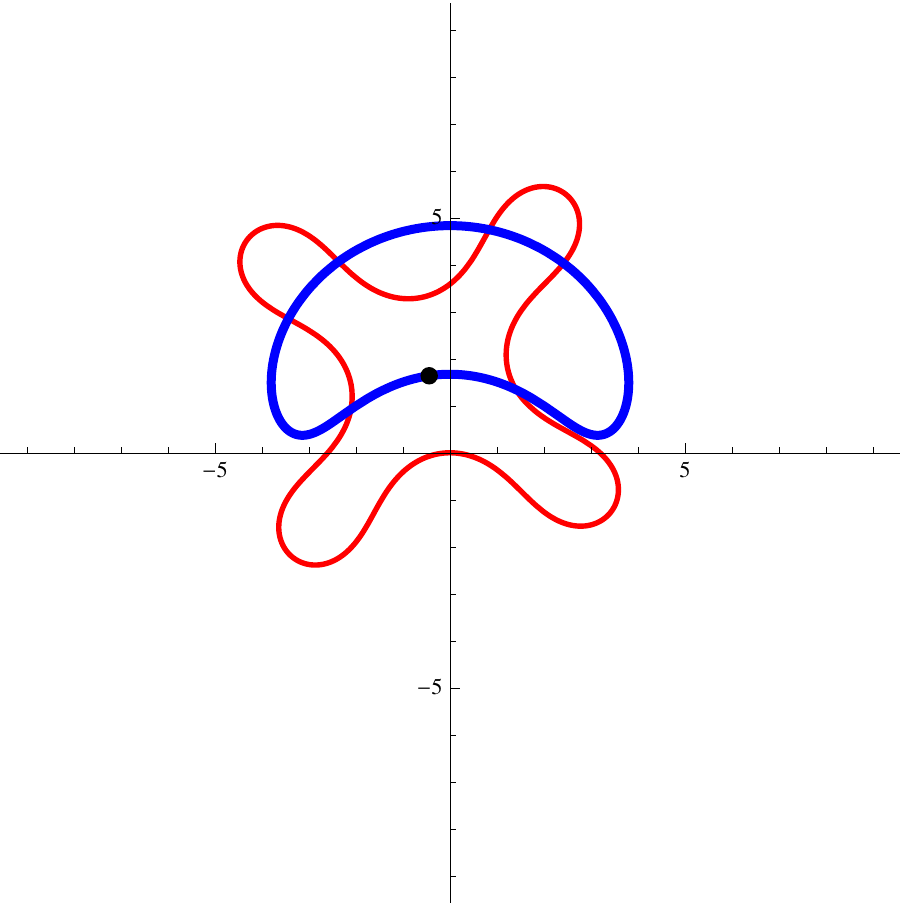}}
\caption{ The first picture shows a cylindrical rotating drop, the second picture shows its profile curve emphasizing the fundamental piece, the third picture show the level set $G=C$ and the last picture shows how the TreadmillSled of the profile curve produces the level set $G=C$, in this particular example the TreadmillSled of the profile curve will go over the level set $G=C$ four times.}
\label{TSofcyl}
\end{figure}

We have that up to rigid motions, the TreadmillSled of a curve determine the curve, see \cite{P3}.  Since any level sets of $G$  can be parametrized using the map $\rho$ given in Definition \ref{def of rho}, and since every parametrization of $G$ is defined for values of $r$ where the polynomial $q$ is positive, it follows from Lemma \ref{roots} that every cylindrical rotating drops can be represented as one of the cases i.) through vi.) above. It is worthwhile to recall, see Remark \ref{prop of rho}, that the parametrization $\rho$ only covers half of the level set of  $G$. Each one of these level sets is symmetric with respect to the $\xi_2$ axis, and the parametrization $\rho$ covers the half on the right.
The cylinder shown in Figure \ref{TSofcyl} correspond with the value $C=5.74356$ and $a=0.2$. For these values of $C$ and $a$ the roots of the polynomial $q$ are approximately $x_1=2.791596$ and $x_2=23.35858$. Figure \ref{Graph of rho} shows the graph of the polynomial $q$ and the part of the level set parametrized by $\rho$ when $\rho$ moves form one root of $q$ to the next.

\begin{figure}[ht]
\centerline{\includegraphics[width=4cm,height=3cm]{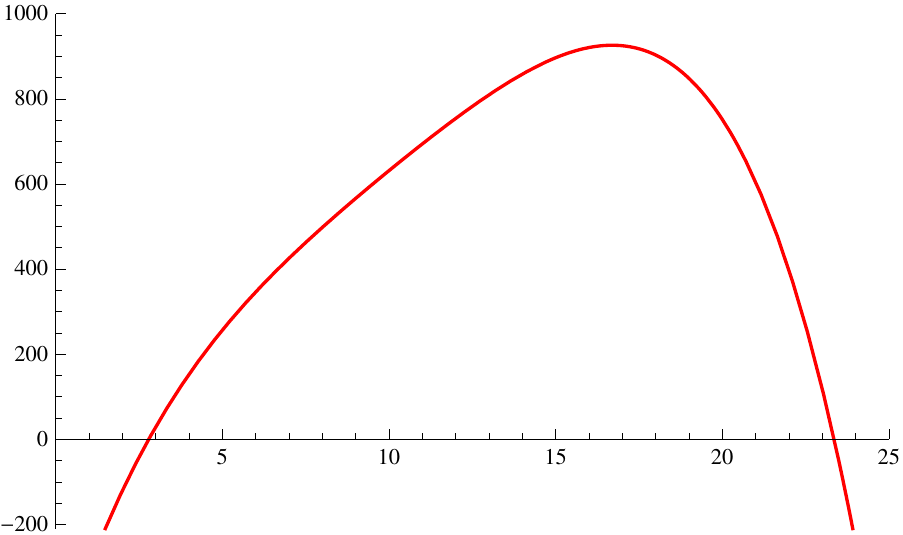}\hskip2cm \includegraphics[width=4cm,height=4cm]{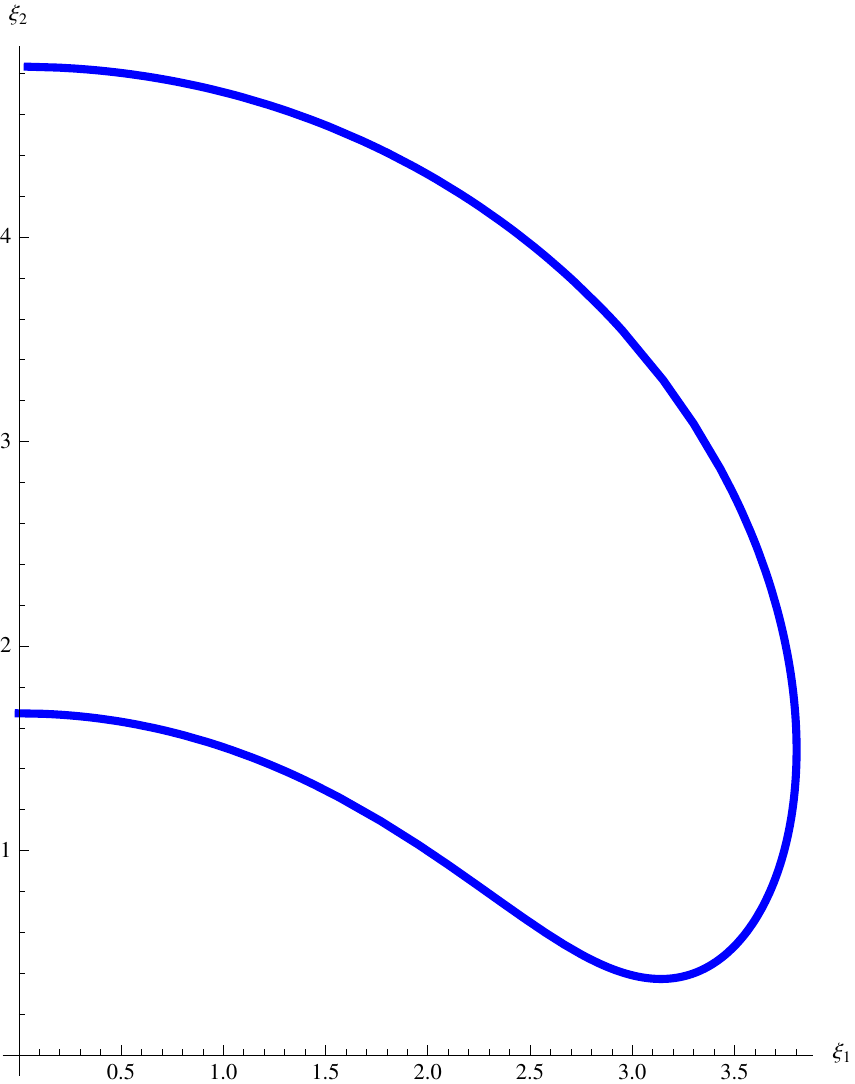}}
\caption{ The first picture shows the graph of the polynomial $q$ and the second picture shows the part of the level set parametrized by $\rho$ when $r$ moves from one root of $q$ to the other. Notice how the points in the fundamental piece whose TreadmillSled are the
initial and final point of the parametrization $\rho$ are the smallest and largest value of the distance to the origin function $R$ for the fundamental piece. We also have that the number of critical points of the function $R$ along a fundamental piece are either $2$ if the fundamental piece is a closed curve or $3$ otherwise. The reason is that critical point of $R$ correspond with values of $\xi_1=0$ in the level sets of $G$.   }
\label{Graph of rho}
\end{figure}

Notice that when the profile curve is a circle, the level set $G=C$ reduces to a point. It is easy to see that the TreadmillSled of a circle is a point. When the profile curve is a circle, we will take the parametrization $\rho$ to be defined just in a point (a root with multiplicty 2 of the polynomial $q$) and not in an interval.

When Case (i) occurs, $q$ has  only  two simple roots $x_1$ and $x_2$. The graph of $q$ looks like the one shown in Figure \ref{Graph of rho}. We can check that if the derivative of $q$ at $x_1$ is positive while the derivative of $q$ at $x_2$ is negative, then the length of the fundamental piece,  according to Equation (\ref{ds}), reduces to the integral $\int_{x_1}^{x_2}\frac{8\, dr}{\sqrt{q(r,a)}}$, which is convergent since the roots are simple . Therefore the length of the fundamental piece is finite.

For values of $C$ and $a$ that falls into  case (ii) the polynomial q has 4 roots $x_1<x_2<x_3<x_4$ and it is positive on the intervals  $(x_1, x_2)$  and   $(x_3, x_4)$. Also, the level set of $G$ has two connected components. Half of each connected component of $G=C$ can be parametrized using the map $\rho$, one half of the connected component of $G=C$ uses the domain $(x_1,x_2)$ for $\rho$ and the half of the other connected component of $G=C$ uses the domain $(x_3,x_4)$ for $\rho$. For case (ii), the graph of the polynomial $q$ and the level set of $G$ looks like those shown in  Figure \ref{poly 2}. Figure \ref{pc2a} and Figure \ref{pc2b} shows images of the fundamental pieces of the cylindrical drops for the values $C=0$ and $a=0.1$. These two values for $a$ and $C$ falls into case (ii). The proof that that the length of the fundamental piece of each surface is finite follow the lines of the proof in case (i).

 \begin{figure}[ht]
\centerline{\includegraphics[width=4cm,height=3cm]{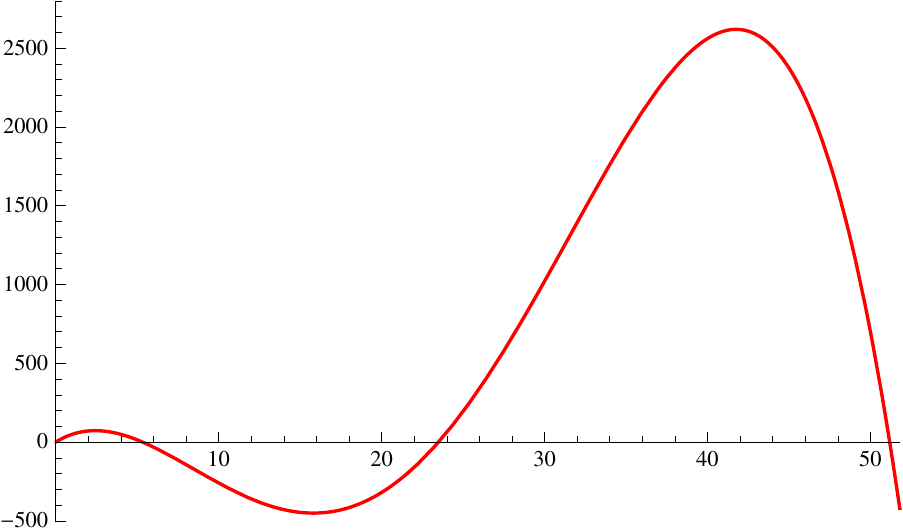}\hskip2cm \includegraphics[width=4cm,height=4cm]{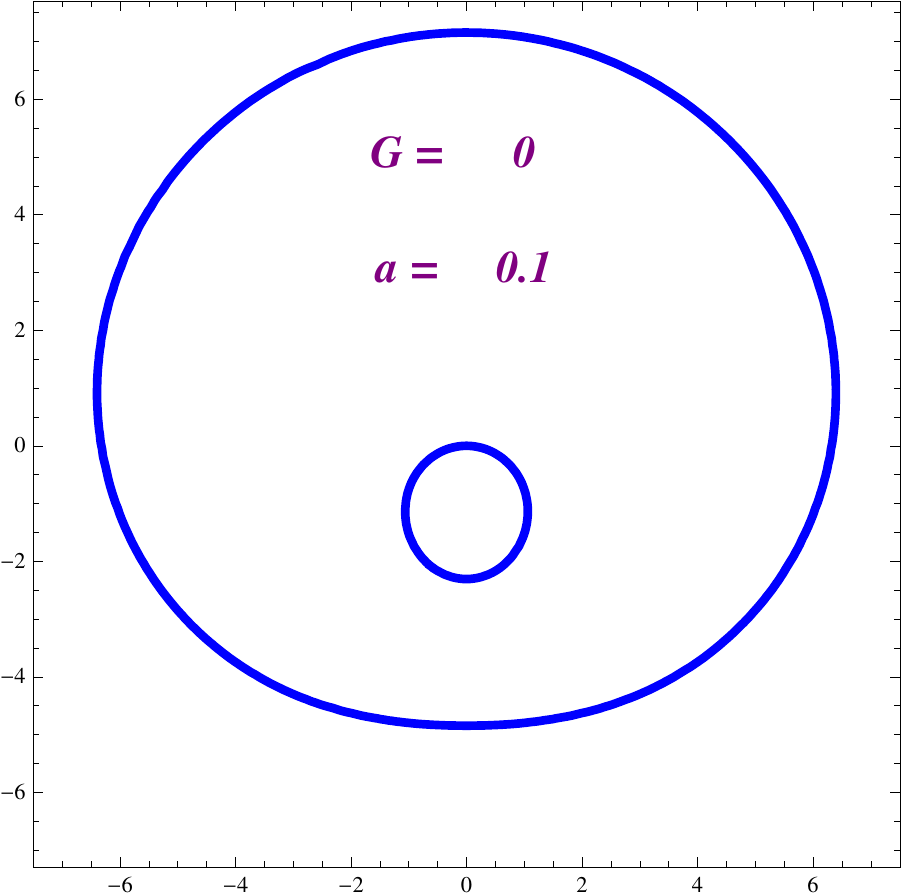}}
\caption{ The following graph shows the polynomial $q$ and the level set of $G$ when $c=0$ and $a=0.1$ . }
\label{poly 2}
\end{figure}

 \begin{figure}[ht]
\centerline{\includegraphics[width=4cm,height=3cm]{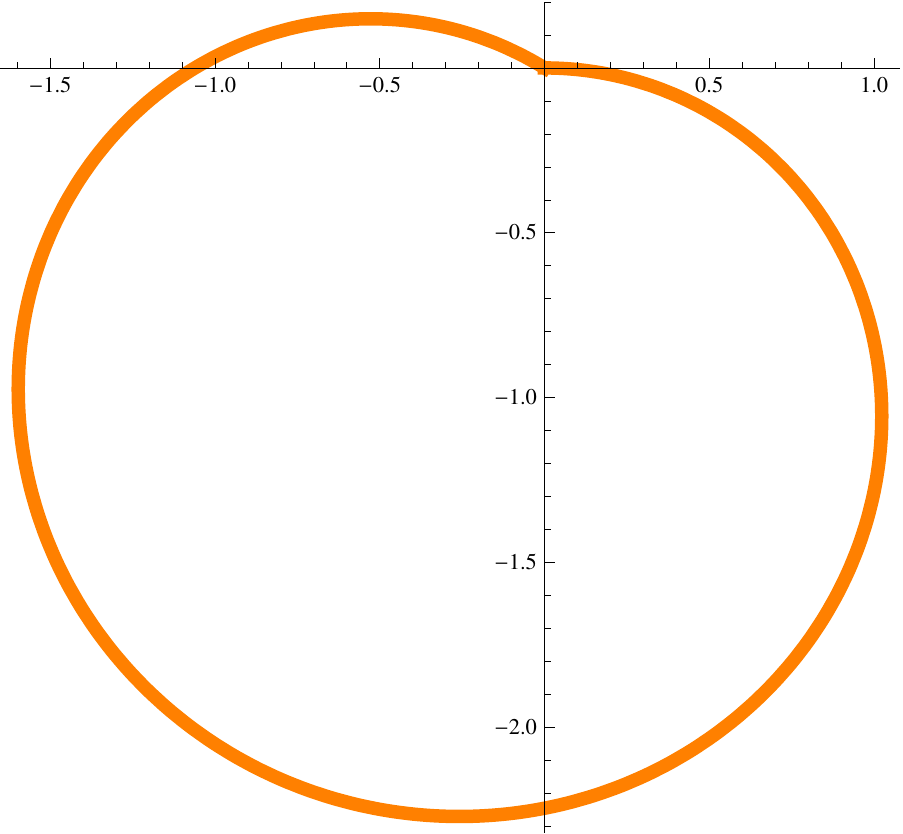}\hskip.2cm \includegraphics[width=4cm,height=4cm]{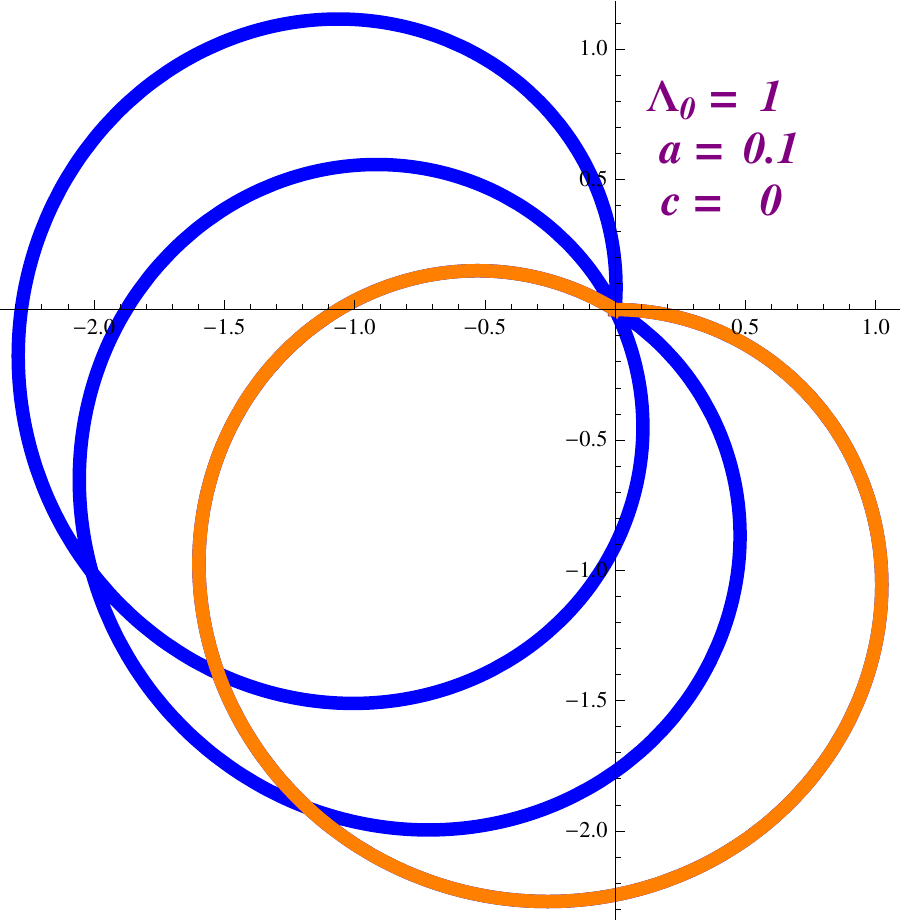}\hskip.2cm\includegraphics[width=3.5cm,height=3.5cm]{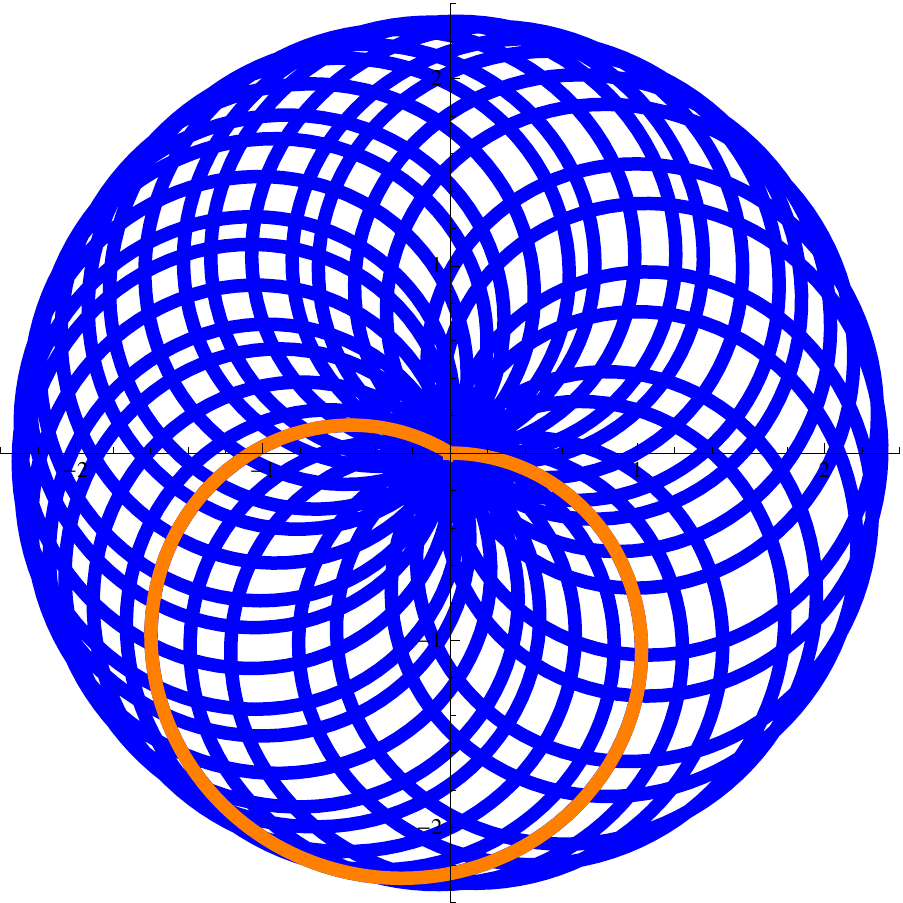}\hskip.2cm\includegraphics[width=3.5cm,height=3.5cm]{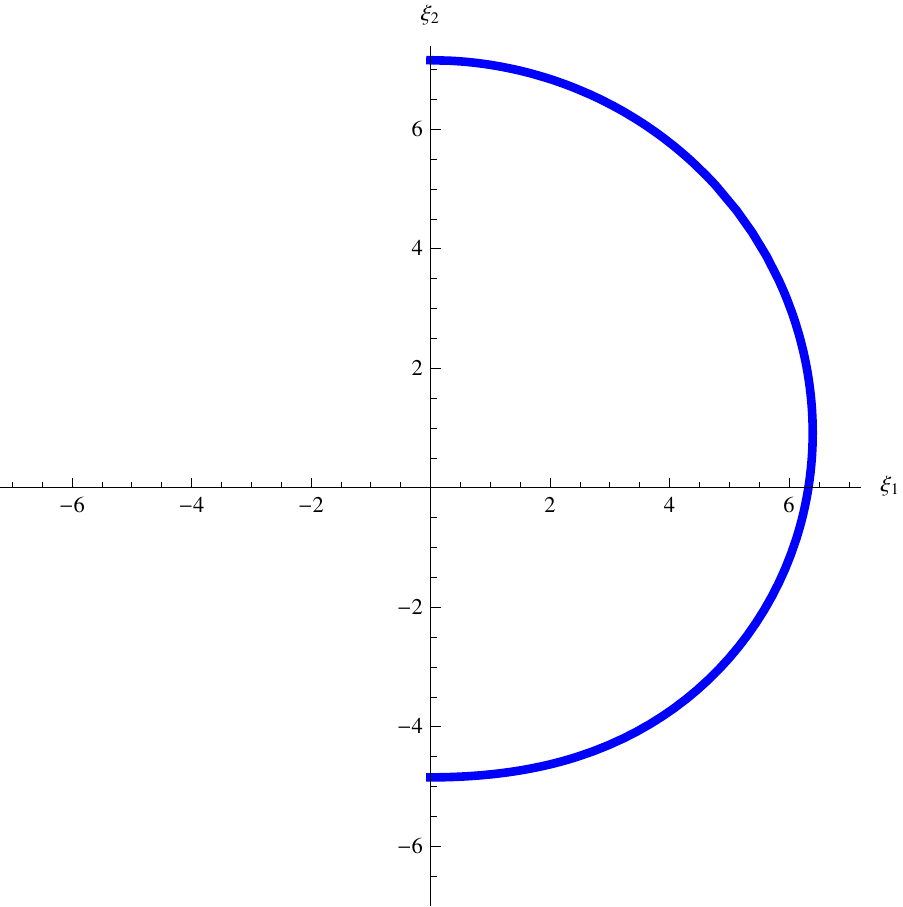}}
\caption{ When $C=0$ and $a=0.1$ there are two cylindrical drops, these images corresponds to one of them. The last picture shows half of the TreadmillSled produced by the profile curve which is shown in the previous three pictures. This last picture is  parametrized by the map $\rho$ when $r$ moves from the first root of $q$  (in this case is $x_1=0$)  and the second root of $q$. The first picture shows the fundamental piece of the cylindrical surface, the second one shows the part of the profile curve made out of  3 copies of the fundamental piece and the third picture shows the part of the profile curve made out of  30 copies of the fundamental piece.}
\label{pc2a}
\end{figure}

 \begin{figure}[ht]
\centerline{\includegraphics[width=4cm,height=3cm]{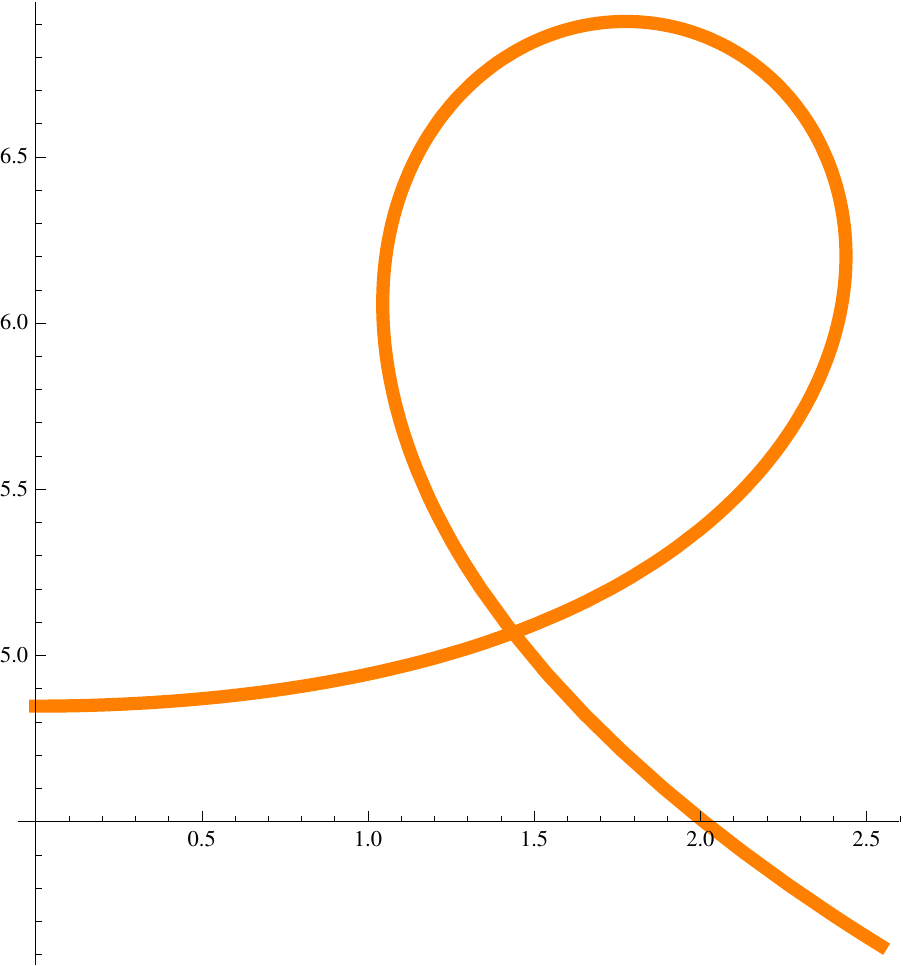}\hskip.2cm \includegraphics[width=4cm,height=4cm]{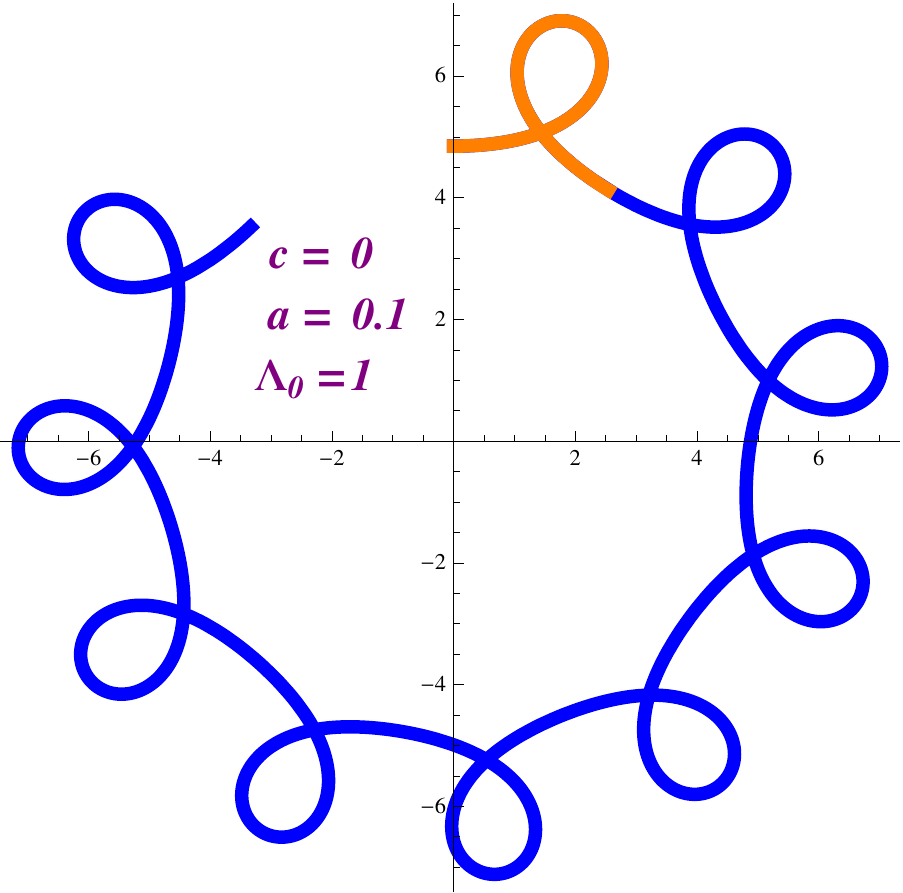}\hskip.2cm\includegraphics[width=3.5cm,height=3.5cm]{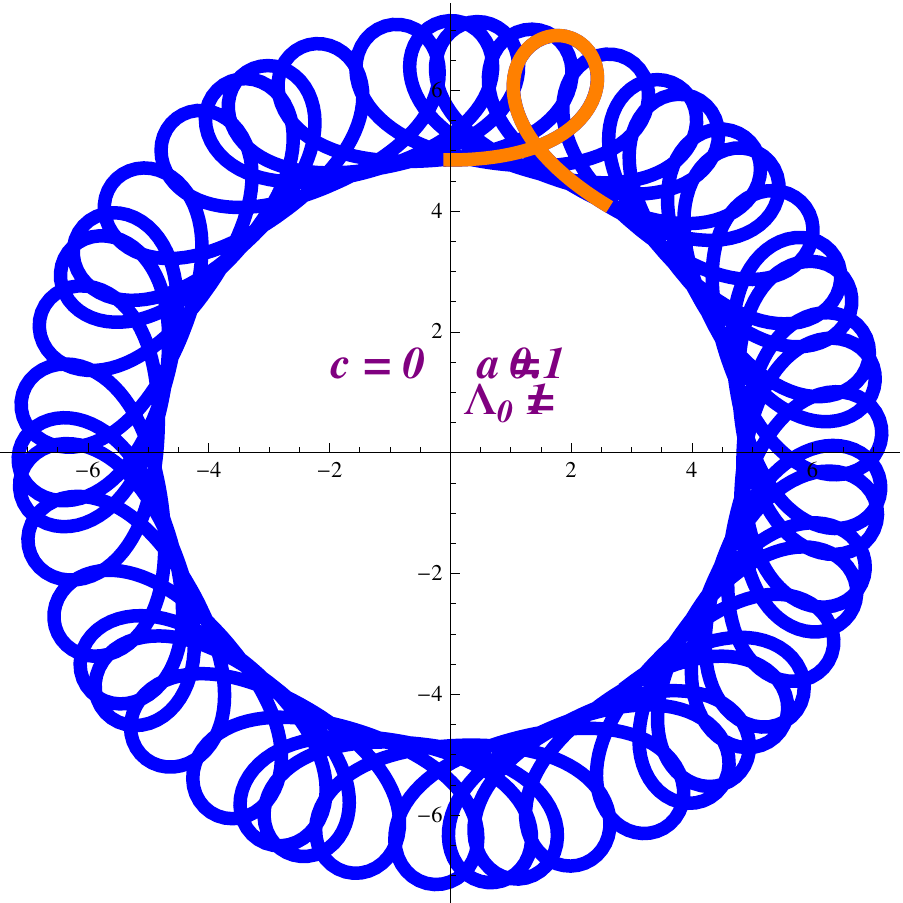}\hskip.2cm\includegraphics[width=3.5cm,height=3.5cm]{rho2a-eps-converted-to.pdf}}
\caption{ When $C=0$ and $a=0.1$ there are two cylindrical drops, these images corresponds to one of them. The last picture shows half of the TreadmillSled produced by the profile curve which is shown in the previous three pictures. This last picture is  parametrized by the map $\rho$ when $r$ moves from the third root of $q$ and the fourth root of $q$. The first picture shows the fundamental piece of the cylindrical surface, the second one shows the part of the profile curve made out of  10 copies of the fundamental piece and the third picture shows the part of the profile curve made out of  50 copies of the fundamental piece.}

\label{pc2b}
\end{figure}

 For values of $(a,C)$ that satisfies the case (iii),  the polynomial $q$ has only one root $x_1$ with multiplicity two. Let us take $R=\sqrt{x_1}$.  A direct verification shows that if $a>0$, then $R_1(a)=-R$ and if we consider the profile curve $\alpha(s)=(R \sin(\frac{s}{R}),-R\cos(\frac{s}{R}))$, then $\xi_1=0$, $\xi_2=R$ and $G(\xi_1,\xi_2)= 2R+\Lambda_0R^2-\frac{a}{4}R^4$. Using the definition of $C_1$  and the fact that $R_1(a)=-R$, we can check that the expression $G(\xi_1, \xi_2)$ reduces to $C=C_1(a)$, which was our goal in order to show that the point $(a,C)$ represents a round cylinder.  Similarly, a direct verification shows that  if $a<0$, then $R_1(a)=R$ and if we consider the profile curve $\alpha(s)=(R \sin(\frac{s}{R}),R\cos(\frac{s}{R}))$, then $\xi_1=0$, $\xi_2=-R$ and $G(\xi_1,\xi_2)= -2R+\Lambda_0R^2-\frac{a}{4}R^4$. Using the definition of $C_1$  and the fact that $R_1(a)=R$, we can check that the expression $G(\xi_1, \xi_2)$ reduces to $C=C_1(a)$.  This finish the proof of part (iii). Figure \ref{poly 3} shows the graph of the polynomial, the round cylinder and the level set of $G=C$ when $a=-2$ and $C=C_1(-2)$. 
 
    \begin{figure}[ht]
\centerline{\includegraphics[width=4cm,height=3cm]{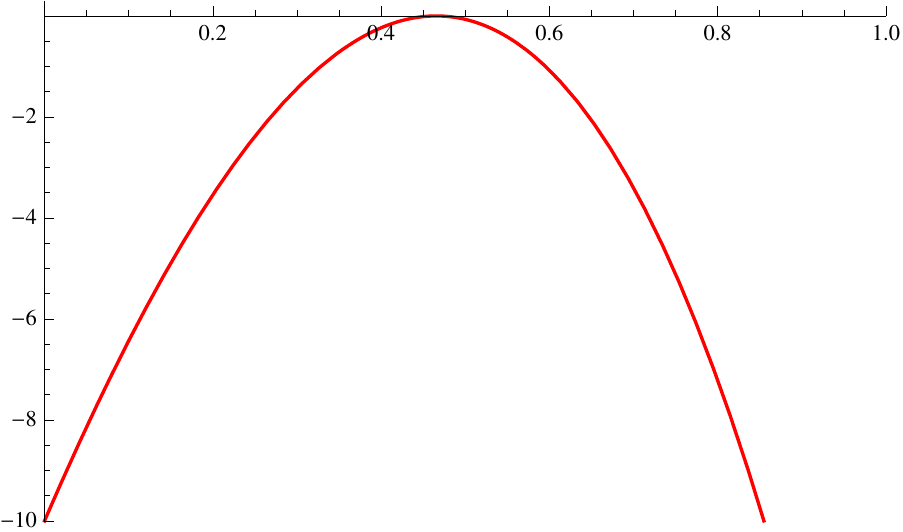}\hskip.2cm \includegraphics[width=4cm,height=4cm]{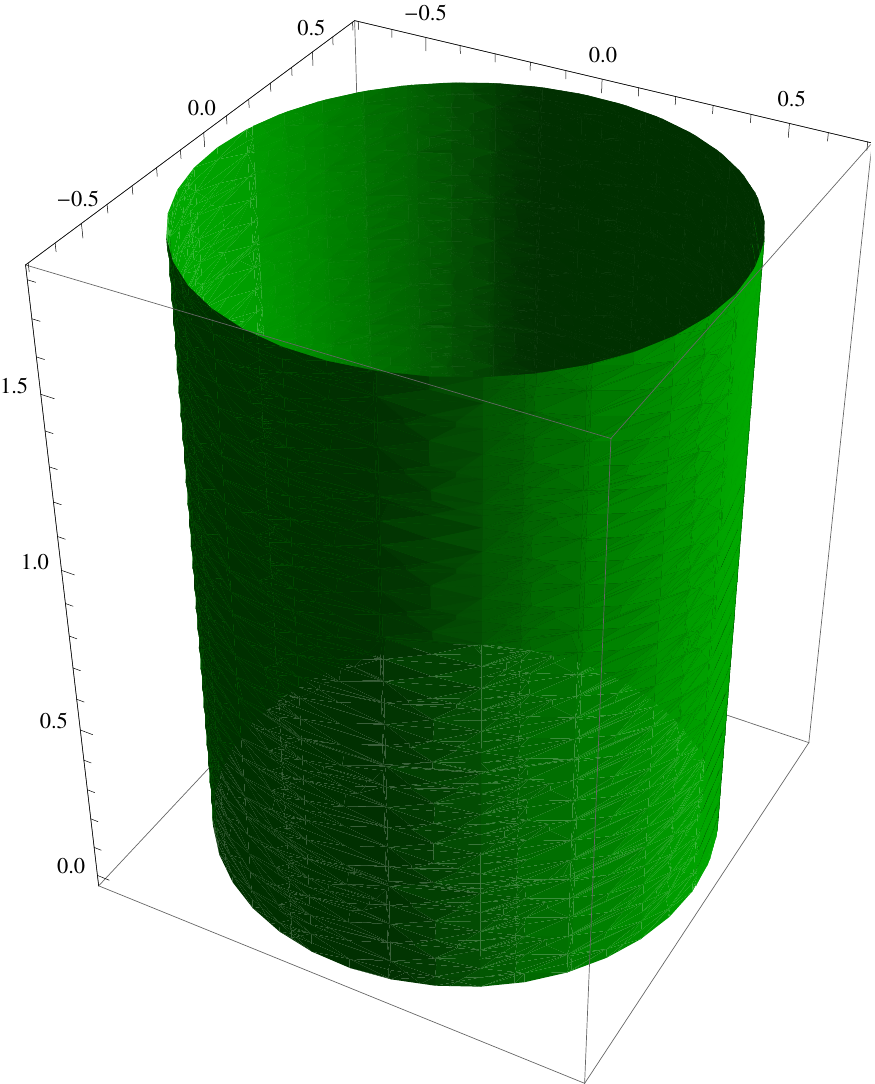}\hskip.2cm\includegraphics[width=3.5cm,height=3.5cm]{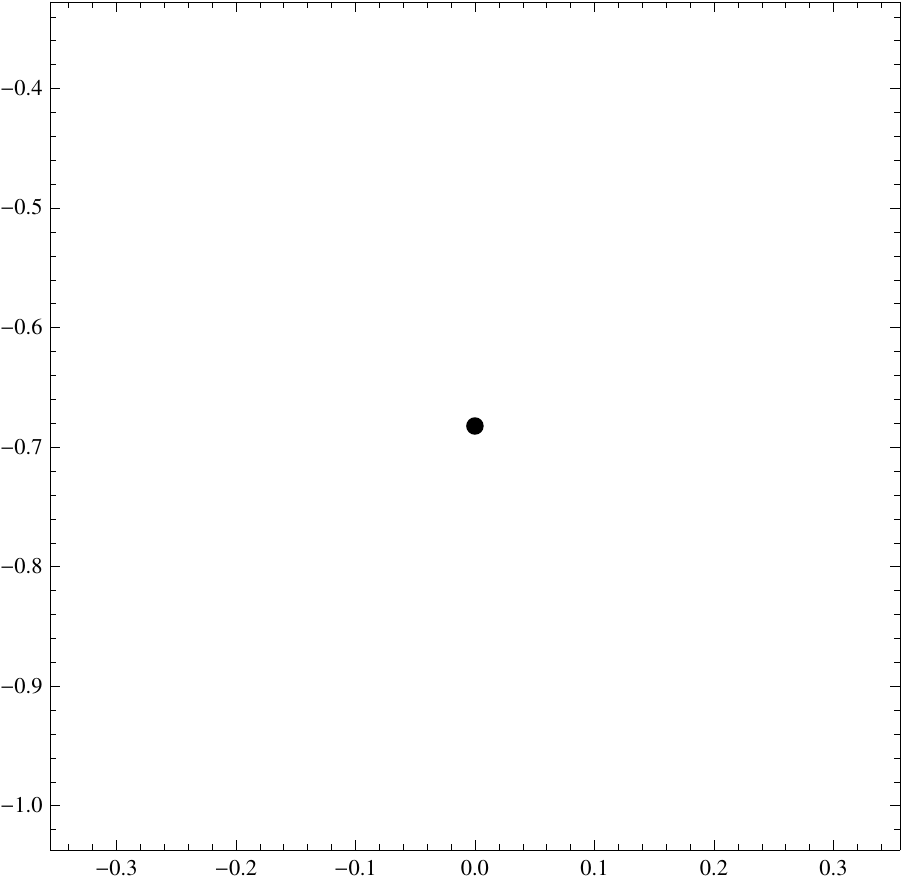}}
\caption{ When $a<0$ and $C=C_1(a)$  the round cylinder of radius $\sqrt{r_1(a)}$ parametrized so that the mean curvature is positive satisfies the rotating drop equation. The first  picture shows polynomial when $a=-2$ and $C=C_1(-2)$, the second picture shows the round cylinder that is a solution of the rotating drop equation for these values of $a$ and $C$ and the third picture shows the level set $G=C$, for this values of $a$ and $C$ the level set reduces to a point.}

\label{poly 3}
\end{figure}

%Start case (iv)

 For values of $(a,C)$ that fall into  case (iv),  the polynomial $q$ has only three roots $x_1<x_2<x_3$ where $x_1$ has multiplicity two and $x_2$ and $x_3$ are simple. The polynomial $q$ is positive for values of $r$ between $x_2$ and $x_3$. In this case, the level set $G=C$ is the union of the point $(0,-\sqrt{x_1})$ and a closed curve. The point $(0,\sqrt{x_1})$ is the TreadmillSled of the profile curve of a circular cylindrical rotating drop, and the closed curve in $G=C$, which can be parametrized by map $\rho$ with  domain the interval  $(x_2,x_3)$, is the TreadmillSled of a rotational drop satisfying that the length of its fundamental piece is finite. Figure \ref{poly 4} shows a typical example of the polynomial and the level set.

  \begin{figure}[ht]
\centerline{\includegraphics[width=4cm,height=3cm]{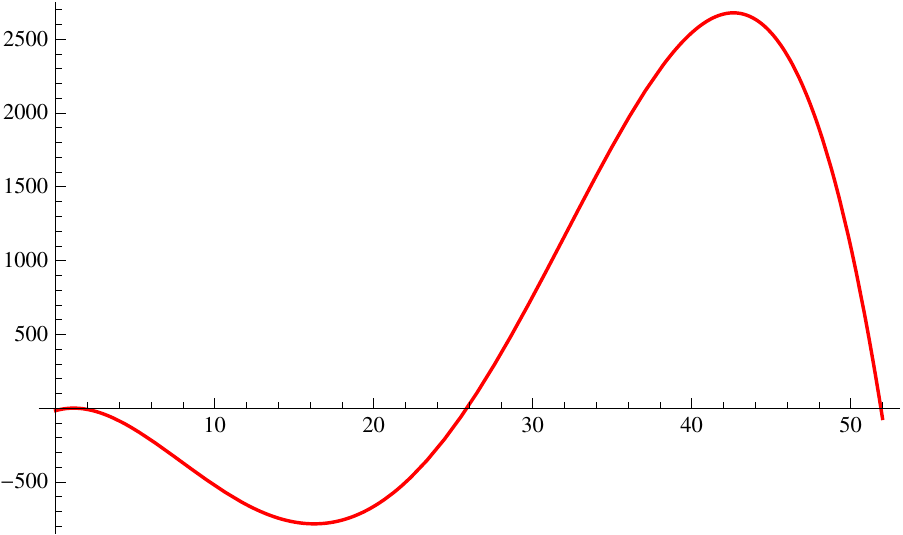}\hskip.2cm \includegraphics[width=4cm,height=4cm]{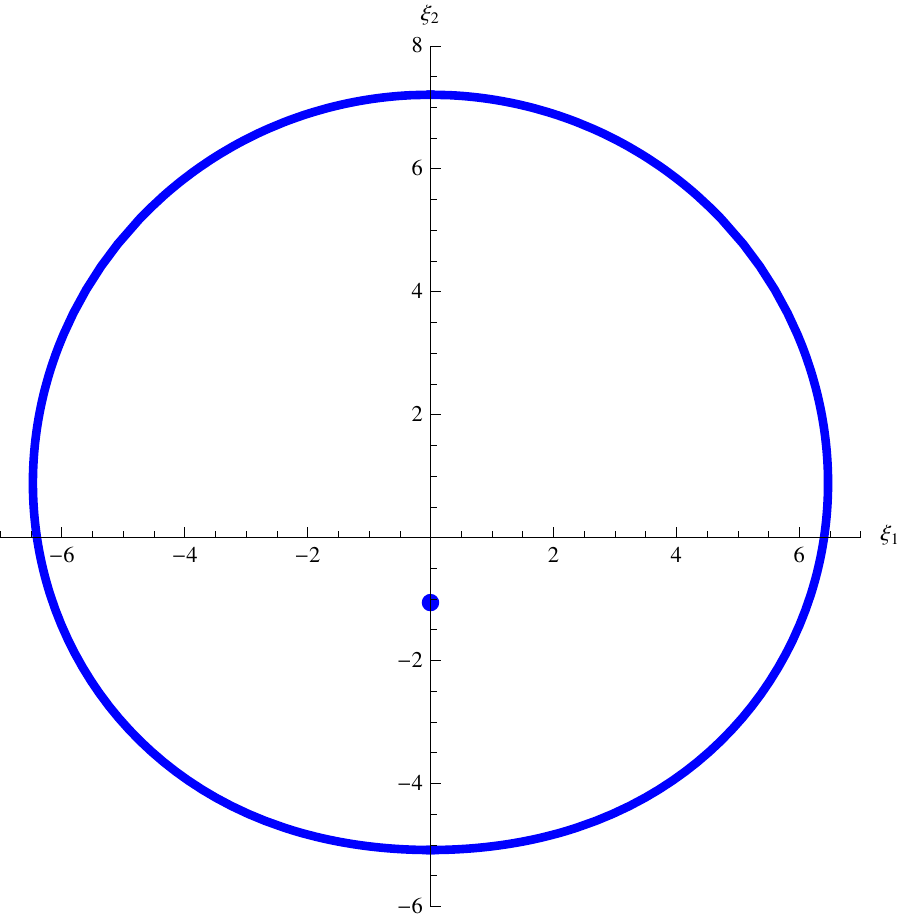}\hskip.2cm\includegraphics[width=3.5cm,height=3.5cm]{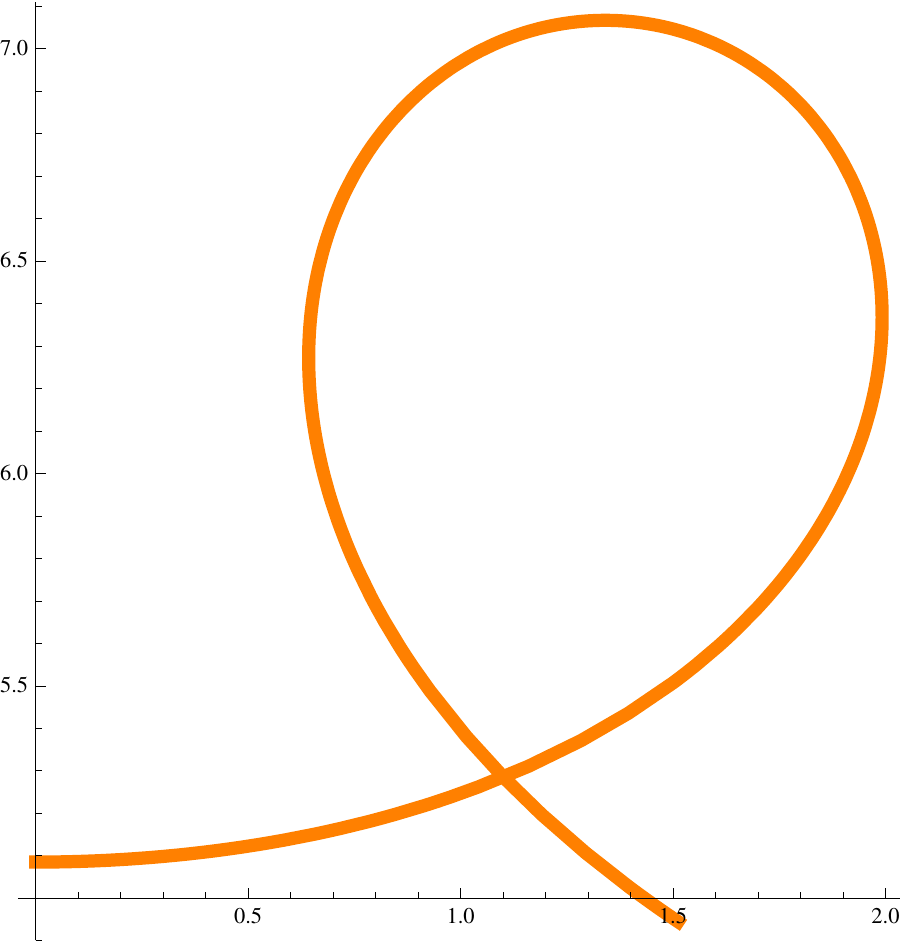},\includegraphics[width=3.5cm,height=3.5cm]{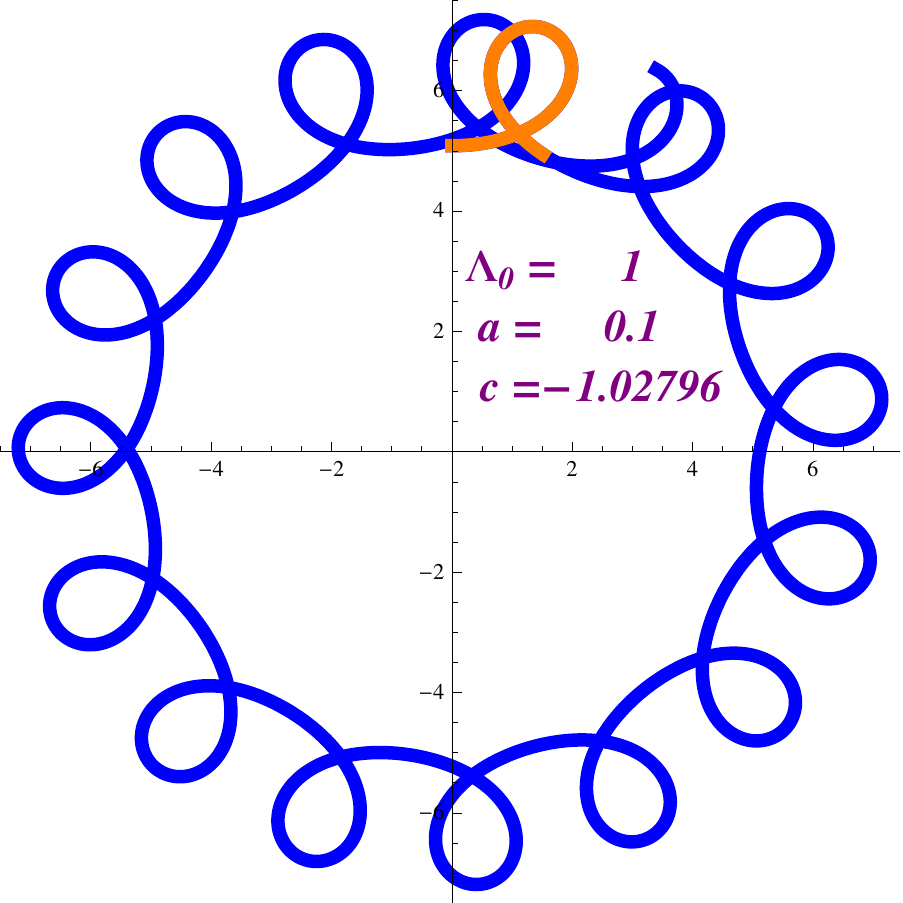}}
\caption{ When $a=0.1$ and $C=C_2(0.1)\approx -1.027962166$ there are two cylindrical drops. For these values of $a$ and $C$, the first image shows the graph of the polynomial $q$, the second image represents the level set $G=C$.  It is the union of a closed curve and a point, this point is the TreadmillSled of a circular cylindrical drop and the closed curve in the set $G=C$ is the TreadmillSled of the cylindrical drop whose profile curve is illustrated in the third and fourth image. The third image represents the fundamental piece of the non-circular rotating drop associated with these values of $a$ and $C$, and the fourth image represents the part of the profile curve made out of $16.5$ copies of the fundamental piece.}

\label{poly 4}
\end{figure}

%End case (iv)

 %Start case (v)

 For values of $(a,C)$ that fall into  case (v),  the  polynomial $q$ has  three distinct  roots $x_1<x_2<x_3$ where $x_2$ has multiplicity two and $x_1$ and $x_3$ are simple. The polynomial $q$ is positive for values of $r$ between $x_1$ and $x_2$ and for values of $r$ between $x_2$ and $x_3$.   Figure \ref{poly 5} shows a typical example of the graph of the polynomial and the level set $G=C$. In this case, the level set $G=C$ is connected but it self-intersects  at the point $(0,-\sqrt{x_2})$. This point alone is the TreadmillSled of a round cylinder.
Any part of a curve that crosses the $\xi_2$-axis horizontally can be the TreadmillSled of a regular curve (see Proposition 2.11 in \cite{P3}) .

 The set $G=C$ minus the point $(0,-\sqrt{x_2})$ has two connected components. One of these connected components can be parametrized using the map $\rho$ with values of $r$ between $x_1$ and $x_2$ and the other using the map $\rho$ with values of $r$ between $x_2$ and $x_3$. Each of these connected components  is the TreamillSled of the whole profile curve  of a rotating cylindrical drop. In this case, these profile curves are not invariant under a group of rotations because the TreadmillSled of any proper subset of the profile curve is a proper subset of one of the connected components of $G=C$ minus $P$. These cylindrical rotating drops are not properly immersed because their length is not bounded due to the fact that these lengths are given by the integrals $\int_{x_1}^{x_2}\frac{8\, dr}{\sqrt{q(r,a)}}$ and $\int_{x_2}^{x_3}\frac{8\, dr}{\sqrt{q(r,a)}}$ , respectively, and, moreover, the distance to the origin function is monotone when $s$ goes to infinity, and it is  bounded. %The point $((0,-\sqrt{x_2}))$ represents a circular cylindrical rotating drop. 
 Figure \ref{pc5} shows part of these two fundamental pieces.

  \begin{figure}[ht]
\centerline{\includegraphics[width=4cm,height=3.5cm]{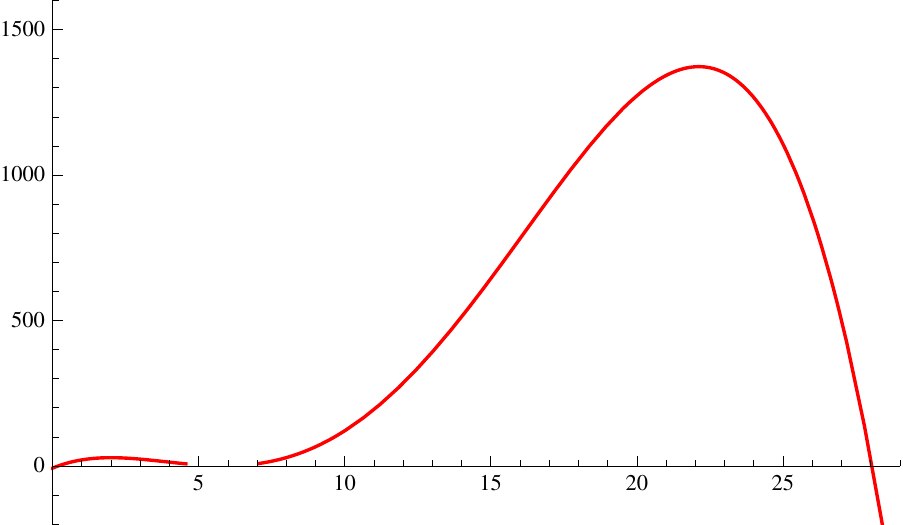}\hskip.6cm \includegraphics[width=4cm,height=4cm]{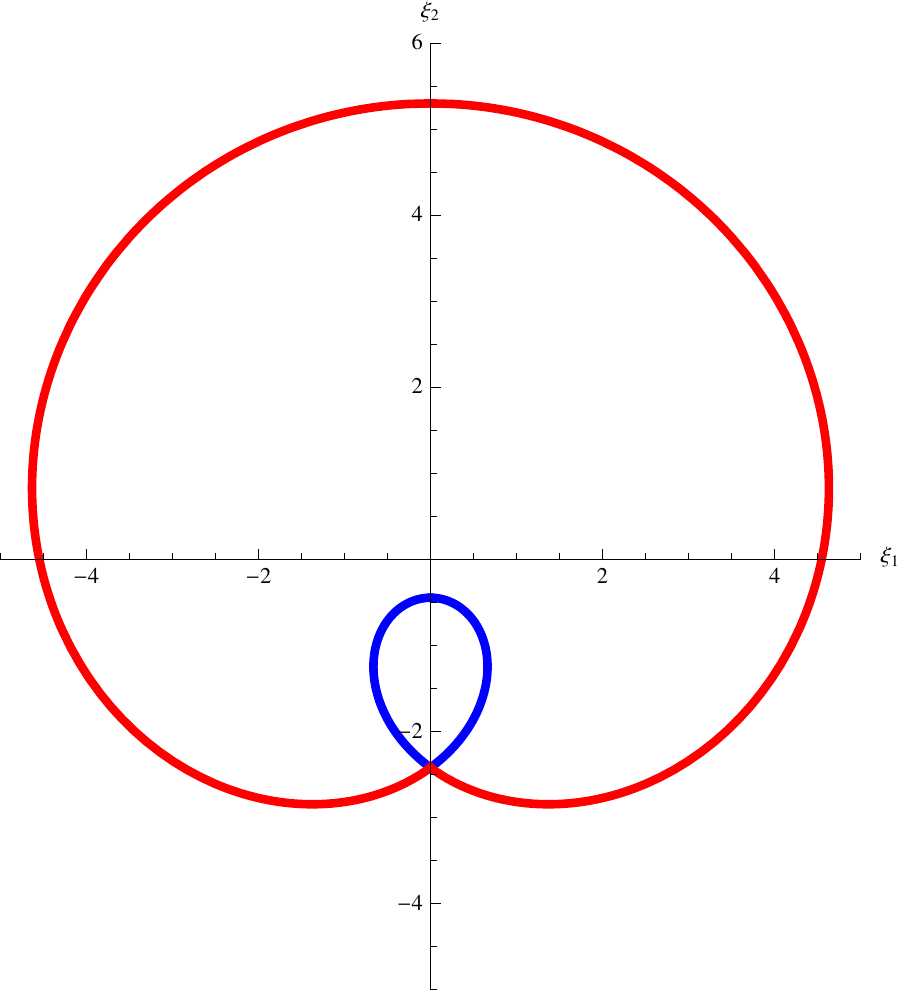}\hskip.2cm}
\caption{ When $a=0.2$ and $C=C_3(0.2)\approx -0.698461$ there are three cylindrical drops. The first  and second images are the graph of the polynomial $q$ and and the level set $G=C$.  Even though the level set is connected, it cannot be the TreadmillSled of a single curve because it is known that if the  TreasmillSled  of a smooth curves intersects the $\xi_2$-axis, then it must do it horizontally. In this case the level set should be viewed as made out of three connected pieces: The point $P$ where the level set $G=C$ self crosses, and the two connected components of $G=C$ minus $P$. Each of these two connected component determine a fundamental piece with not bounded length, and the point $P$ determines a round circle.}
\label{poly 5}
\end{figure}

  \begin{figure}[ht]
\centerline{\includegraphics[width=4cm,height=3.5cm]{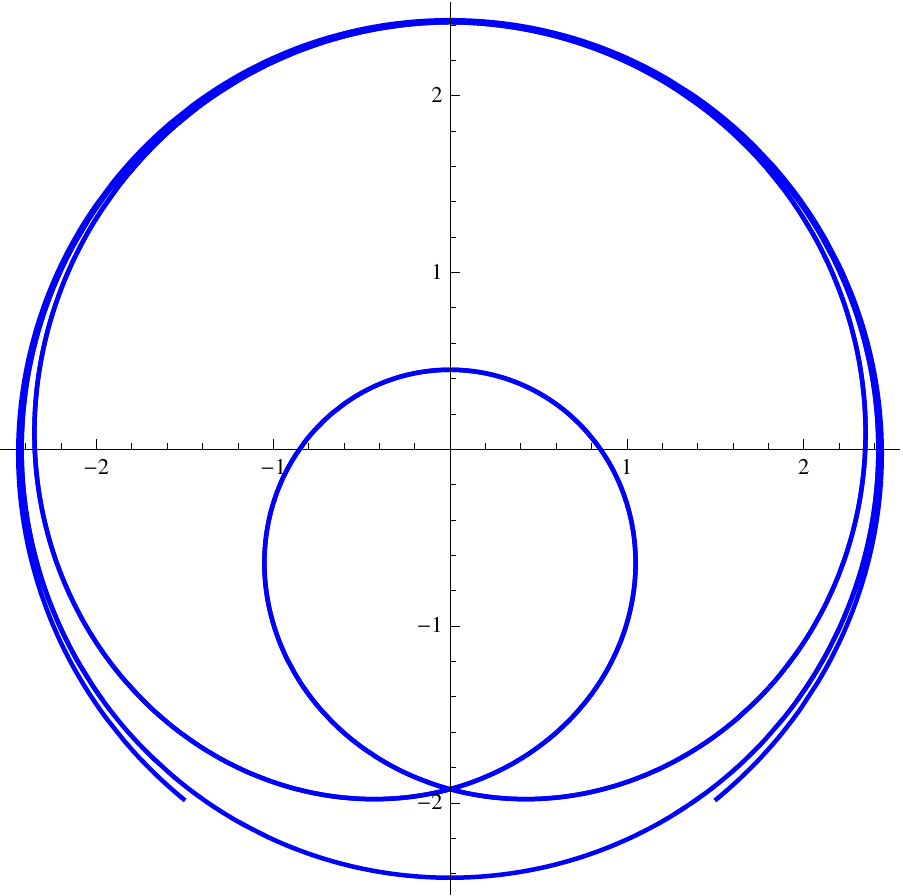}\hskip.2cm \includegraphics[width=4cm,height=4.7cm]{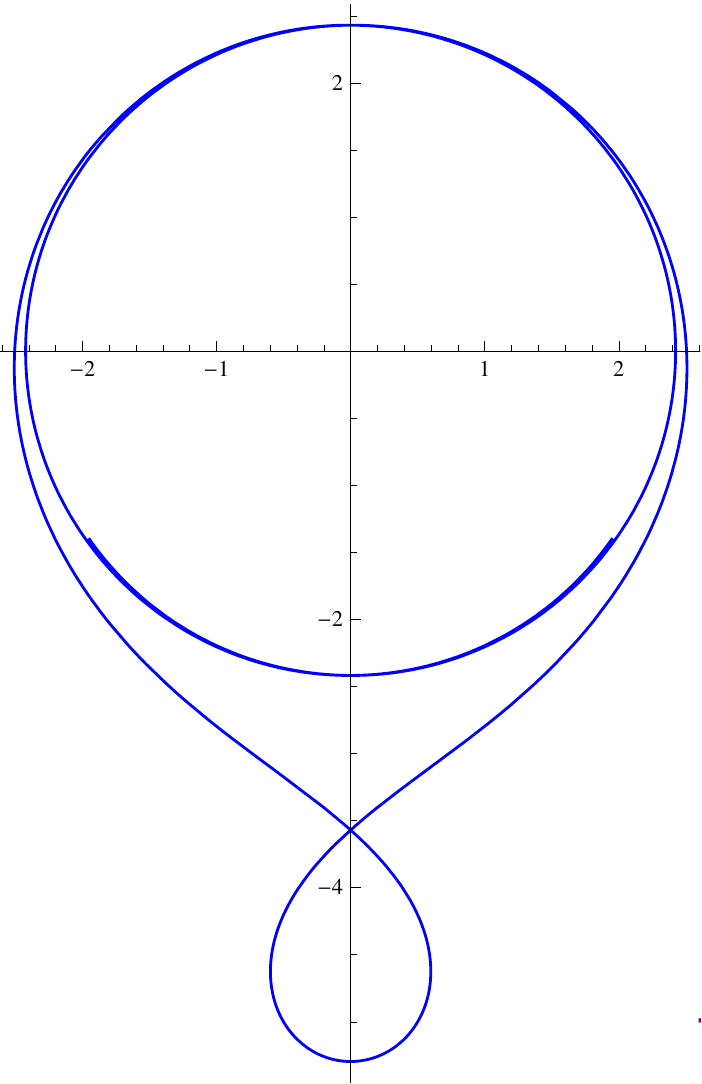}\hskip.2cm}
\caption{ When $a=0.2$ and $C=C_3(0.2)\approx -0.698461$ there are three cylindrical drops. One is circular cylinder and the other two have profile curves with unbounded length that are not invariant under a cyclic group. These last two profile curves have a circle of radius $\sqrt{x_2}\approx 2.42362$ as set of limit points. Moreover these profile curves have infinite winding number with respect to the origin. Recall that the circle of radius $\sqrt{x_2}$  is the profile curve of the other cylindrical drop associated with these values of $a$ and $C$. The first picture shows part of the profile curve which TreadmillSled is parametrized by the map $\rho$ defined with values of $r$ between the first and the second root of $q$. The second picture shows part of the profile curve which TreadmillSled is parametrized by the map $\rho$ defined with values of $r$ between the second and the third root of $q$.}
\label{pc5}
\end{figure}
%end case (v)

 %Start case (vi)

 For $(a,C)=(8/27,-9/8)$, this is, for case (vi),  the the polynomial $q$ has only two roots $x_1<x_2$ where $x_1=\frac{9}{4}$ has multiplicity three and $x_2=\frac{81}{4}$ is simple. The polynomial $q$ is positive for values of $r$ between $x_1$ and $x_2$. Figure \ref{poly 6} shows the graph of the polynomial and the level set $G=C$. In this case the level set $G=C$ is connected but it has a singularity  at the point $(0,-3/2)$.  We can check that the point $(0, -3/2)$ corresponds to a  circular cylindrical rotating drop with radius $3/2$. The set $G=C$ minus the point $(0,-3/2)$ is connected and can be parametrized using the map $\rho$ with values of $r$ between $9/4$ and $81/4$. This part of the set $G=C$ is the TreamillSled of the profile curve of a rotating cylindrical drop whose length is not bounded.

  \begin{figure}[ht]
\centerline{\includegraphics[width=4cm,height=3.5cm]{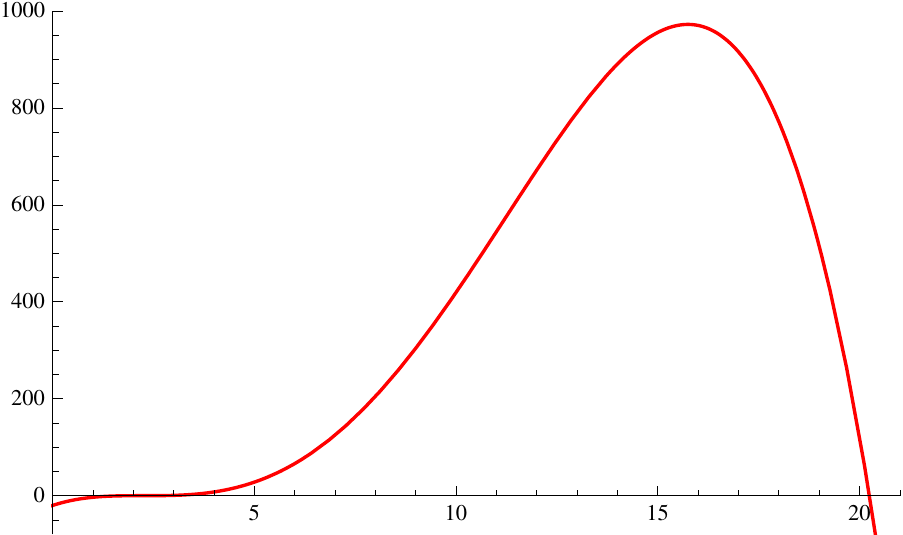}\hskip.6cm \includegraphics[width=4cm,height=4cm]{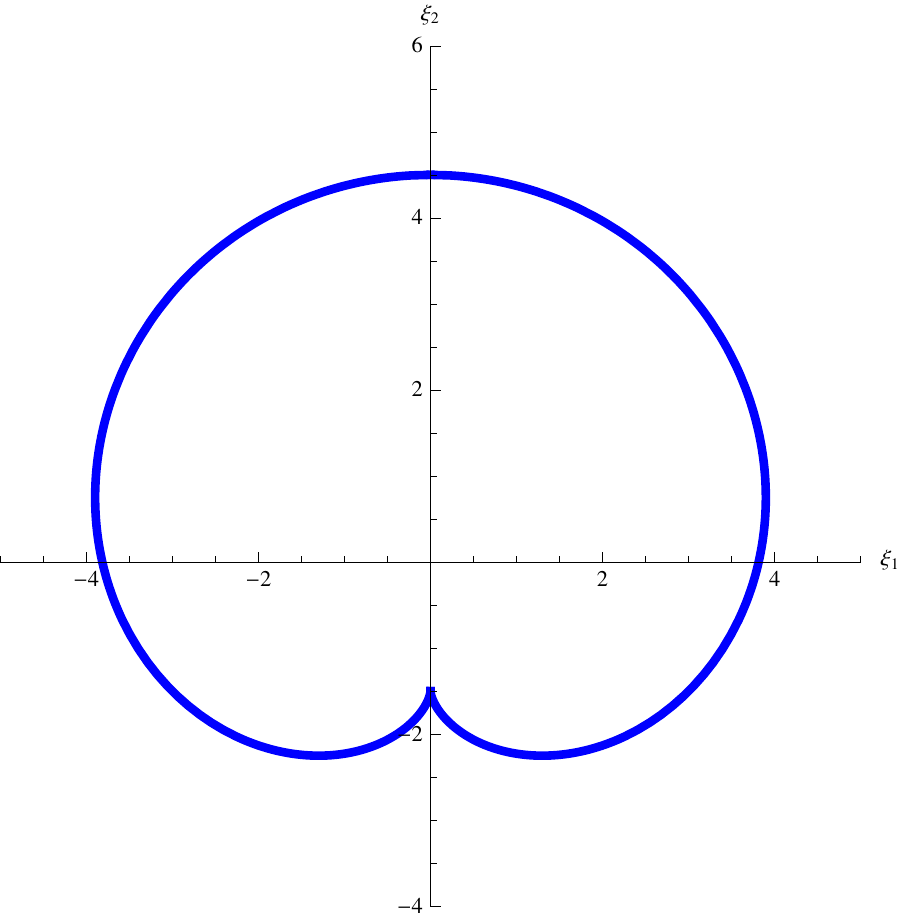}\hskip.2cm \includegraphics[width=4cm,height=4cm]{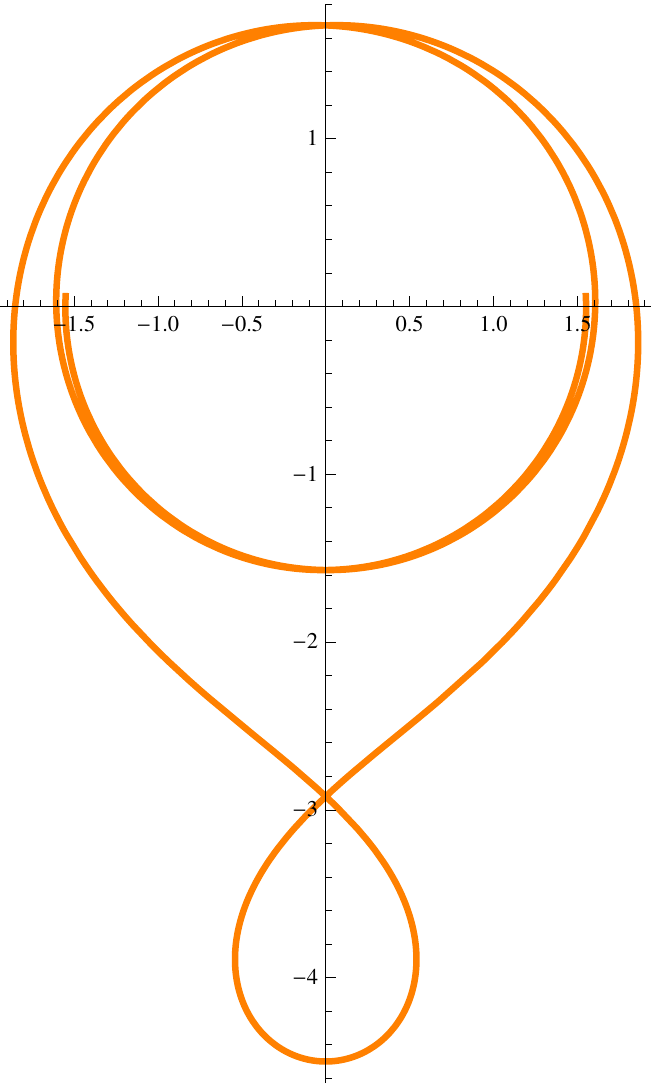}}
\caption{ When $a=\frac{8}{27}$ and $C=-\frac{9}{8}$ there are two cylindrical rotating drops, one is a round cylinder with radius $\frac{3}{2}$ and the  other is a non-circular cylinder . The first  and second images are the graph of the polynomial $q$ and and the level set $G=C$.  Even though the level set is connected, it cannot be the TreadmillSled of a single curve because it is known that the  TreasmillSled  of a smooth curve must intercept the $\xi_2$-axis horizontally. In this case the level set should be viewed as made out of two connected pieces: The $\xi_2$ intercept $P=(0,-\frac{3}{2})$ and the set $G=C$ minus $P$. The point $P$ determines a round circle and the  set  $G=C$ minus $P$ determines a fundamental piece with not bounded length. Part of this fundamental piece is shown in the third picture of this figure. }
\label{poly 6}
\end{figure}

%end case (vi)

%Start case (vii)

Since we know that the profile curve of every rotating cylindrical drop satisfies the integral equation $G=C$ and cases (i)-(vi) cover all the possible level sets for the level sets of $G$ then every rotating cylindrical drop fall into one of the fist 6 cases of this proposition. This proves (vii). 
%End case (vii)

%Start case (viii)
In order to prove (viii) we notice that when a cylindrical rotating drop is not exceptional, has a fundamental piece with finite length whose TreadmillSled is a closed regular curve (a connected component of the set $G=C$). By the properties of the operator TreadmillSled we have the the whole profile curve is a union of rotations of the fundamental piece. The angle of rotation is given by $\Delta{\tilde \theta}=\Delta{\tilde \theta}(C,a,x_1,x_2)$, where $\rho:[x_1,x_2]\rightarrow \bfR{2}$ parametrizes the TreadmillSled of the profile curve. We therefore have that the profile curve is invariant under the group of rotations

\begin{eqnarray}
\label{the group}
&&\\
&&\mathbb{G}=\{(x_1,x_2)\rightarrow (\cos(n \Delta{\tilde \theta}) x_1 +  \sin(n \Delta{\tilde \theta})) x_2 , -\sin(n \Delta{\tilde \theta})) x_1+  \cos(n \Delta{\tilde \theta})) x_2 ): n\in \mathbb{Z}\} \:.\nonumber
\end{eqnarray}

It is not difficult to show that if $\Delta \tilde \theta/\pi$ is a rational number, then the group $\mathbb{G}$ is finite and the cylinder is properly immerse. Moreover, if $\Delta \tilde \theta/\pi$ is not a rational number, then the group $\mathbb{G}$ is not finite and the cylinder is dense in the region bounded by the two cylinders of radius $\sqrt{r_1}$ and $\sqrt{r_2}$. A more detailed explanation of this last statement can be found in \cite{P1}.

%End case (viii)

\end{proof}

\begin{figure}[ht]
\centerline{\includegraphics[width=9cm,height=5.5cm]{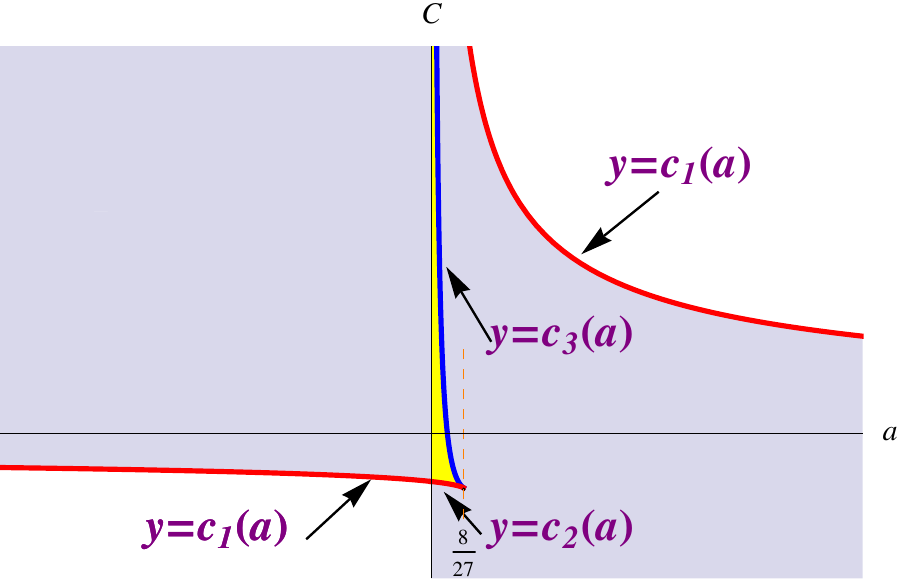}}
\label{c graphs 3}
\caption{ Every  point in the gray region represents a rotating cylindrical drop. Every point in the yellow region represents two rotating cylindrical drops. Every point in either the III or I quadrant on any of the two red curves represents a rotating cylindrical drop. Points along the graph $y=C_2(a)$ represent two surfaces, one of them is a circular cylinder. Every point in the blue curve represents three rotating cylindrical drops, one circular cylinder and two more that differ from any other rotating cylindrical drop in the sense that profile curve any of these two surfaces is not invariant under a group of rotations. }

\end{figure}

\subsection{Embedded and Properly Embedded Examples}

In this subsection we will find some embedded examples and we will display their profile curves. As pointed out before, when the cylindrical drop is not exceptional, its profile curve is 
a union of rotations of fundamental pieces that ends up being invariant under the group $\mathbb{G}$ define in \ref{the group}. It is not difficult to see that a necessary condition for the cylindrical drop to be embedded is that $\Delta {\tilde \theta}=2 \pi/m$ for some integer $m$. We will show that this condition is not sufficient. In order to obtain this potentially embedded example we need to understand the function $\Delta {\tilde \theta}(C,a,x_1,x_2)$.  We will apply the Intermediate Value Theorem in order to solve the equation $$\Delta {\tilde \theta}=2 \pi/m\:.$$ We know that for any $a$ there is a first (or last) value of $C$, which we denote by $C_0(a,x_1,x_2)$, for which the function $\Delta {\tilde \theta}$ is define. We will compute the limit of $\Delta {\tilde \theta}$ when $C$ goes to $C_0$ using Lemma \ref{lemma 1}. Before continuing we will show images of some of the embedded examples. These graphs were generated by the software Mathematica 8 by solving the system of ordinary differential equations (see equation (\ref{dtheta}))

$$x^\prime(t)=\cos(\theta(\theta)), \quad y^\prime(t)=\sin(\theta(\theta)), \quad \theta^\prime(t)=-\Lambda_0+\frac{a}{2}\, (x(t)^2+y(t)^2), \quad$$

 \begin{figure}[h!]
\centerline{\includegraphics[width=3.5cm,height=3.5cm]{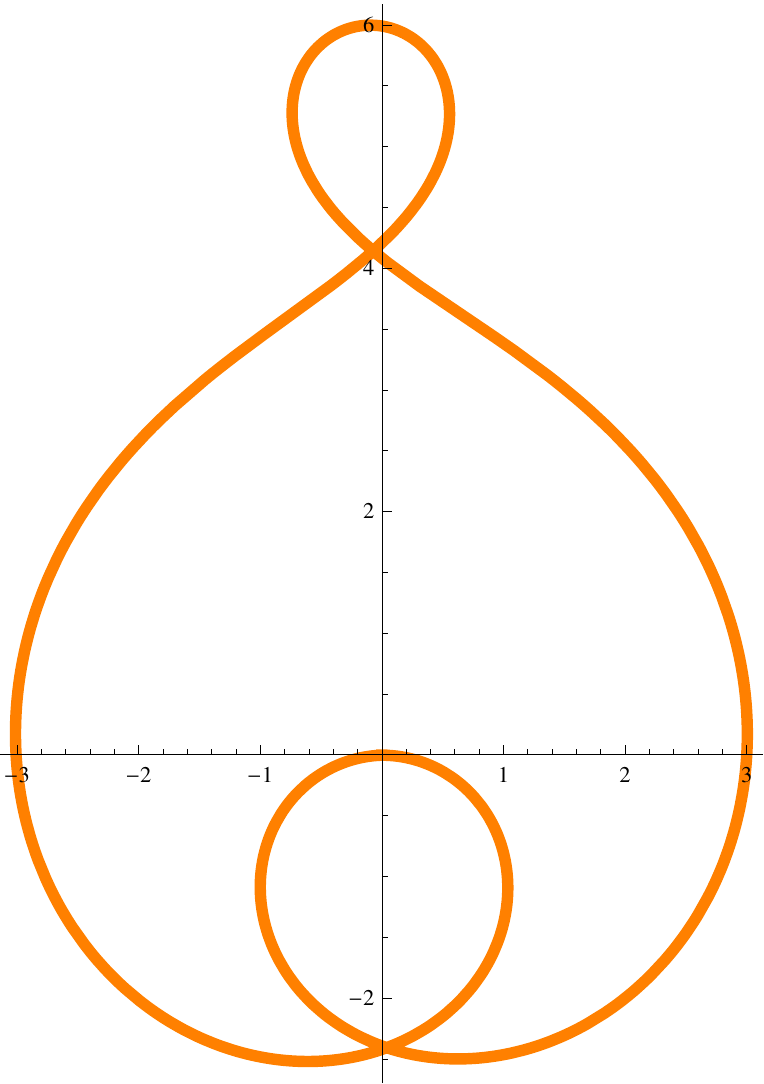}\hskip-1cm \includegraphics[width=3.5cm,height=1.7cm]{emb2-eps-converted-to.pdf}\hskip.2cm\includegraphics[width=3.5cm,height=3.5cm]{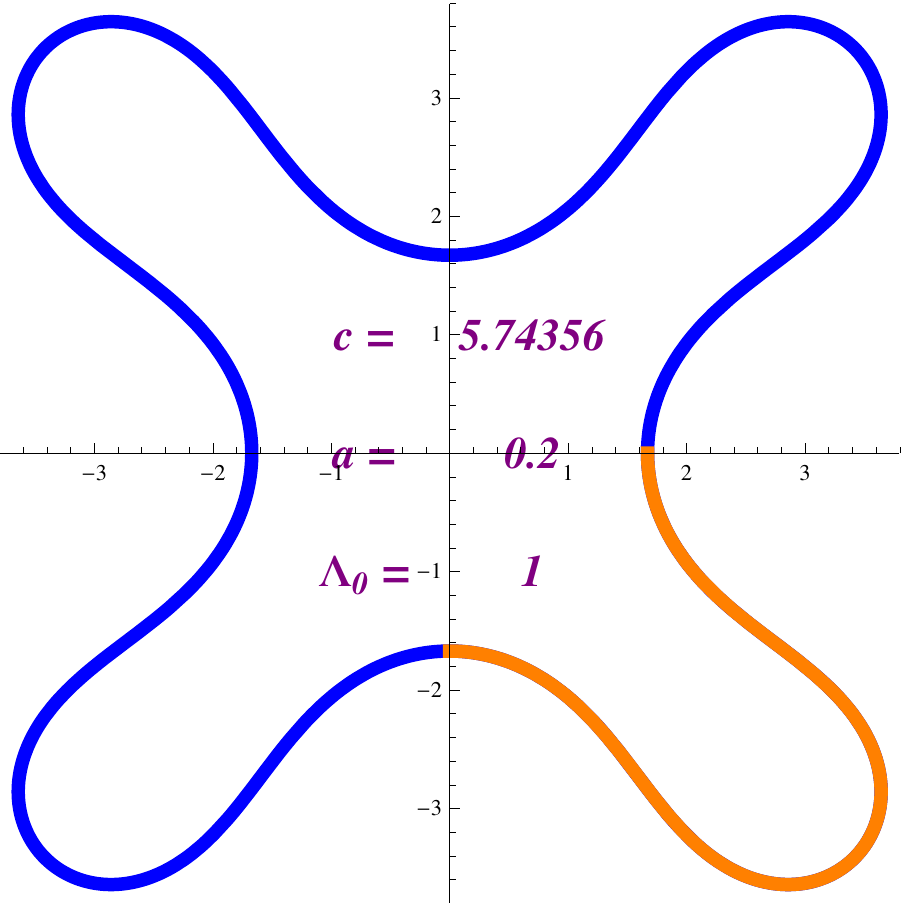}\hskip2cm\includegraphics[width=3.5cm,height=3.5cm]{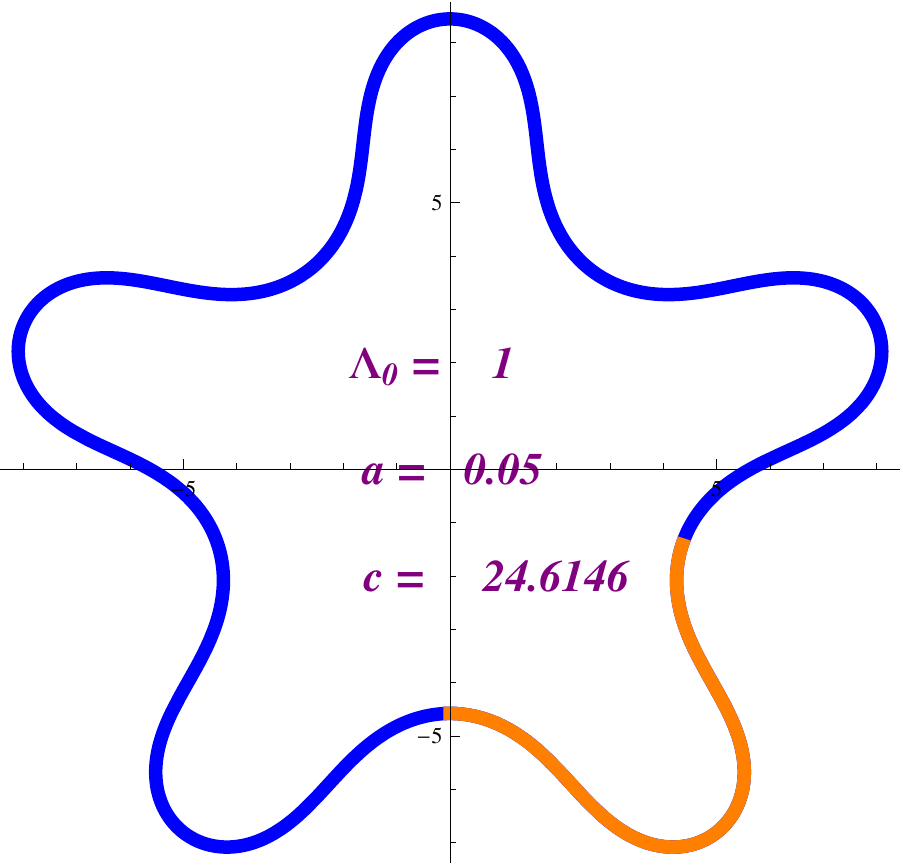}}
\caption{The first example is clearly not embedded but it satisfies  $\Delta{\tilde \theta}=2 \pi$. }
\label{embg1}
\end{figure}

 \begin{figure}[h!]
 \label{G1}
\centerline{\includegraphics[width=3.5cm,height=3.5cm]{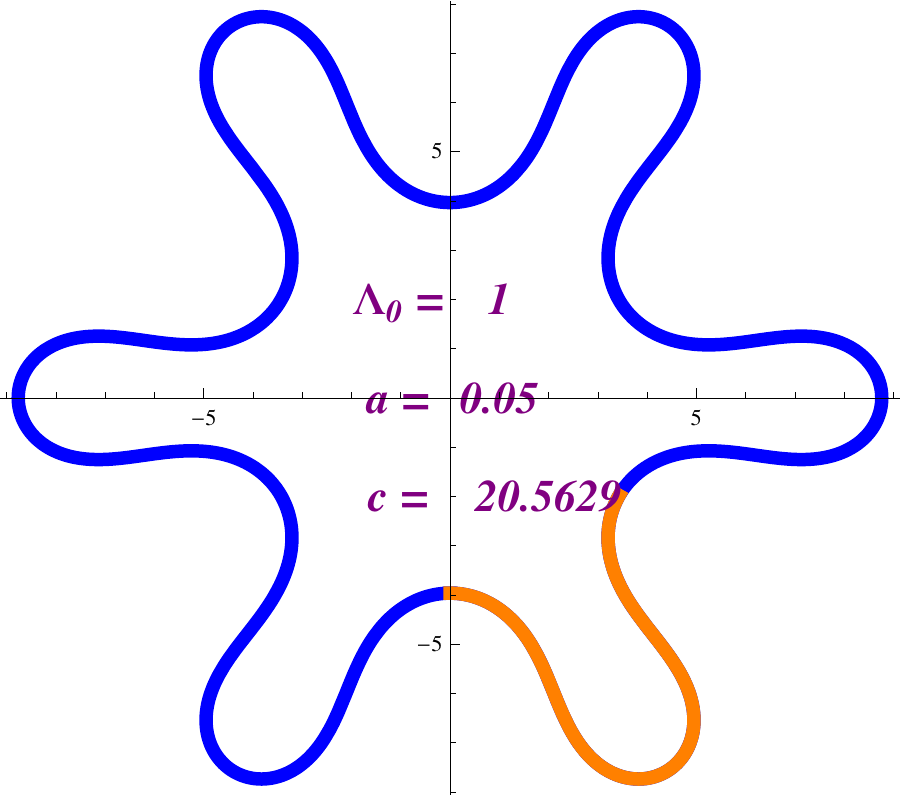}\hskip.2cm \includegraphics[width=3.5cm,height=3.5cm]{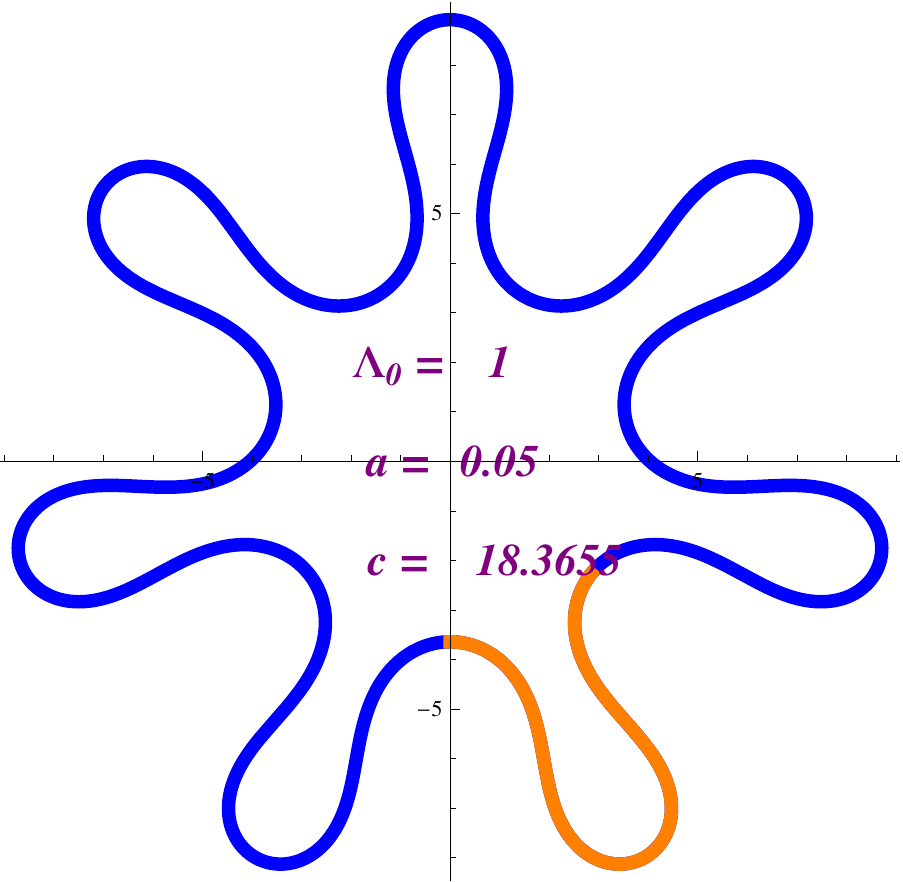}\hskip.2cm\includegraphics[width=3.5cm,height=3.5cm]{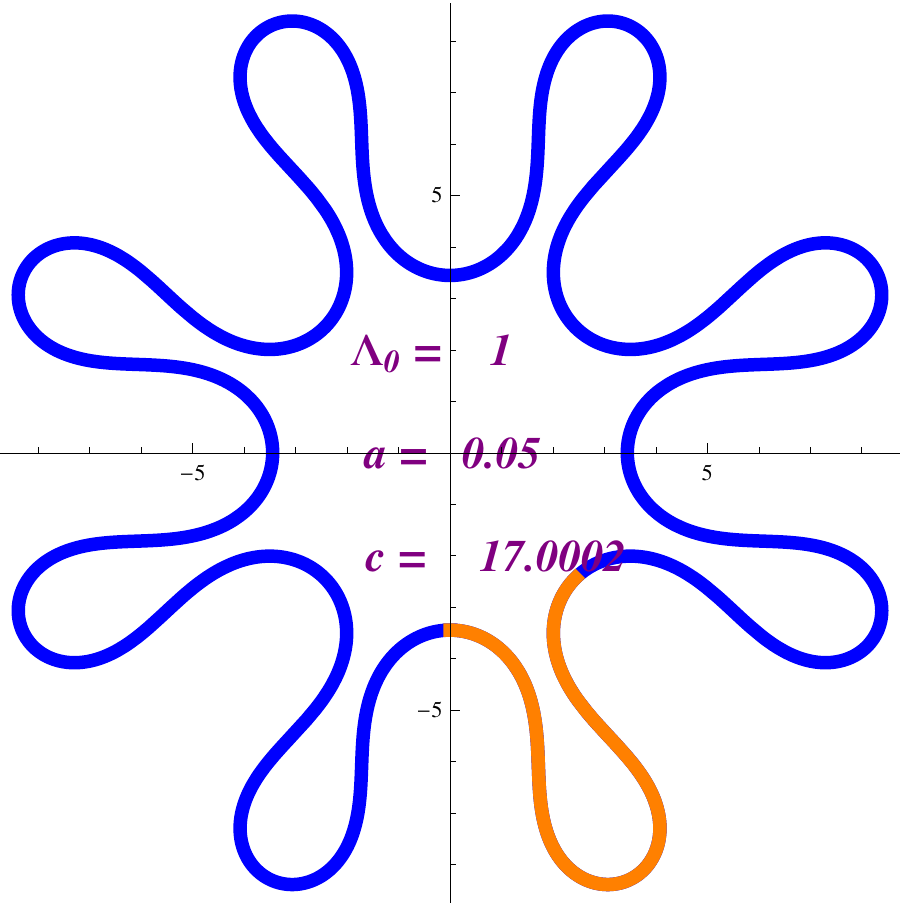}\hskip.2cm\includegraphics[width=3.5cm,height=3.5cm]{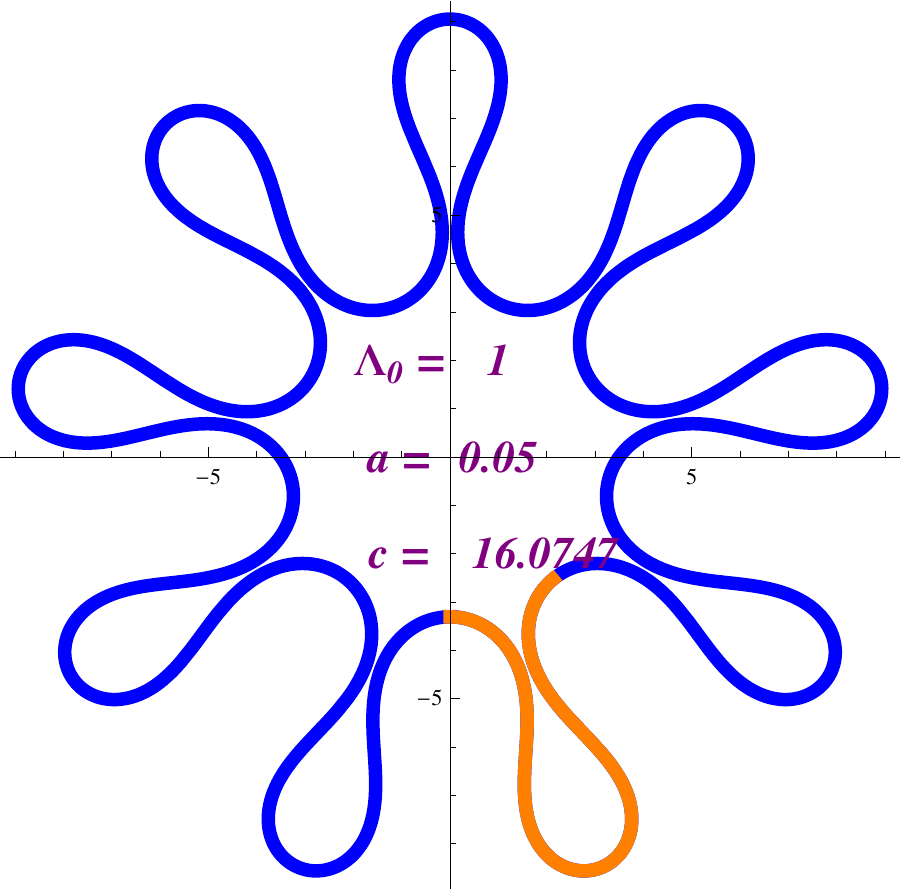}}
\caption{ }
\label{embg2}
\end{figure}
\FloatBarrier

 \begin{figure}[ht]
\centerline{\includegraphics[width=3.5cm,height=3.5cm]{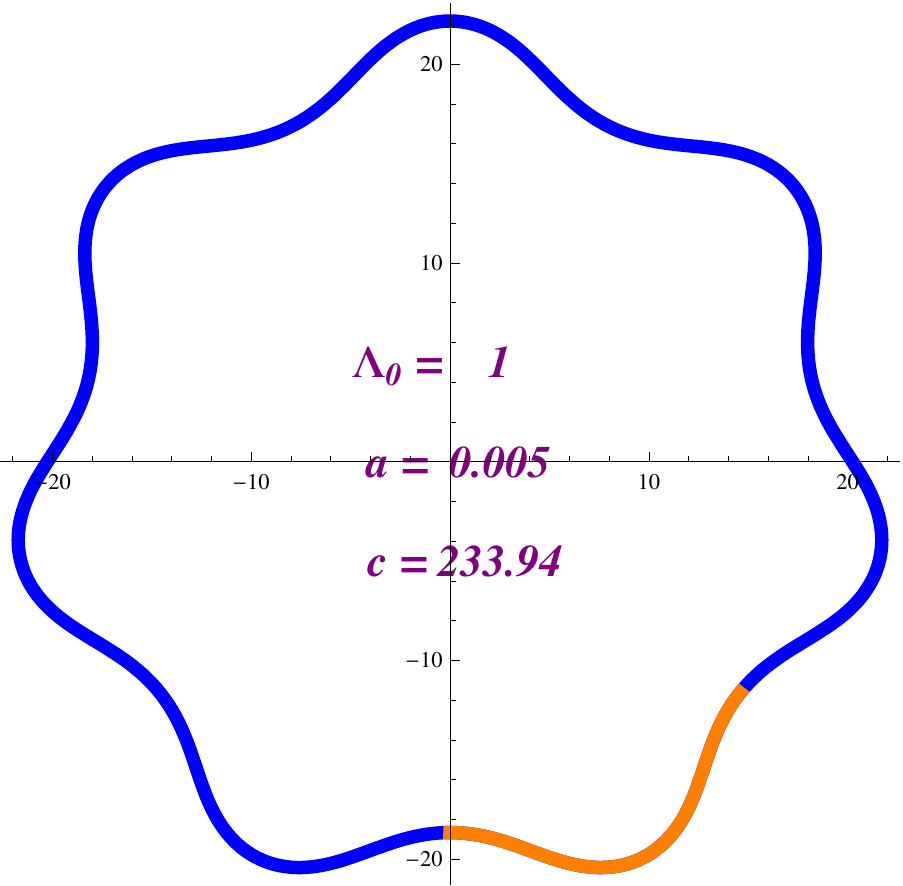}\hskip.2cm \includegraphics[width=3.5cm,height=3.5cm]{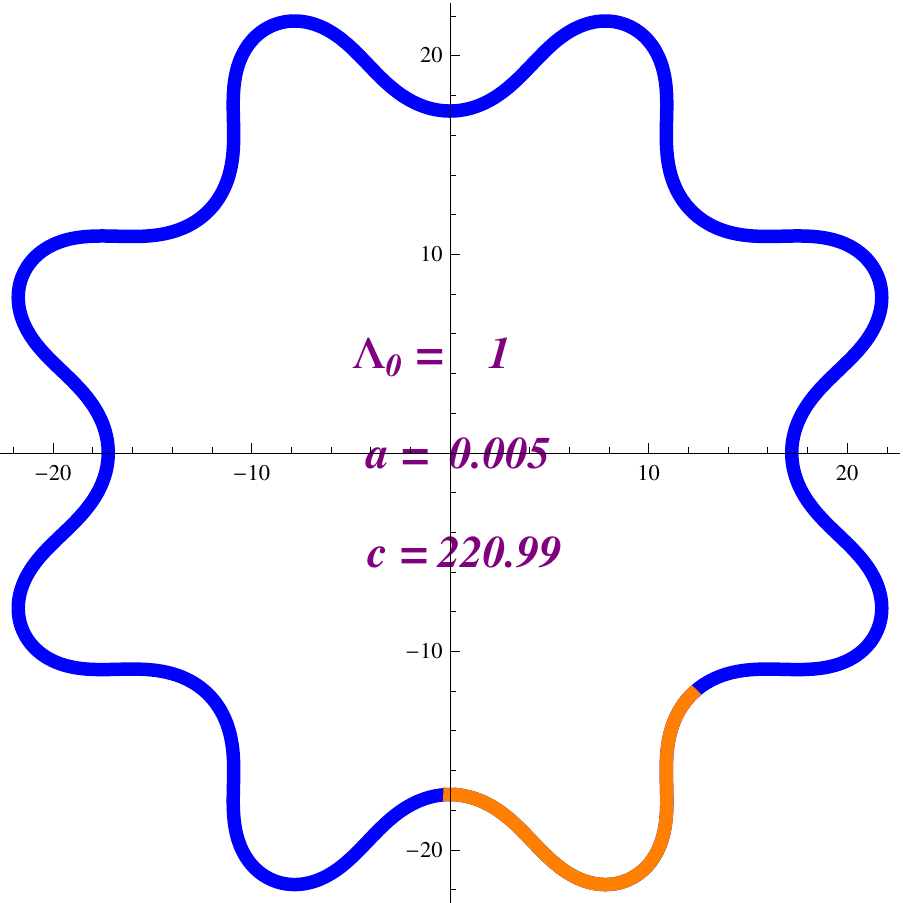}\hskip.2cm\includegraphics[width=3.5cm,height=3.5cm]{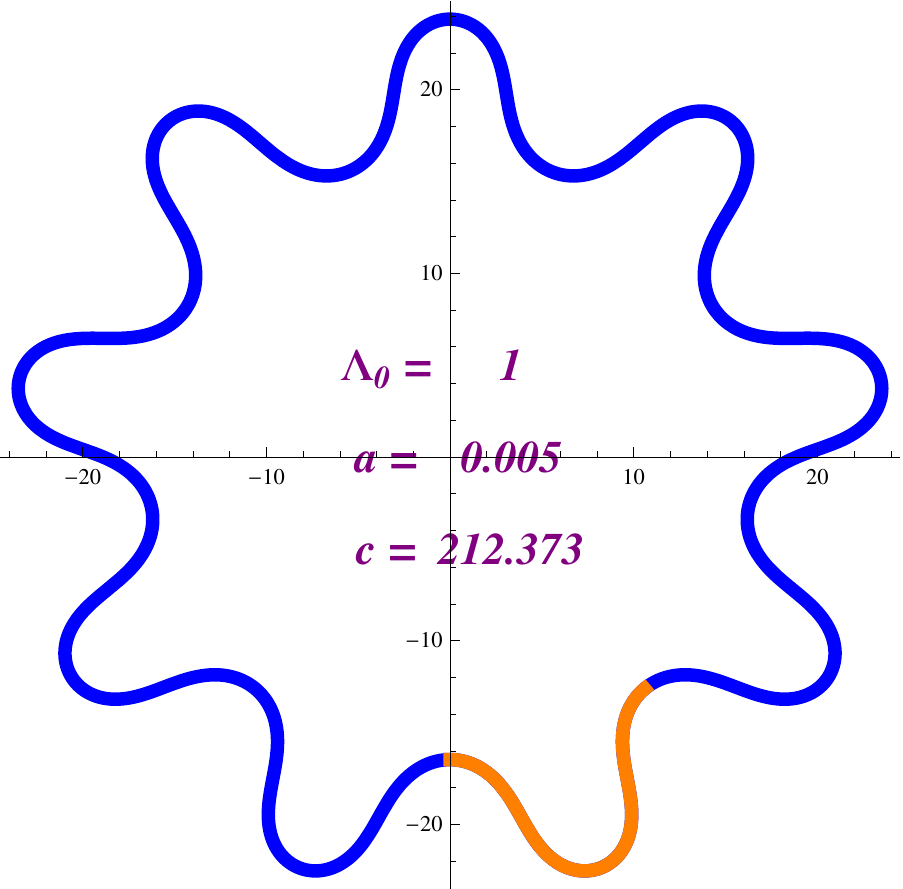}\hskip.2cm\includegraphics[width=3.5cm,height=3.5cm]{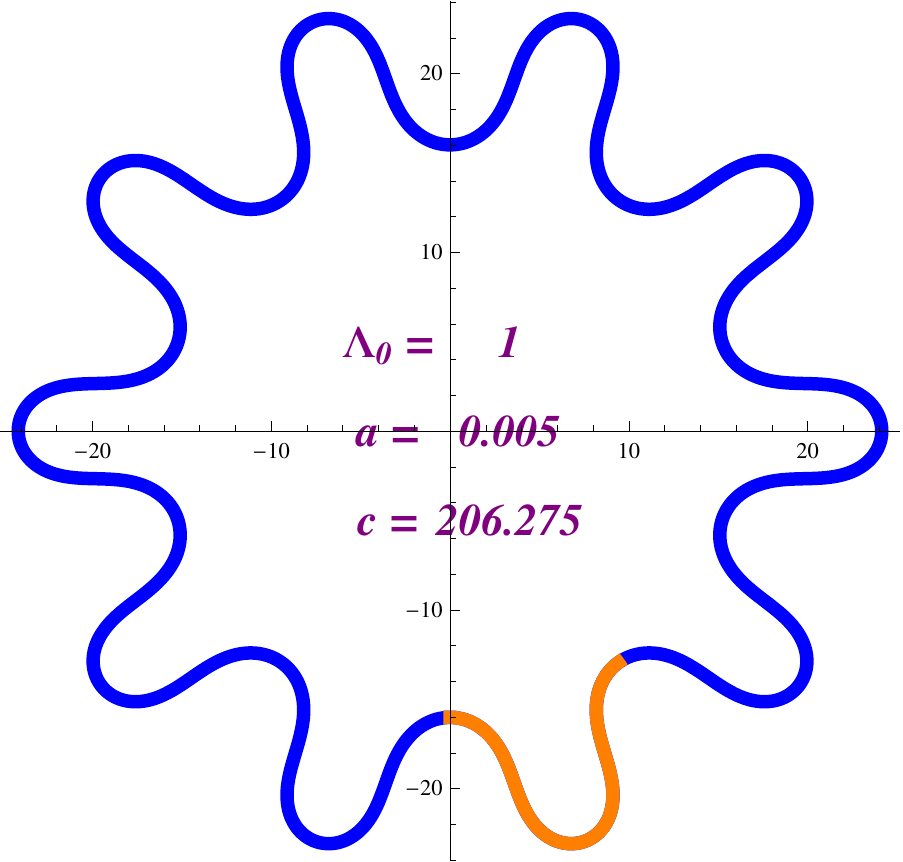}}
\caption{ }
\label{embg3}
\end{figure}

 \begin{figure}[ht]
\centerline{\includegraphics[width=3.5cm,height=3.5cm]{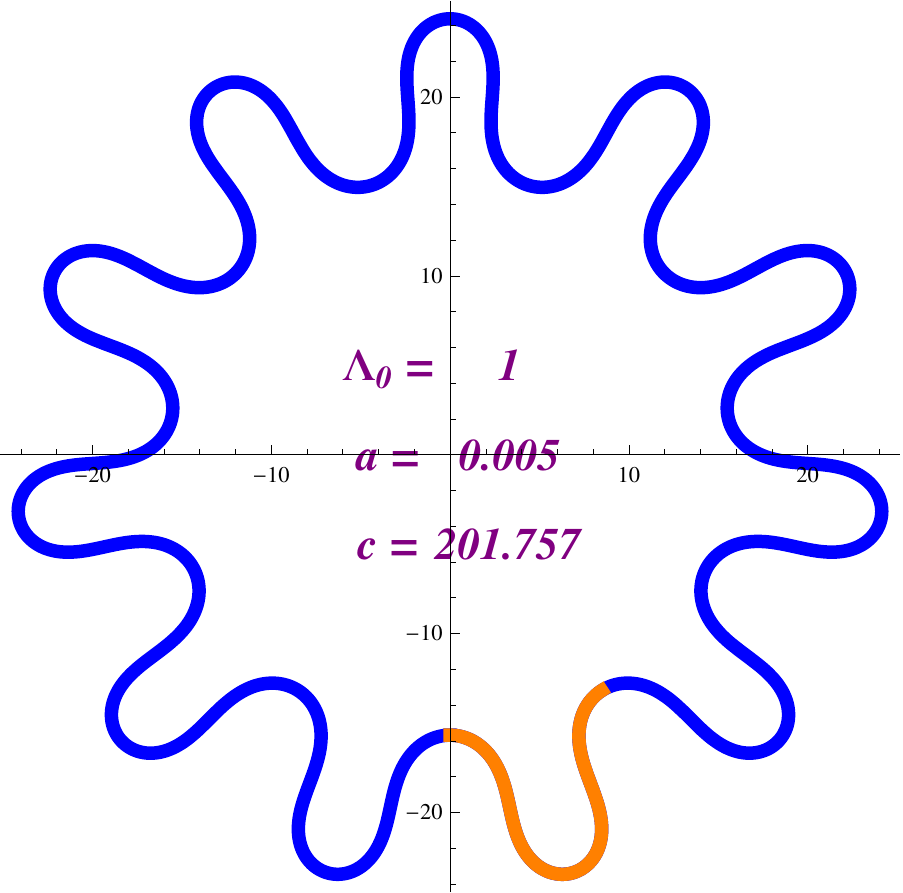}\hskip.2cm \includegraphics[width=3.5cm,height=3.5cm]{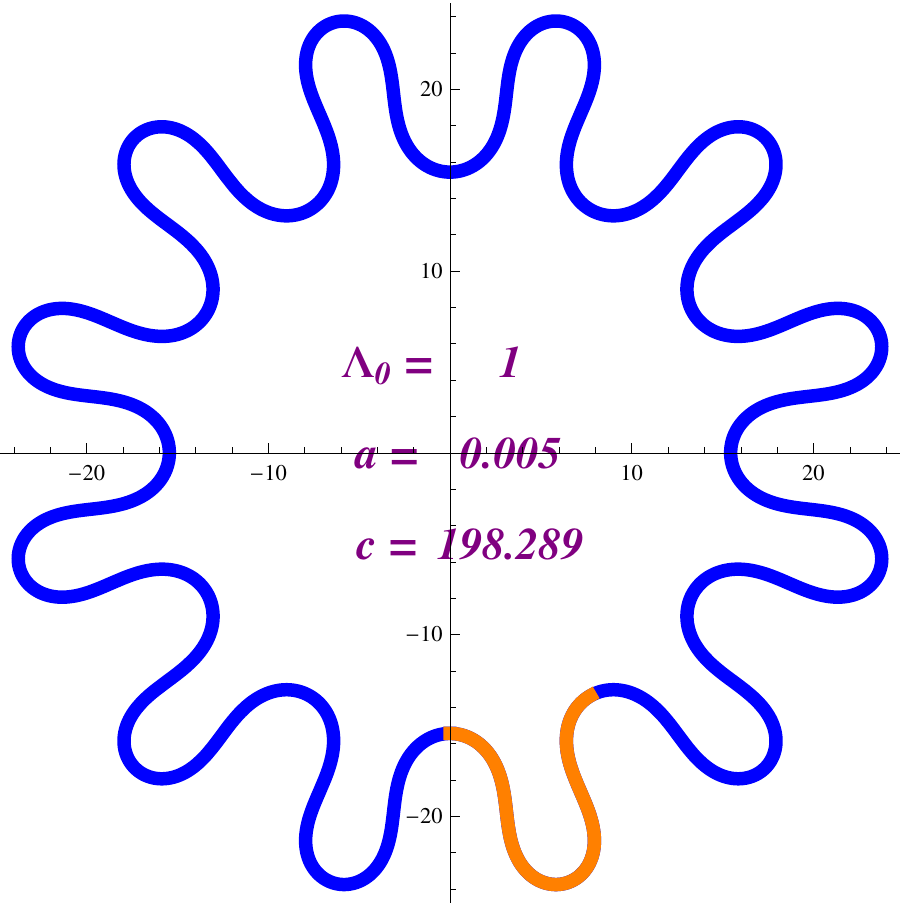}\hskip.2cm\includegraphics[width=3.5cm,height=3.5cm]{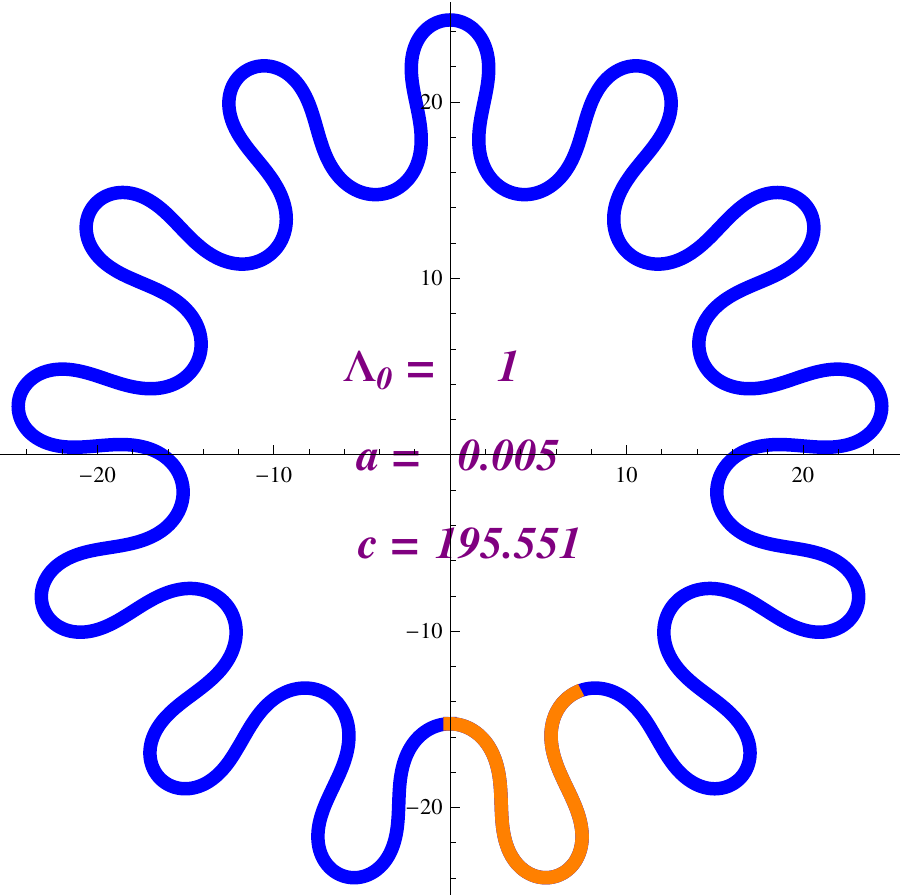}\hskip.2cm\includegraphics[width=3.5cm,height=3.5cm]{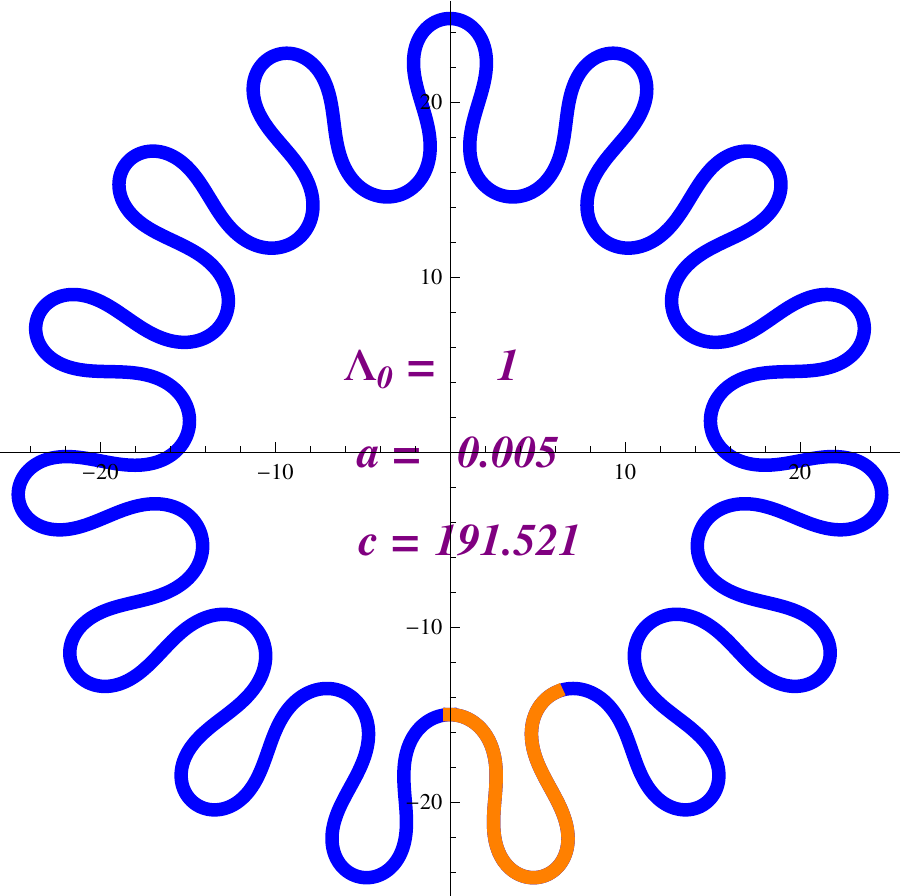}}
\caption{ }
\label{embg4}
\end{figure}

 \begin{figure}[ht]
\centerline{\includegraphics[width=3.5cm,height=3.5cm]{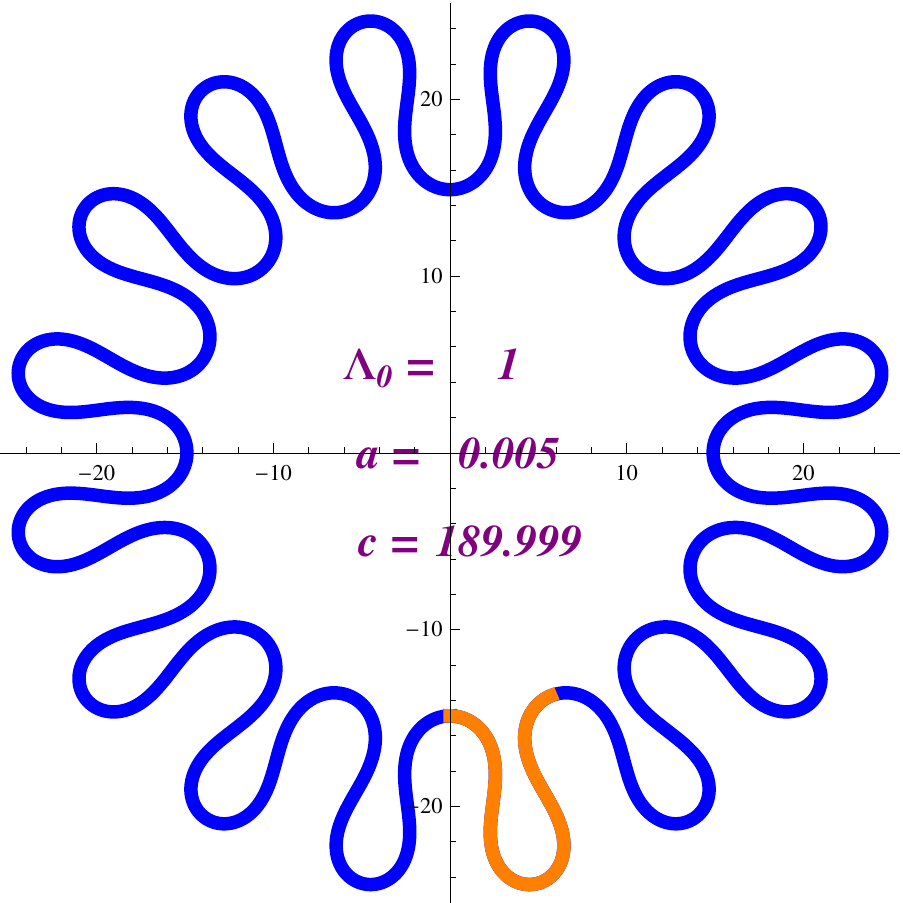}\hskip.2cm \includegraphics[width=3.5cm,height=3.5cm]{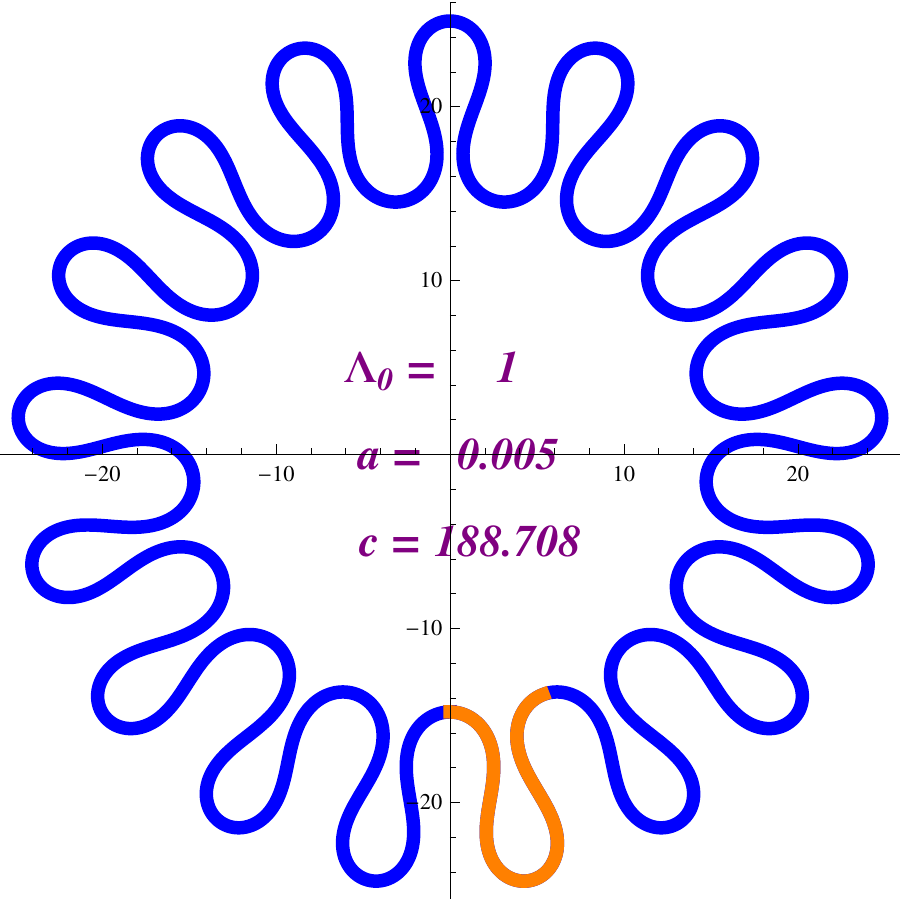}\hskip.2cm\includegraphics[width=3.5cm,height=3.5cm]{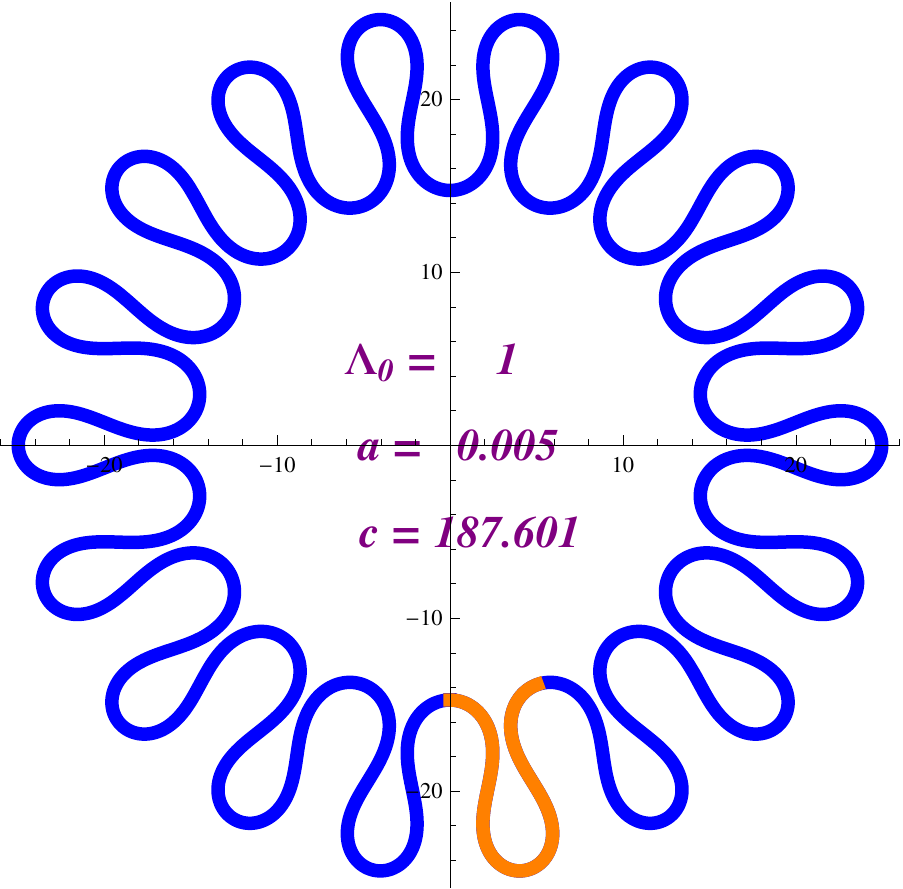}\hskip.2cm\includegraphics[width=3.5cm,height=3.5cm]{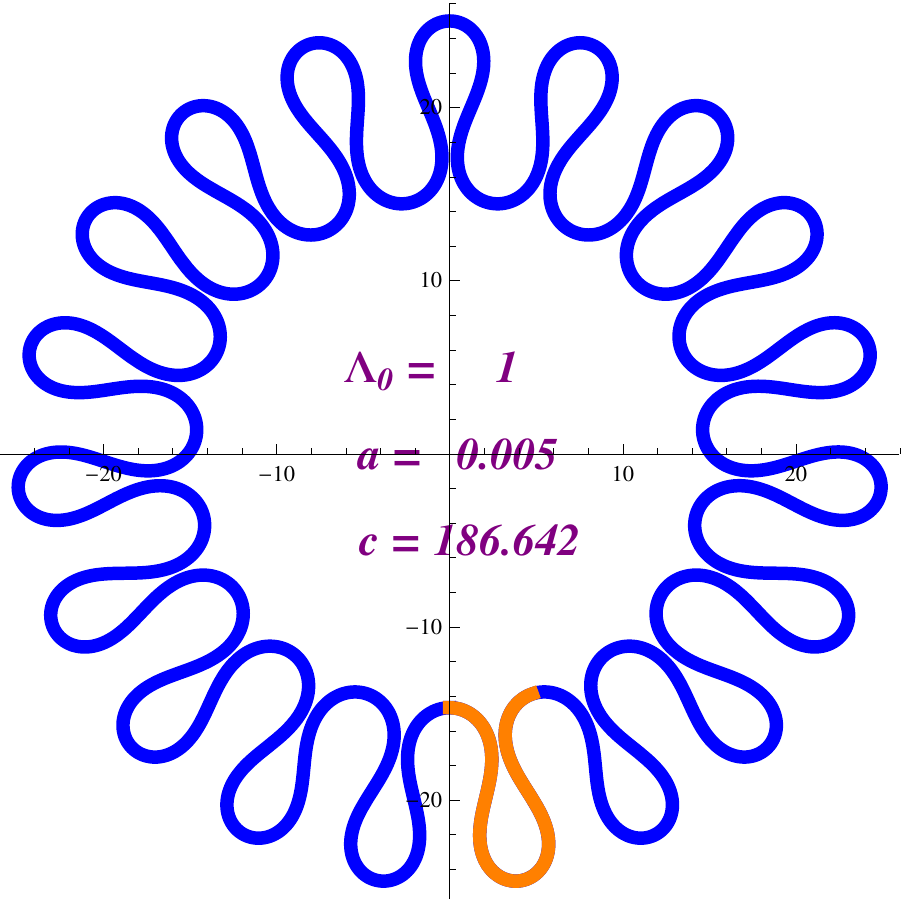}}
\caption{ }
\label{embg5}
\end{figure}

 \begin{figure}[h!]
\centerline{\includegraphics[width=3.5cm,height=3.5cm]{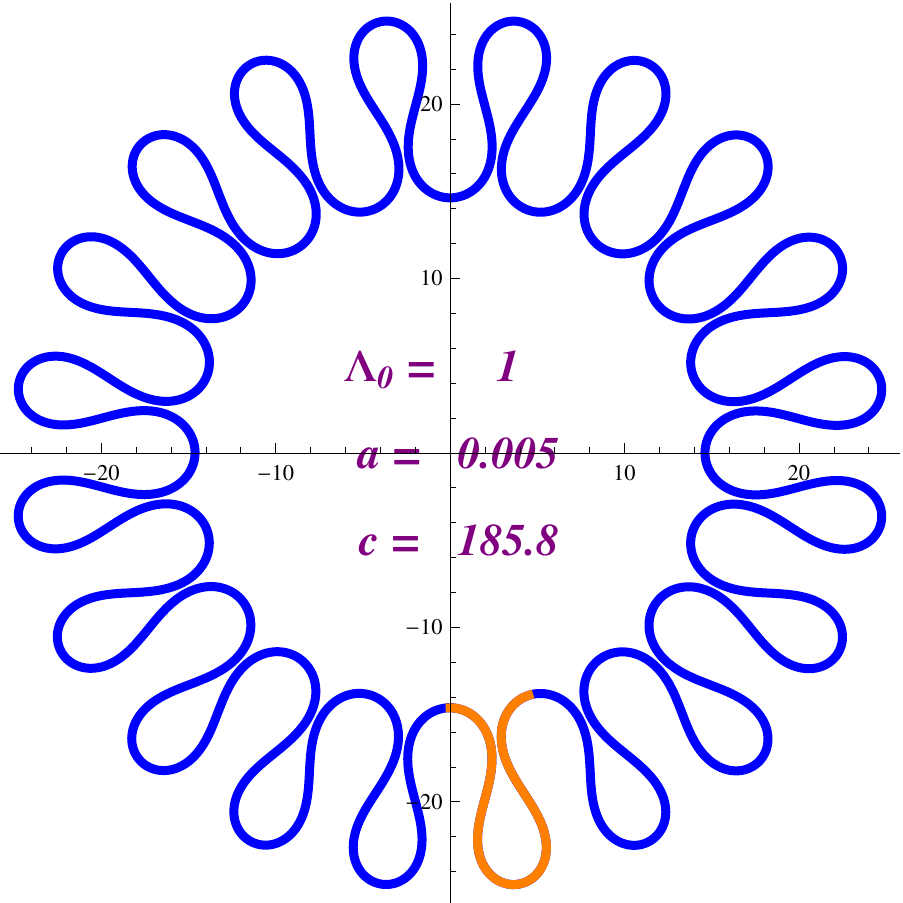}\hskip.2cm \includegraphics[width=3.5cm,height=3.5cm]{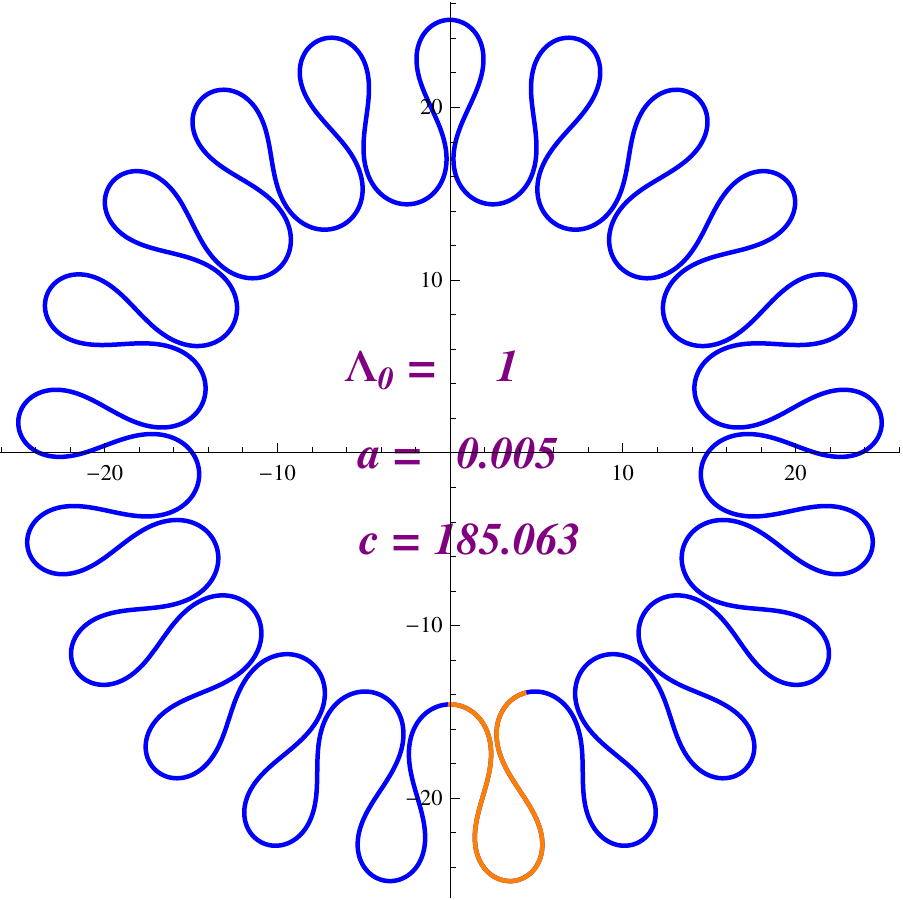}\hskip.2cm\includegraphics[width=3.5cm,height=3.5cm]{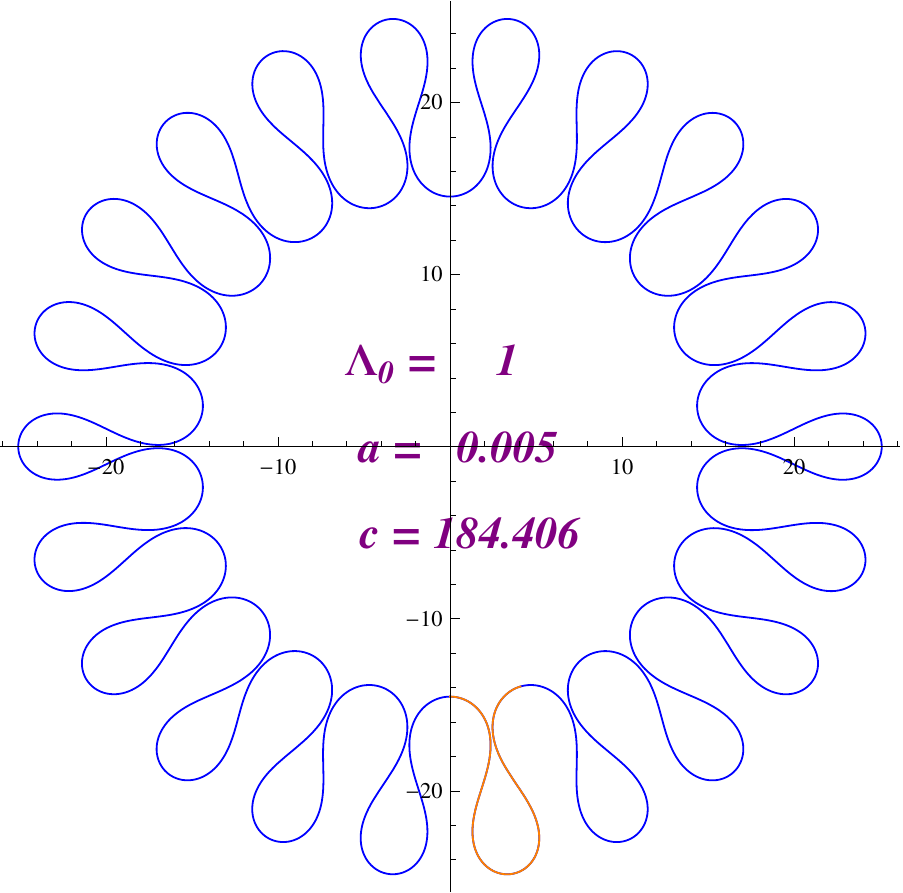}}
\caption{}
\label{embg6}
\end{figure}
\FloatBarrier

\begin{prop}
Let $r_1$, $r_2$ and $r_3$ be the functions given in  Definition \ref{def of ri} and let $C_1$, $C_2$ and $C_3$ be the functions defined in Lemma \ref{roots}. For $i=1,2 $, define the  functions

$$b_i(a)= \frac{(\pi (4 C_i + r_i (a r_i-4)))}{r_i \sqrt{16 + 8 a (C_i - 3 r_i) + 6 a^2 r_i^2}}\:,$$
and let $x_1$ and $x_2$ be the first two roots of the polynomial $q(r,a,C)$.

{\bf a.)} If $a<0$,  then $\lim_{C\to C_1(a)^+}\Delta{\tilde \theta}(C,a,x_1,x_2)=b_1(a)$. 

{\bf b.)} If $a>0$ then $\lim_{C\to C_1(a)^-}\Delta{\tilde \theta}((C,a,x_1,x_2)=b_1(a)$. 

{\bf c.)} If $a>0$ then $\lim_{C\to C_2(a)^+}\Delta{\tilde \theta}((C,a,x_1,x_2)=b_2(a)$. 

\end{prop}

\begin{proof}
 
Since  $\Delta{\tilde\theta}=\int_{x_1}^{x_2}\frac{4C+r(a r-4)}{r\sqrt{q}}$, then in every case, when $C$ approaches the limit value the two roots approaches $r_i$ (i=1 or 2) which is a root of $q$ with multiplicity $2$. Therefore Lemma \ref{lemma 1} applies and the proposition follows. Notice that the value $A$ in Lemma  \ref{lemma 1} is given by 

$$A=-\frac{1}{2}q^{\prime\prime}(r_i)=16 + 8 a (C_i - 3 r_i) + 6 a^2 r_i^2\:.$$

\end{proof}

 \begin{figure}[h!]
\centerline{\includegraphics[width=4.5cm,height=3.5cm]{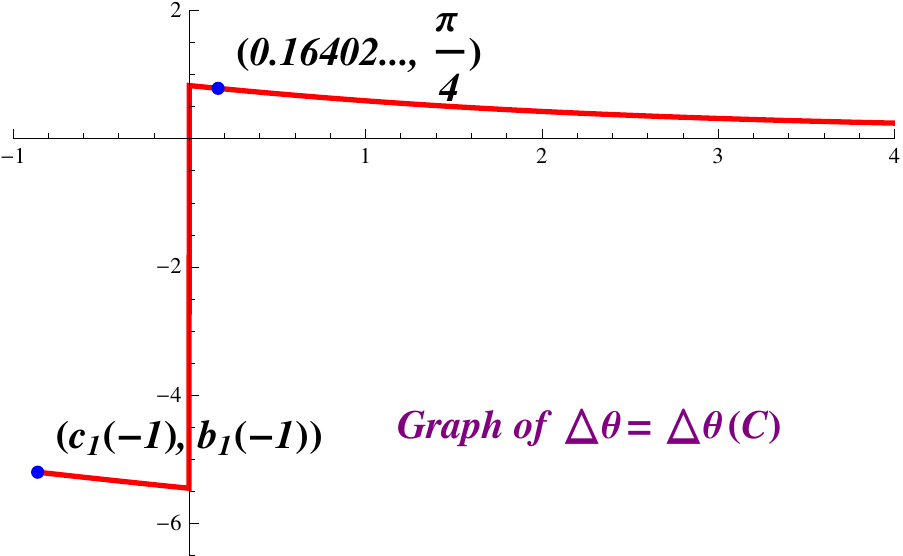}\hskip.2cm \includegraphics[width=4cm,height=3.5cm]{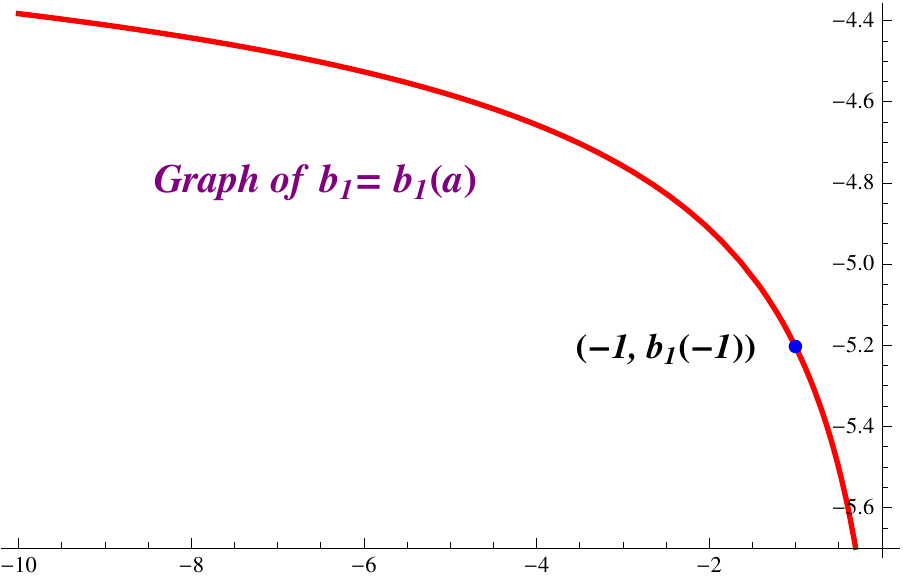}\hskip.2cm\includegraphics[width=3.5cm,height=3.5cm]{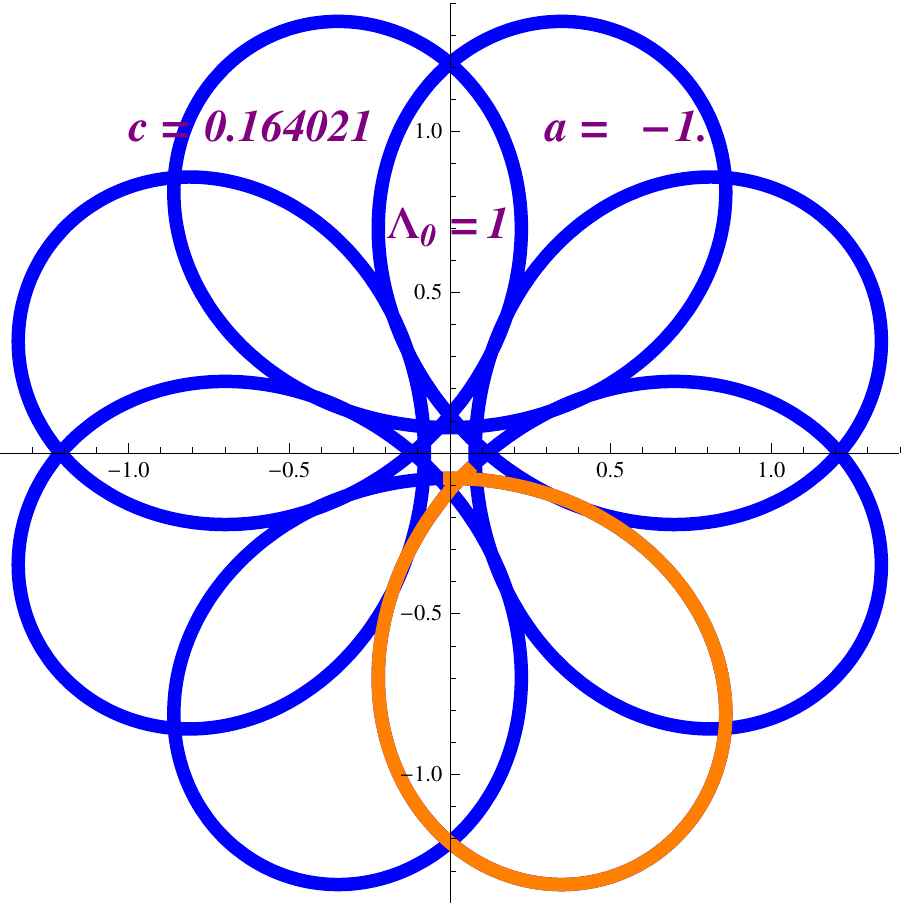}}
\caption{ The first picture shows the graph of the funcition $\Delta{\tilde \theta}(C)$ for $a=-1$.  The first point in the graph $(c_1(-1),b_1(-1))$ has been highlighted as well as the point the graph with second entry equal to $\frac{\pi}{4}$. The profile curve of immersed surface represented by the point $(0.164021,\frac{\pi}{4})$ is shown in the last picture. The picture in the middle shows the graph of the function $b_1(a)$ which represents the limit of the function $\Delta {\tilde \theta}(C)$ when $C$ goes to the first value in its domain. The point in this graph when $a=-1$ has been highlighted. }
\label{lp1}
\end{figure}

 \begin{figure}[h!]
\centerline{\includegraphics[width=4.5cm,height=3.5cm]{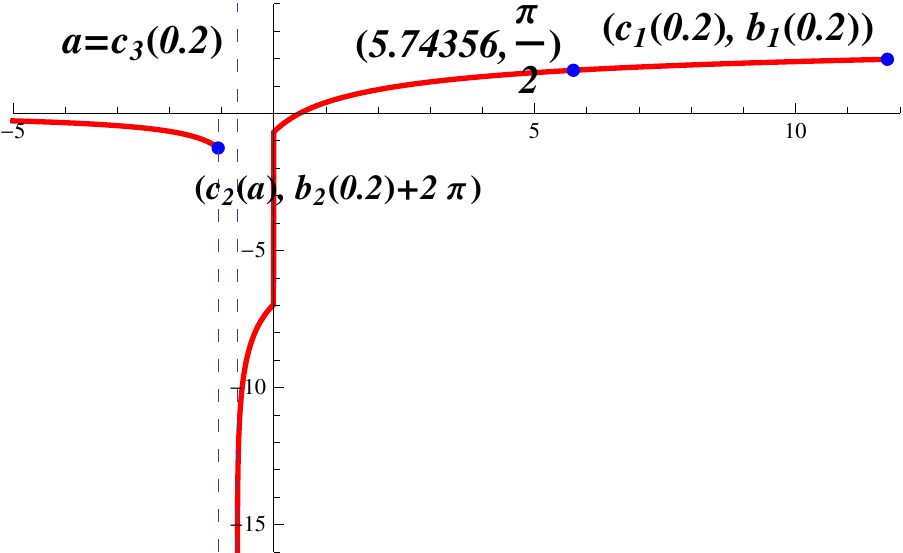}\hskip.2cm \includegraphics[width=4cm,height=3.5cm]{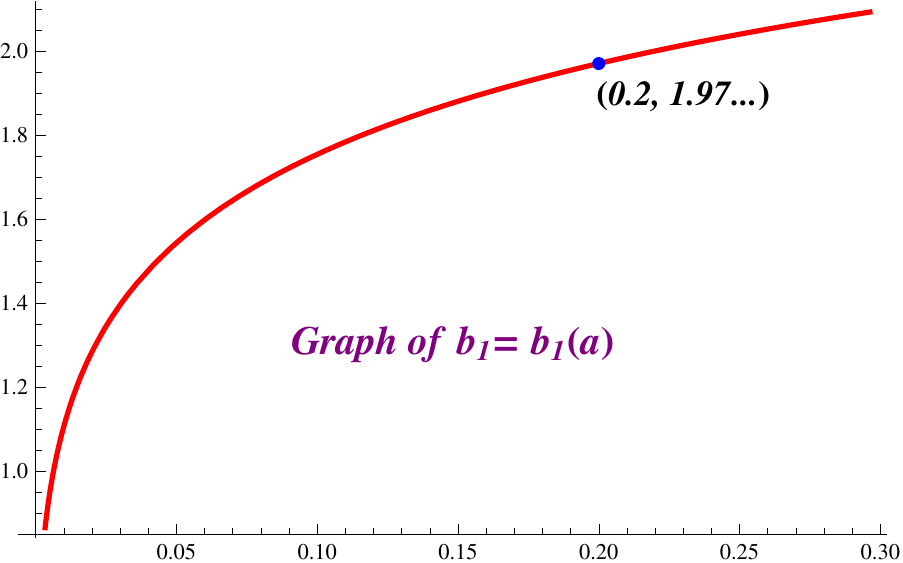}\hskip.2cm\includegraphics[width=3.5cm,height=3.5cm]{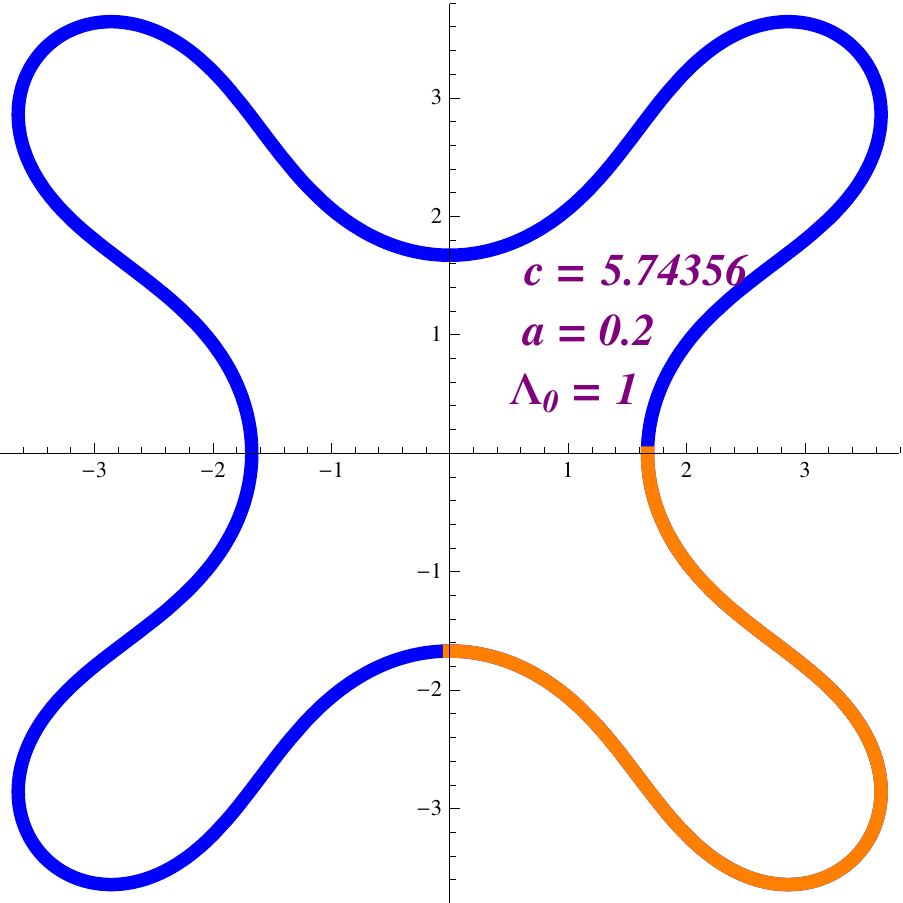}}
\caption{ The first picture shows the graph of the funcition $\Delta{\tilde \theta}(C)$ for $a=0.2$.  The last point in the graph $(c_1(0.2),b_1(0.2))$ has been highlighted as well as the point the graph with second entry equal to $\frac{\pi}{2}$. The profile curve of immersed surface represented by the point $(5.74356,\frac{\pi}{2})$ is shown in the last picture. The picture in the middle shows the graph of the function $b_1(a)$ which represents the limit of the function $\Delta {\tilde \theta}(C)$ when $C$ goes to the last value in its domain. The point in this graph when $a=0.2$ has been highlighted.}
\label{lp2}
\end{figure}
\FloatBarrier

Consider the graph of $C\rightarrow \Delta{\tilde\theta}$ for some particular value of $a$ in the cases considered in Proposition 
\ref{conf space b=1}.

When $a$ is between $0$ and $\frac{8}{27}$ and $C$ is between $C_2(a)$ and $C_3(a)$, there are two surfaces associated with these values of $a$ and $C$ due to the fact that the curve $G=C$ has two connected components. These surfaces are those described in case (ii) of Proposition  \ref{conf space b=1}. A natural question to ask is if it is possible to find values of $a$ and $C$ such that both surfaces are properly immersed, that is, can we find $a$ and $C$ such that $\Delta{\tilde \theta}(C,a,x_1,x_2)/2\pi$ and $\Delta {\tilde \theta}(C,a,x_3,x_4)/2\pi$  are rational numbers. Here $x_1<x_2<x_3<x_4$ are the four roots of $q(r,C,a)$. As a consequence of the following proposition we have that either both surfaces are properly immersed or both surfaces are dense.

\begin{prop}\label{complex variable prop}
Let $\Lambda_0=1$, let $a\in (0, 8/27)$, let $C\in (C_2(a), C_3(a))$ and let $x_1<x_2<x_3<x_4$ be the four roots of $q(r,a,C)=-16 C^2 + 64 r+32 C\,  r-16 \, r^2 - 8 a C\,  r^2+8 a \, r^3 -a^2\,  r^4 $. Then
 if  $C<0$, there holds
$$ \Delta {\tilde \theta}(C,a,x_3,x_4)  =\, \int_{x_3}^{x_4} \frac{(4C+ a r^2-4 \Lambda_0 r) }{r \sqrt{q(r,a,C)}}\, dr=
\, \int_{x_1}^{x_2} \frac{(4C+ a r^2-4 \Lambda_0 r) }{r \sqrt{q(r,a,C)}}\, dr+2\pi= \Delta {\tilde \theta}(C,a,x_1,x_2)+2 \pi \:.$$
 If $C>0$ then there holds

$$ \Delta {\tilde \theta}(C,a,x_3,x_4)  =\, \int_{x_3}^{x_4} \frac{(4C+ a r^2-4 \Lambda_0 r) }{r \sqrt{q(r,a,C)}}\, dr=
\, \int_{x_1}^{x_2} \frac{(4C+ a r^2-4 \Lambda_0 r) }{r \sqrt{q(r,a,C)}}\, dr= \Delta {\tilde \theta}(C,a,x_1,x_2)\:.$$

\end{prop}

\begin{proof}

Let $U$ be the region shown in Figure \ref{region} bounded by the curves $\gamma_1$, $\gamma_2$, $\gamma_3$ and $\gamma_4$. The curves $\gamma_i$ will carry the orientations induced by the fact that they are components of $\partial U$.  Define

$$f(z):= \frac{i (4C+ a z^2-4 z) }{z a (z-x_2) \sqrt{T_1(z)} (z-x_4) \sqrt{T_2(z)}} \:,$$

where  $\sqrt{\quad}$ is the principal branch of the square root, $T_1(z)=\frac{z-x_1}{z-x_2}$ and $T_2(z)=\frac{z-x_3}{z-x_4}$.

When the radius of $\gamma_1$, and the radii of the circular arcs of  $\gamma_2$ and $\gamma_3$ tend to zero, we have that $\int_{\gamma_1}f(z)\, dz$ goes to $-2 \pi \frac{C}{|C|}$, $\int_{\gamma_2}f(z)\, dz$ goes to $2 \Delta{\tilde\theta}(a,C,x_1,x_2)$ and $\int_{\gamma_3}f(z)\, dz$ goes to $-2 \Delta{\tilde\theta}(a,C,x_3,x_4)$.  When the radius of $\gamma_4$ goes to $\infty$, we have that $\int_{\gamma_4}f(z)\, dz$ goes to $-2 \pi $.

The proof follows  because the integral $\int_{\gamma_1+\gamma_2+\gamma_3+\gamma_4}f(z)dz$ vanishes since $f$ is analytic in 
$U$.

 \begin{figure}[ht]
\centerline{\includegraphics[width=5cm,height=5cm]{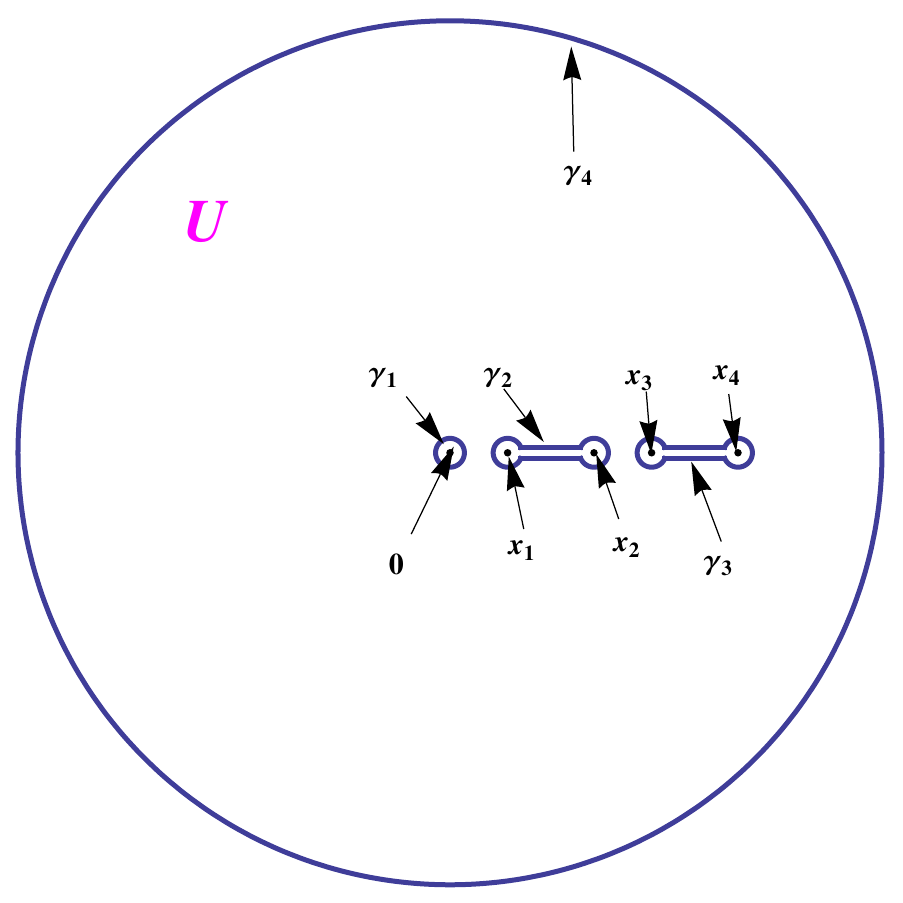}}
\caption{Region of integration used in the proof of Proposition \ref{complex variable prop}.}
\label{region}
\end{figure}

\end{proof}

 \begin{figure}[ht]
\centerline{\includegraphics[width=3.5cm,height=3.0cm]{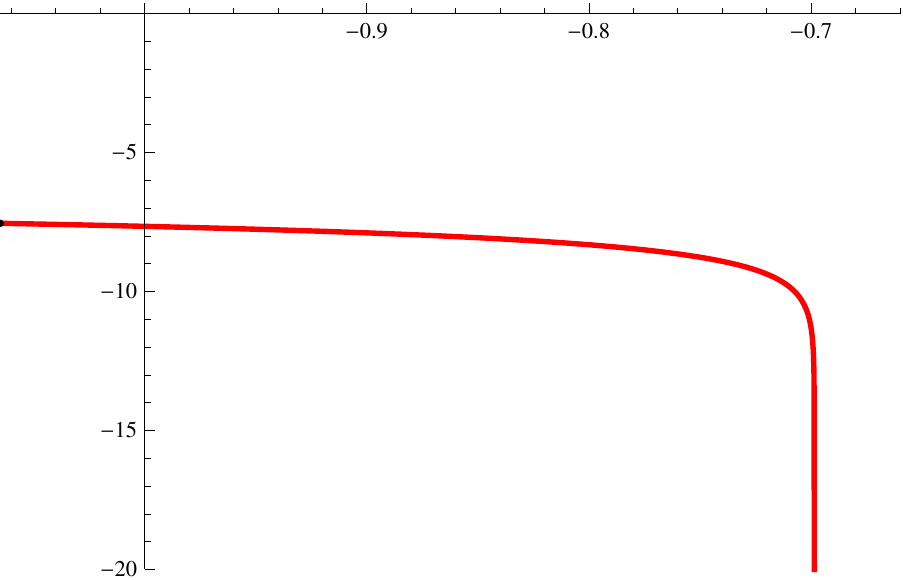}\hskip.2cm \includegraphics[width=3.5cm,height=3.0cm]{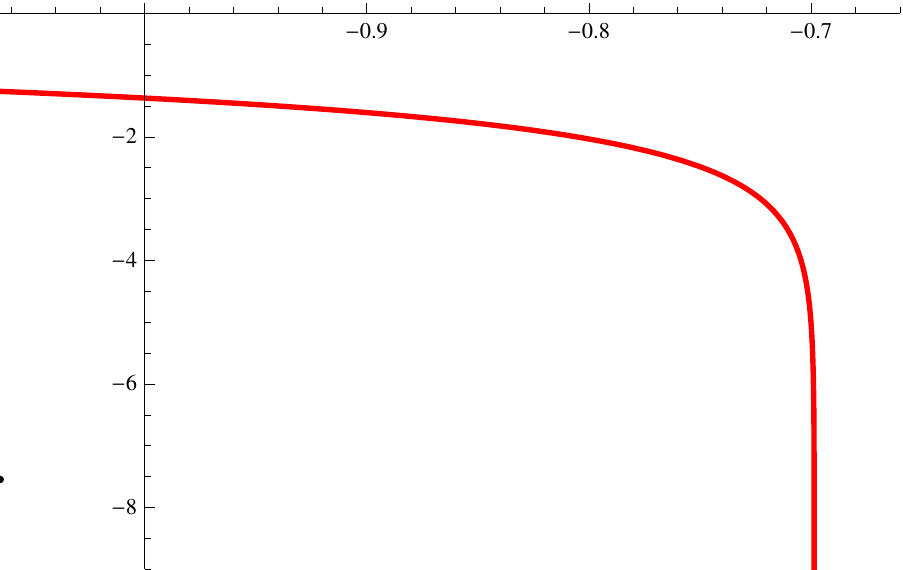}\hskip.2cm\includegraphics[width=3.5cm,height=3.0cm]{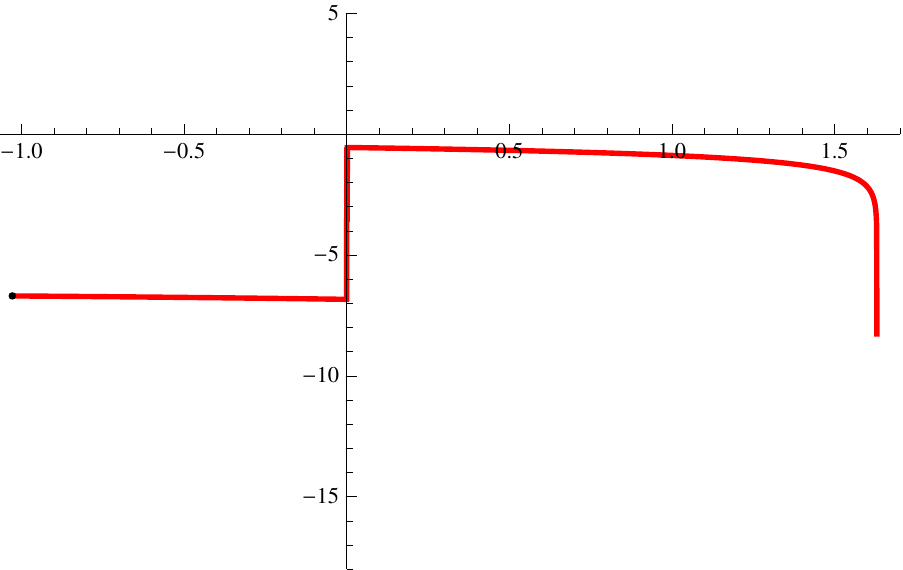}\hskip.2cm\includegraphics[width=3.5cm,height=3.5cm]{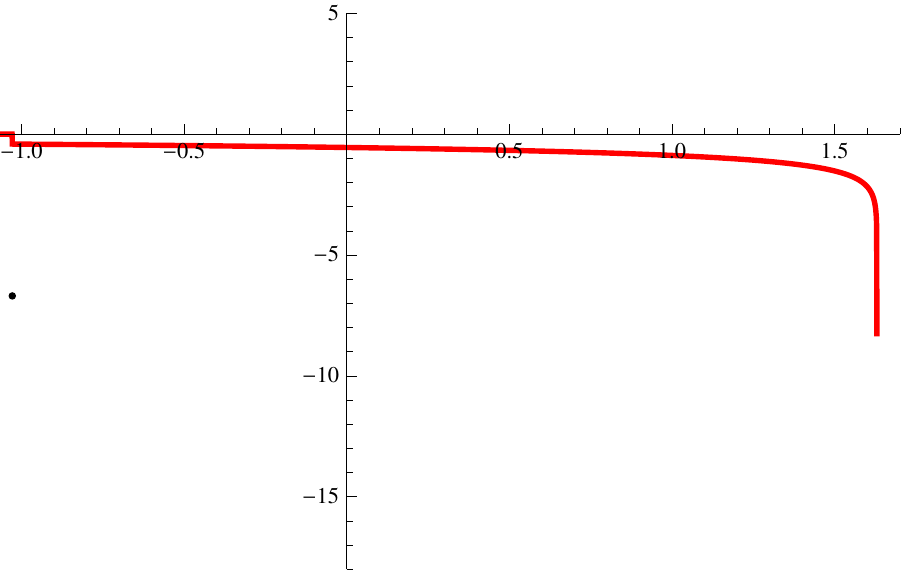}}
\caption{ The first picture shows the graph of the function $\Delta{\tilde\theta}(C,0.2,x_1,x_2)$ and the second picture shows the graph of the function $\Delta{\tilde\theta}(C,0.2,x_3,x_4)$. We have $C_2(0.2)<C_3(0.2)<0$. The third picture shows the graph of the function $\Delta{\tilde\theta}(C,0.1,x_1,x_2)$ and the fourth picture shows the graph of the function $\Delta{\tilde\theta}(C,0.1,x_3,x_4)$. We have $C_2(0.1)<0<C_3(0.1)$.}
\label{compv}
\end{figure}

 \begin{figure}[ht]
\centerline{\includegraphics[width=3.5cm,height=3.0cm]{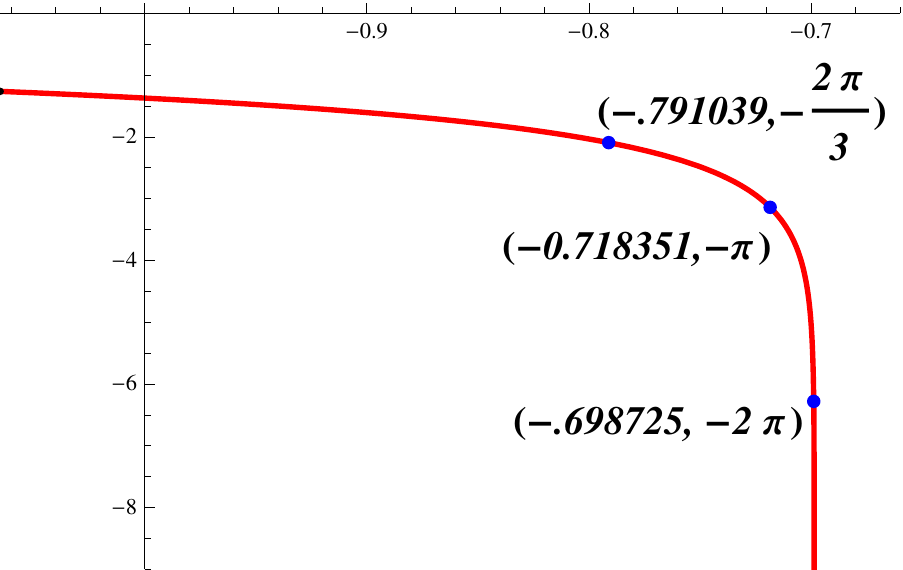}\hskip.2cm \includegraphics[width=3.cm,height=3.5cm]{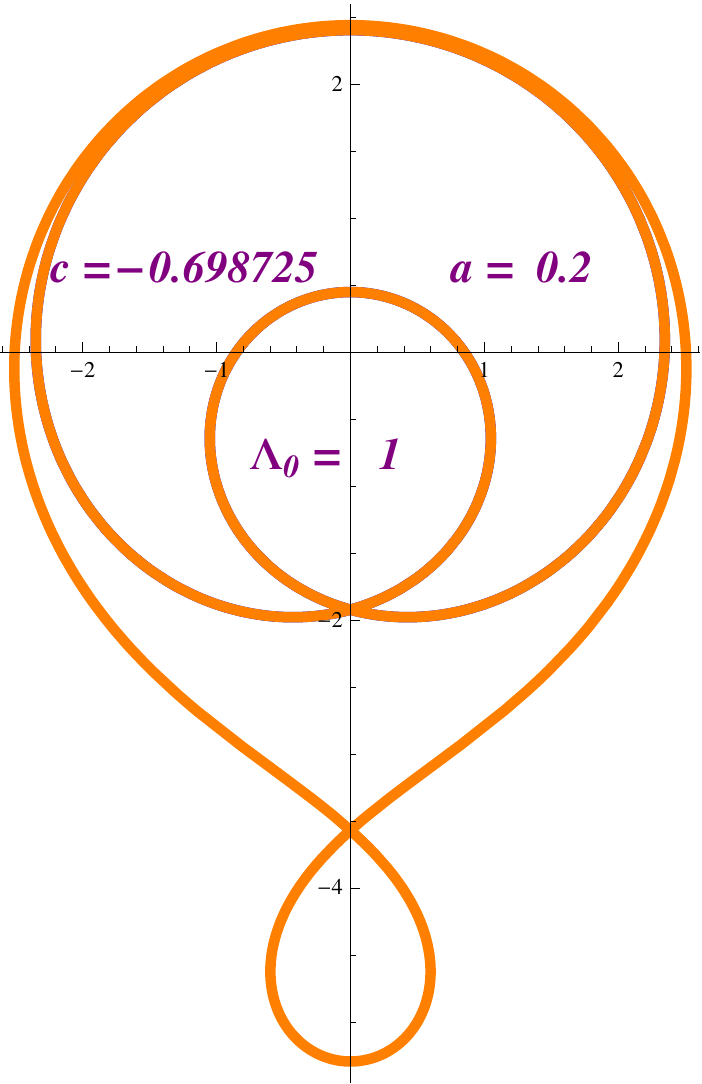}\hskip.2cm\includegraphics[width=3.5cm,height=2.5cm]{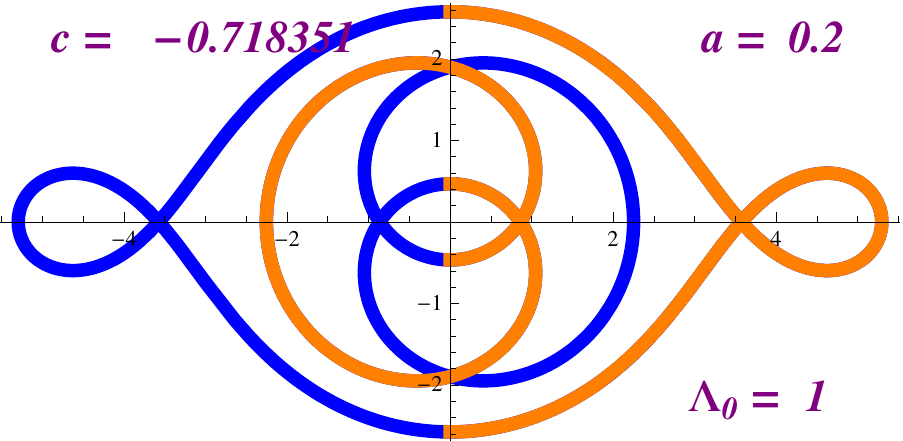}\hskip.2cm\includegraphics[width=3.5cm,height=3.2cm]{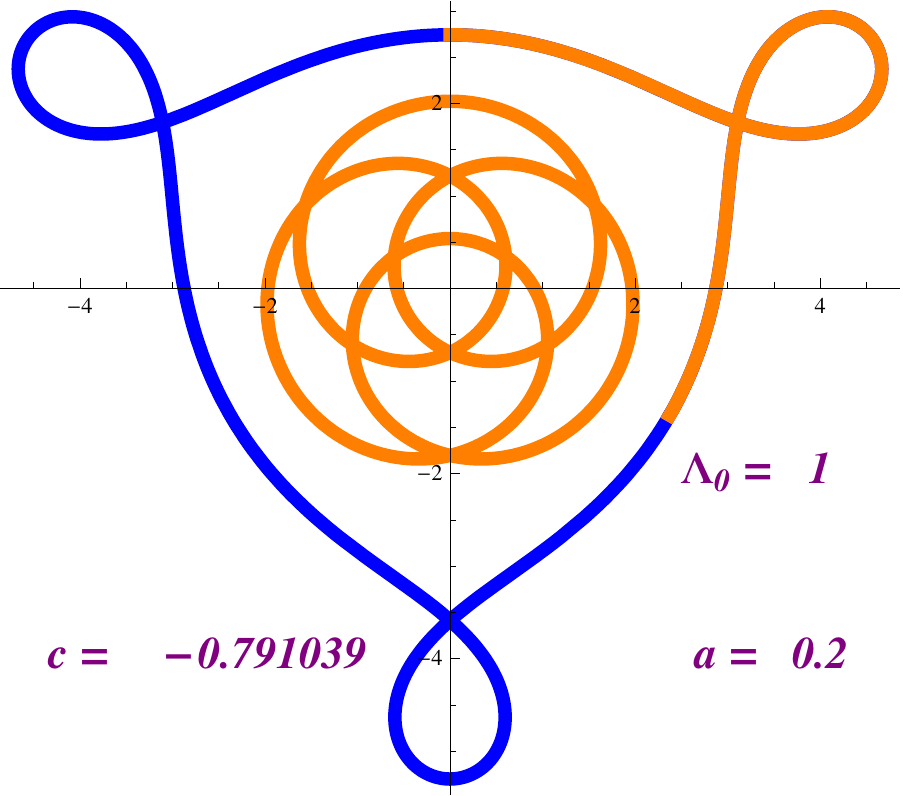}}
\caption{ The first picture shows the graph of the function $\Delta{\tilde\theta}(C,0.2,x_3,x_4)$. Three points in this graph have been highlighted, these point correspond to the surfaces on the second, third and fourth picture. These graphs are disconnected due to the fact that the level set $G=C$ is disconnected. The pair of surfaces correspond to those explained in case (ii) from Proposition \ref{conf space b=1}. We can see how the variation of the angle of the profile curve is the same for both surfaces as predicted by Proposition \ref{complex variable prop}.}
\label{tg}
\end{figure}
\FloatBarrier

%------------------------------------------------------------

%-------------------------------------------------------------
\eject

\section{Stability of cylinders}
In this section, we will consider the stability of the equilibrium surfaces described above for both free and fixed boundary value problem. Let $\Sigma_h$ be a part of an equilibrium surface of the form $\alpha\times [-h/2, h/2]$ where $\alpha$ denotes the generating curve. The curve $\alpha$ will always be assumed to be closed. The surface  $\Sigma_h$ arises as a critical point for the functional (\ref{E1}).

First, we regards the surface $\Sigma_h$ for a free boundary problem where the supporting surface consists of two horizontal planes.
 The first variation of this functional  for a variation $\delta X=\psi \nu +T$ is (\cite{PP}),  
$$ \int_{\Sigma_h} (-2H-\frac{a}{2}R^2+\Lambda_0)\psi\:d\Sigma 
 +\oint_{\partial \Sigma_h}  dX\times (\nu-\frac{a}{2} \nabla'\frac{R^4}{16}+\Lambda_0 \nabla '\frac{R^2}{4})\cdot \delta X\:.$$
 Here, $\nabla'$ denotes the gradient operator in ${\bf R}^3$.  By taking first compactly supported variations, one deduces immediately that in the interior of the surface 
$-2H-\frac{a}{2}R^2+\Lambda_0\equiv 0$ must hold. If the surface is constrained to have free boundary components in the planes $x_3\equiv \pm h/2$, then for the surface to be in equilibrium, the boundary integral must vanish for all admissible variation, that is all $\delta X$ with $\delta X\cdot E_3\equiv 0$ on the boundary. This amounts to the condition
$$ E_3\:||\:\alpha'\times (\nu-\frac{a}{2} \nabla'\frac{R^4}{16}+\Lambda_0 \nabla '\frac{R^2}{4})\:.$$
However, $\alpha'$, $\nabla'R^2$ and $\nabla'R^4$ are all co-planar vectors so their cross products are already vertical. We conclude that $\alpha'\times \nu$ must be vertical along the boundary and so the surface meets both the horizontal planes in a right angle. This condition holds automatically for a cylindrical surface and so these surfaces are in equilibrium for the free boundary problem.

The second variation of energy for the free boundary problem is
\begin{equation}
\label{svfb}
\delta^2{\mathcal E}=-\int_\Sigma \psi L\psi\:d\Sigma +\oint_{\partial \Sigma} \psi \partial_n \psi\:ds=\int_\Sigma |\nabla\psi|^2-(|d\nu|^2+aq)\psi^2\:d\Sigma \end{equation}
Here $L=\Delta+(|d\nu|^2+a(q-x_2\nu_3))$, $q=x_1\nu_1+x_2\nu_2+x_3\nu_3$, and the admissible functions $\psi$ need only be smooth and satisfy
\begin{equation}
\label{mv0}\int_\Sigma \psi\:d\Sigma =0\:.\end{equation} The second variation formula (\ref{svfb}) can be derived in a manner similar to that used in \cite{KP}.
Therefore  stability for the free boundary value problem can be characterized by
\begin{equation}
\label{stab} \inf \frac{ -\int_\Sigma \psi L[\psi]\:d\Sigma +\oint_{\partial \Sigma} \psi \partial_n \psi\:ds }{\int_\Sigma \psi^2\:d\Sigma}\ge 0\:,\end{equation}
where the infimum is taken over smooth functions on $\Sigma$ satisfying (\ref{mv0}).
This gives rise to a spectral problem on the space of smooth functions satisfying (\ref{mv0}):
\begin{equation}
\label{evp}  (L+\lambda_j)[\psi] -\frac{1}{|\Sigma |}\int_\Sigma \psi =0\:{\rm in}\:\Sigma\:, \quad \partial_n\psi=0\:, {\rm on}\:\partial \Sigma \:.\end{equation}

Our first result concerns the general solution of the free boundary problem.
\begin{prop}
\label{A}
When $a>0$ holds, the free boundary problem has no stable, embedded critical points.
\end{prop}
\begin{proof} 

The proof uses an idea of R. Lopez  \cite{RL} which was applied in the case of closed drops.

The functions $\nu_i$, $i=1,2$ satisfy the equation $L[\nu_i]=ax_i$. Using, say $\nu_1$, in the second variation formula, yields
$$\delta^2{\mathcal E}=-a\int_\Sigma x_1\nu_1\:d\Sigma\:.$$
By applying the Divergence Theorem to the field $x_1E_1$, this becomes
$$\delta^2{\mathcal E}=-aV(\Sigma)<0\:.$$
Also, note that the variation fields $E_i$, $i=1,2$ are admissible since they are tangent to the horizontal plane and thus satisfy
$$\int_\Sigma \nu_i\:d\Sigma =\int_\Omega \nabla \cdot E_i\:d^3x=0\:,$$
where $\Omega$ is the three dimensional region enclosed by $\Sigma$ and any two
\end{proof} 
\begin{rem} The previous result also holds in the immersed case if it assumed that the signed volume enclosed by the surface is positive.\end{rem}

Let $\Sigma$ be an equilibrium cylinder over a curve $\alpha$, i.e. $\Sigma=\alpha \times [-h/2, h/2]$. We assume $\alpha$ is parameterized by arc length $s$. For $\Sigma$, $2H=\kappa:=$
the curvature of $\alpha$. The second variation formula (\ref{svfb})  becomes
\begin{equation}
\label{sv}
\delta^2{\mathcal E}=\int_\alpha\int_{-h/2}^{h/2} \psi_s^2+\psi_z^2-(\kappa^2+a\xi_2)\psi^2 \:dz\:ds\:.\end{equation}

We analyze (\ref{evp}) using separation of variables. Writing $\psi=u(s)f(z)$, we see that (\ref{evp}) becomes
\begin{equation}
\label{one}u(s)f_{zz}+f(z)u_{ss}+(\kappa^2(s)+a\xi_2(s)+\lambda)u(s)f(z)=A\:, f_z(\pm h/2)=0\:.\end{equation}
for a suitable constant $A$.
In addition, we must have
$$\int_{-h/2}^{h/2} f(z)\:dz=0\:,{\rm or}\: \int_\alpha u(s)\:ds=0\:.$$
Dividing (\ref{one}) by $uf$, we arrive at
$$ \frac{f_{zz}(z)}{f(z)}=-\frac{u_{ss}+(\kappa^2(s)+a\xi_2(s)+\lambda)u(s)}{u(s)}+\frac{A}{f(z)u(s)}\:.$$
The right hand side must be independent of $s$, so by varying $z$ in $f$ in the denominator, we get a contradiction unless $A=0$ or $f\equiv$ constant. In either case,
both sides are equal to a constant $:=-\mu$. If $f$ is non constant, and hence $A =0$, the smallest value of $\mu$ consistent with the boundary condition $f_z=0$ is $\mu=(\pi/h)^2$, when
$f=\sin(\pi z/h)$. Renaming $\lambda$ as $\lambda+\mu$, we have that $u$ satisfies the equation
\begin{equation}
\label{two}
u_{ss}+(\kappa^2(s)+a\xi_2(s)+\lambda)u(s)=0\:.
\end{equation}
\begin{lem}
\label{sl0}
 The function $\xi_1$ satisfies the equation
\begin{equation}
 \label{sl}
 (\xi_1)_{ss}+(\kappa^2+a\xi_2)\xi_1=0\:.\:\end{equation}\end{lem}
\begin{proof} The function $\xi_1=(R^2)_s/2$ satisfies $(\xi_1)_s=\kappa \xi_2+1$.   Recall that $\kappa=2H=\Lambda_0-aR^2/2$, so
$(\xi_1)_{ss}=\kappa_s\xi_2+\kappa(\xi_2)_s=(\Lambda_0-R^2/2)_s\xi_1-\kappa^2\xi_1=-a\xi_2\xi_1-\kappa^2\xi_1$.
Therefore (\ref{sl}) holds.\end{proof}

\begin{lem}
Assume that $\alpha$ is not a circle. The function $\xi_1$ satisfies (\ref{two}) with $\lambda=0$ and, in addition satisfies
\begin{equation}
\label{mv} 
\int_\alpha \xi_1\:ds=0\:.
\end{equation}
Hence, the lowest eigenvalue of (\ref{two}) on $\alpha$,  without the  constraint that the mean value be zero, is negative.
\end{lem}
\begin{proof} The fact that $\xi_1$ satisfies the equation follows from the previous lemma.
Since $\xi_1=(R^2/2)_s$, (\ref{mv}) holds. Also, $\xi_1\equiv 0$ is and only if $R^2\equiv$ constant so $\xi_1\equiv 0$ if and only if $\alpha$ is a circle.\end{proof}

\begin{thm} 
\label{CP}
Let $\Sigma_h=\alpha\times [-h/2, h/2]$ be an equilibrium surface.
With regard to the free boundary problem, we have
\begin{enumerate}

\item If the function $R^2$ has more than two critical points on $\alpha$, then $\Sigma$ is unstable for all $h$.
\item $\Sigma_h$ is unstable for $h>>0$.
\end{enumerate}
\end{thm}
\begin{proof} Since $\xi_1:=(R^2)_s/2$, the eigenfunction $\xi_1$ vanishes at more than two points. From the local behavior if eigenfunctions for Sturm-Liouville problems, $\xi_1$ has a sign change at each of its zeros and so $\xi_1^{-1}(\{0\})$ must consist of at least four points. Let $\mu_1<\mu_2<\mu_3<..$ denote the distinct eigenvalues of 
$$({\hat L}+\mu)[u(s)]:=u_{ss}+(\kappa^2+a\xi_2+\mu)u=0\:,$$
on $\alpha$ considered {\it without} the integral constraint that the mean value of $u(s)$ be zero. It is known, ( \cite{CL}, Theorem 3.1), that the number of zeros of an eigenfunction of $\mu_j$ is exactly $j-1$. Since $\xi_1$ has at least four zeros, we must have $\mu_i<0$ for $i=1,\dots,4$. 

The eigenfunction for $\mu_1, ... ,\mu_4$ therefore span a space of functions on $\Sigma_h$ having vanishing normal derivative on $\partial \Sigma_h$ on which the quadratic form
$$f\mapsto -\int_{\Sigma_h}fL[f]\:d\Sigma\:.$$
is negative definite. In this space, there is a subspace having dimension at least three, of functions satisfying the additional linear condition
$$\int_{\Sigma_h} f\:d\Sigma =0\:.$$
This proves (i).

In any case, if $\alpha$ is not a circle, then $R^2$ has at least two distinct critical points, so $\mu_1<0$ holds. If $u_1(s)$ is an eigenfunction belonging to $\mu_1$, then the function
$f:=u_1(s)\sin(\pi z/h)$ has vanishing normal derivative on $\partial \Sigma_h$, has zero mean value, and satisfies
$$ -\int_{\Sigma_h}fL[f]\:d\Sigma=\frac{\pi^2}{h^2}+\mu_1\:,$$
which is negative for $h>>0$. 

We remark that the first curve shown in Figure  \ref{embg1} is an example of a curve for which $R^2$ has exactly two critical points. However, in this case  $\Sigma_h$ is unstable for all $h$ by  Proposition (\ref{A}), since $a>0$ holds. 
\end{proof} 
\begin{cor}
If $\alpha$ is embedded and non circular, then $\alpha\times [-h/2,h/2]$ is unstable for the free boundary problem for all $h$.
\end{cor}
\begin{proof} The Four Vertex Theorem implies that $\kappa_s=0$ at, at least, four points on $\alpha$. Since $\kappa=\Lambda_0=aR^2/2$, $R^2_s$ must vanish at these points and the result then follows from the previous theorem. \end{proof}
We define the Morse index of $\Sigma$ to be the number of negative eigenvalues for the problem
$$(L+\lambda)\psi={\rm constant}\: {\rm in }\:\Sigma, \quad \partial_n\psi=0\:, {\rm on}\:\partial \Sigma\:, \int_\Sigma \psi=0\:.$$
By linear algebra,  a lower bound for this index is $J-1$, where $J$ is the number of negative eigenvalues for the same problem {\it without} the condition that the mean value of
the functions vanish. The  proof of Theorem \ref{CP} can be modified to show that if the function $R^2$ has at least $J$ critical points on $\alpha$, then the number of negative eigenvalues 
without the integral constraint is at least $J$ and hence the Morse index is at least $J-1$. 
In the first part of the paper, we explained that for $\alpha$ to be immersed,  $\Delta {\tilde \theta}/2\pi$ is a rational number $=:m/k$ in lowest terms. Then the number of fundamental
pieces which make up $\alpha$ is either $k$ or $2k$ depending, respectively, on whether $m$ is even or odd. Each fundamental piece contributes at least two critical points of $\xi_1$ which correspond to the maxima and minima of $R^2$. In this way, we obtain a simple lower bound for the Morse Index of the free boundary problem.

For a round cylinder of radius $R$ and height $h$, (\ref{svc} ) becomes
\begin{equation}
\label{svc}
\inf_{\int_0^{2\pi}\int_{-h/2}^{h/2}\psi\:R d\theta dz=0} \frac{\int_0^{2\pi}\int_{-h/2}^{h/2} (\psi_z^2+\frac{\psi_\theta^2}{R^2}-(1/R^2 +aR)\psi^2)R\:dz d\theta}
{\int_0^{2\pi}\int_{-h/2}^{h/2}\psi^2 \:R d\theta dz}\:.\end{equation}

Clearly this is non negative if $aR+1/R^2\le 0$ holds and so round cylinders satisfying this inequality are stable for the free boundary problem for every $h>0$.

The only interesting cases are then when $a<0$ and $aR+1/R^2>0$ hold. In this case, the function minimizing (\ref{svc}) is $\sin (\pi z/h)$ which yields that the cylinder is stable for the
free boundary problem if and only if
$$\frac{\pi^2}{h^2}\ge aR+1/R^2\:,$$
holds. To summarize, we have
\begin{thm}
The only stable embedded cylinders for the free boundary problem are round cylinders of radius $R$ and height $h$ which satisfy $a<0$, $aR+1/R^2>0$ and $\pi^2/h^2 \ge aR+1/R^2$.
\end{thm}

The cylindrical surface  $\Sigma_h$ is an equilibrium surface if and only if the generating curve $\alpha$ is a critical point for the functional
\begin{equation}
\label{curv} {\mathcal G}_{(a,\Lambda_0)}[\alpha]={\mathcal L}[\alpha]-\frac{a}{2}\int_{U \cap{ \bf R}^2} R^2\:d{\mathfrak A}+\Lambda_0 {\mathfrak A}\:.\end{equation}
Here ${\mathcal L}[\alpha]$ denotes the length of $\alpha$ and ${\mathfrak A}$ is the area enclosed by $\alpha$.  (This area should be considered as a signed area when $\alpha$ is not embedded.) The admissible variations of the curve are those which preserve the enclosed area. With regard to the stability of $\Sigma_h$, there are three possibilities:
\begin{itemize}
\item $\alpha$ is strongly stable for ${\mathcal G}_{(a,\Lambda_0)}$, i.e. the second variation of ${\mathcal G}_{(a,\Lambda_0)}$ is non negative for all variations of $\alpha$. In this case, $\Sigma_h$ is stable for all $h>0$. This  case only occurs for certain circles with particular choices of $a$.
\item The curve $\alpha$ is stable for ${\mathcal G}_{(a,\Lambda_0)}$ but not strongly stable. In this case $\Sigma_h$ will be stable if and only if $h$ is sufficiently small.
\item  $\alpha$ is unstable for ${\mathcal G}_{(a,\Lambda_0)}$. In this case $\Sigma_h$ is unstable for all values of $h$. This is the case if either $a>0$ holds or if $\alpha$
is embedded and is not a circle.
\end{itemize}

Note that for round cylinders satisfying  $a<0$, $aR+1/R^2>0$ snd $\pi^2/h^2 \ge aR+1/R^2$, stability will be lost if the radius $R$ is decreased sufficiently while keeping the height fixed.
It can be shown using results of Crandall and Rabinowitz, \cite{CR1971}, that a bifurcation occurs at the value of $R=R_0$ such that $\pi^2/h^2 = aR_0+1/R_0^2$ and there 
is a family of equilibrium solutions of the free boundary problem of the form  $X(\sigma)=([R_0+\sigma \sin(\pi z/h)]\cos(\theta), [R_0+\sigma \sin(\pi z/h)]\sin(\theta), z) +{\mathcal O}(\sigma^2)$.  In particular, there exist non cylindrical, embedded solutions of the free boundary problem.

We now discuss the stability of cylinders for the fixed boundary problem. Stability is still characterized by the non negativity of (\ref{stab}) but now the competing functions
must vanish on the boundary.  This means that stability for the fixed boundary problem is monotone; any subdomain of a stable domain is itself stable, a property which does not hold for the free boundary problem. Also,  for sufficiently small $h$, $\alpha\times [-h/2, h/2]$ will always be stable for the fixed boundary problem.

 Proposition \ref{A} does not apply.
By choosing $\psi=\sin(2\pi z/h)$, we obtain that $\alpha\times [-h/2, h/2]$ is stable for the fixed boundary value problem only if
$$\frac{4\pi^2}{h^2}\ge \int_C\kappa^2\:ds+2a{\mathfrak A}\ge \frac{\pi^2}{{\mathcal L}^2}+2a\mathfrak{A}\:$$
holds where ${\mathfrak A}$ is the area enclosed by $\alpha$ and $L$ is the length of $\alpha$. Because the total curvature of an embedded curve is at least $2\pi$, in this case the previous
necessary condition for stability can be improved to
$$\frac{4\pi^2}{h^2}\ge \frac{4\pi^2}{{\mathcal L}^2}+2a\mathfrak{A}\:$$
There is nothing, however, to insure that the right hand side is positive. Indeed, for some round cylinders the quantity is negative. Our next result will give an upper bound on the height of any stable equilibrium cylindrical surface which is not a round cylinder. In particular, this upper bound only depends on the geometry of the generating curve $\alpha$.
\begin{thm} Let $\alpha \times  [-h/2, h/2]$ be any stable cylindrical equilibrium surface for the fixed boundary problem which is not a round cylinder. Then there holds
\begin{equation}
\label{CCC}
\frac{4e^{4\bar \xi_1}}{(\max (G_{\xi_2})^2)}\le\frac{e^{ 2({\bar \xi_1}-{\underline{\xi}_1)}}  {\mathcal L} [\alpha]}{\oint_\alpha (1+\kappa \xi_2)^2\:ds}\le \frac{4\pi^2}{h^2}\:.
\end{equation}
\label{inq}
Here ${\bar \xi}_1$ and ${\underline \xi}_1$ are, respectively the maximum and minimum of $\xi_1$ on the generating curve.
\end{thm}
\begin{proof} From Lemma \ref{sl0}, we have that $\xi_1(s)$ satisfies ${\mathcal L}[\xi_1]=0$. From the we easily obtain
\begin{equation}
\label{DD}{\hat  L}[e^{\xi_1}]=e^{\xi_1}({\hat L}[\xi_1]+((\xi_1)_s)^2)=e^{\xi_1} ((\xi_1)_s)^2=e^{\xi_1} (1+\kappa \xi_2)^2\:.\end{equation}
We take $\psi:=\sin(2\pi z/h)e^{\xi_1}$ in (\ref{sv}). Note that this function vanishes on the boundary and integrates to zero. Integrating first with respect to $h$, we obtain
\begin{eqnarray*}
\delta^2{\mathcal E}&=&\frac{h}{2}\oint_\alpha \frac{4\pi^2}{h^2}u^2+u_s^2-(\kappa^2(s)+a\xi_2(s))u^2\:ds\\
&=&\frac{h}{2}\oint_\alpha \frac{4\pi^2}{h^2}e^{2\xi_1}-e^{\xi_1}{\hat L}[e^{\xi_1}]\:ds\\
&=&\frac{h}{2}\oint_\alpha \frac{4\pi^2}{h^2}e^{2\xi_1}-e^{2\xi_1}(1+\kappa \xi_2)^2\:ds\\
&\le& \frac{h}{2}\bigl(\frac{4\pi^2}{h^2}e^{2{\bar \xi}_1} {\mathcal L}[\alpha]- e^{2{\underline \xi}_1}\oint_\alpha (1+\kappa \xi_2)^2\:ds  \bigr)\:.
\end{eqnarray*}
This proves the second inequality in (\ref{inq}).  
From (\ref{dz}) and (\ref{the ode}), we have
$$(1/2)G_{\xi_2}=\xi_1'=1+\kappa \xi_2\:.$$
So estimating the integral in the denominator by the maximum of the integrand times the length of the curve and using that ${\underline \xi_1}=-{\bar \xi_1}$ gives the result.
 \end{proof} 

A slight modification, replacing $\sin(2\pi z/h)$ with $\sin(\pi z/h)$, if the proof gives the following.
\begin{thm} Let $\alpha \times  [-h/2, h/2]$ be any stable cylindrical equilibrium surface for the free boundary problem which is not a round cylinder. Then there holds
\begin{equation}
\label{CCC}
\frac{4e^{4\bar \xi_1}}{(\max (G_{\xi_2})^2)}\le\frac{e^{ 2({\bar \xi_1}-{\underline{\xi}_1)}}   {\mathcal L} [\alpha]}{\oint_\alpha (1+\kappa \xi_2)^2\:ds}\le \frac{\pi^2}{h^2}\:.
\end{equation}
Here ${\bar \xi}_1$ and ${\underline \xi}_1$ are, respectively, the maximum and minimum of $\xi_1$ on the generating curve.
\end{thm}

For the case of round cylinders, there are two distinct ways in which stability may fail. Using $\psi_1=\sin(2\pi z/h)$ in (\ref{svc}) yields the necessary condition
\begin{equation}
\label{c1}
\frac{4\pi^2}{h^2}\ge aR+1/R^2\:
\end{equation}
for stability.While using $\psi_2:=\cos(\pi z/h)$ yields the necessary condition for stability
\begin{equation}
\label{c2}
\frac{\pi^2}{h^2}\ge aR\:.
\end{equation}

If $R^3\ge 1/(3a)$, then as $h$ increases (\ref{c2}) will fail before (\ref{c1}), i.e. stability is lost through a non axially symmetric variation. In this case, one can show,  again using 
results of Crandall and Rabinowitz \cite{CR1971} that a bifurcation occurs which breaks the rotational symmetry and there exists equilibrium surfaces bounded by two co-axial circles which
are not rotationally symmetric.

\end{document}